\documentclass[12pt,a4paper]{article}
 \usepackage{amsmath,amstext,amssymb,amscd}
\usepackage[english,russian]{babel}
 \usepackage{graphicx}

 \usepackage[T2A]{fontenc}
 \usepackage[cp1251]{inputenc}

\newcommand{\Xcomment}[1]{}

\newtheorem{theorem}{Теорема}
\newtheorem{lemma}{Лемма}

\newtheorem{prop}{Предложение}

\makeatletter \@addtoreset{equation}{section} \makeatother

\newenvironment{proof}{\noindent{Доказательство.}~}%
{\hfill$\Box$\medskip}

\begin{document}

\title{Cubillages of cyclic zonotopes}

 \author{Vladimir I.~Danilov\thanks{Central Institute of Economics and
Mathematics of the RAS, 47, Nakhimovskii Prospect, 117418 Moscow, Russia;
email: danilov@cemi.rssi.ru.}
 \and
Alexander V.~Karzanov\thanks{Central Institute of Economics and Mathematics of
the RAS, 47, Nakhimovskii Prospect, 117418 Moscow, Russia; email:
akarzanov7@gmail.com. Corresponding author. }
  \and
Gleb A.~Koshevoy\thanks{The Institute for Information Transmission Problems of
the RAS, 19, Bol'shoi Karetnyi per., 127051 Moscow, Russia, and HSE University,
Moscow, Russia; email: koshevoyga@gmail.com. Supported in part by grant RSF
16-11-10075, and by Laboratory of Mirror Symmetry NRU HSE, RF Government grant,
ag. № 14.641.31.0001. }
 }

\date{}
\maketitle

\begin{abstract}
This paper (written in Russian) presents a survey of new and earlier results on
fine zonotopal tilings (briefly, \emph{cubillages}) of cyclic zonotopes. The
combinatorial theory of these objects is of interest in its own right and also
has a connection to higher Bruhat orders, triangulations of cyclic polytopes,
and Tamari--Stasheff posets applied in the study of Kadomtsev--Petviashvily
equations, and etc.
 \medskip

\textbf{Keywords:} higher Bruhat order, Tamari--Stasheff poset, polycategory,
rhombus tiling, separated sets, purity
\end{abstract}

\begin{flushright}
Светлой памяти Владимира \\ Воеводского
посвящается
\end{flushright}

\tableofcontents

\addcontentsline{toc}{section}{Введение}

\section*{Введение}

Замощение фигур другими фигурами -- излюбленная тема геометров. Вспомним
кристаллические решетки и параллелоэдры, упаковки шаров,  мозаики Пенроуза и
связанные с этим ромбические тайлинги. Одно из      направлений, интенсивно
развивающееся в последние 30 лет, касается  триангуляций т.н. циклических
многогранников (см.  \cite{R} и  обзор Рэмбо и Райнера \cite{RR} в сборнике
\cite{ATL}). Здесь мы хотим      обсудить во многом параллельную теорию
кубильяжей циклических зонотопов, то есть разбиений зонотопов на
(комбинаторные) кубы.

Как это ни покажется странным, импульсы к рождению этой теории были не только
геометрическими. {Хотя зонотопальные (и кубические) разбиения издавна
привлекали внимание геометров (см., например, работу \cite{She}, или
исторический очерк в \cite{Zbook}), переворот вызвали две работы, скорее
алгебраические или комбинаторные.}  Одна  -- работа Леклерка и Зелевинского
\cite{LZ} 1998 года -- была связана с вопросом  квазикоммутирования квантовых
миноров и привела к изучению  ромбических тайлингов зоногонов; обзор некоторых
достижений на этом пути см. в статье \cite{UMN}. Другая  -- статья Манина и
Шехтмана \cite{MSch-1} 1986 года (тоже мотивированная квантовой тематикой)  --
была посвящена обобщению слабого порядка Брюа на      симметрической группе
$S_n$; полученные упорядоченные множества $B(n,d)$ были названы высшими
порядками Брюа. Несколько позже Воеводский и Капранов \cite{VK} дали
интерпретацию этих порядков в терминах кубильяжей циклических зонотопов,
частным (двумерным) случаем которых являются упомянутые выше ромбические
тайлинги. Значительные продвижения в этом направлении были получены Циглером
\cite{Z}, Галашиным \cite{Ga}, и Галашинм и Постниковым \cite{GaP} (см. также
\cite{Ba}).

Обо всем этом (включая наши собственные результаты) и предполагается рассказать
в настоящей работе. Мы условно делим изложение на две тесно связанные части:
геометрическую и комбинаторную (или теоретико-множественную). В  первой
рассказывается про зонотопы, кубические тайлинги  (кубильяжи) зонотопов,
вводятся разные полезные объекты в кубильяже: перегородки, тоннели, и т.п.
Здесь же вводится основной рабочий инструмент -- редукции и экспансии,
незаменимые для проведения индуктивных рассуждений.

После введения этих общих понятий мы ограничиваем свое внимание исключительно
случаем циклических (или -- более правильно -- вполне положительных) зонотопов
$Z(n,d)$. Именнно  кубильяжи таких зонотопов связаны с высшими порядками Брюа
Манина-Шехтмана. При работе внутри фиксированного кубильяжа важную роль играет
структура естественного порядка $\preceq$ на множестве его кубов. Когда же мы
переходим к сравнению различных кубильяжей, тут на первый план выходит понятие
флипа, {некоторой локальной перестройки кубильяжа}. Последнее проще всего
объяснить на примере простейшего (после куба) зонотопа $Z(d+1,d)$, т. н.
капсида. У него есть всего два кубильяжа, и замена одного кубильяжа на другой и
называется флипом. В общем случае флип -- это замена `капсидного' фрагмента
кубильяжа другим (`флипованным') фрагментом. Главный факт о  кубильяжах состоит
в том, что с помощью флипов можно перейти от любого кубильяжа зонотопа $Z(n,d)$
к любому другому.

Другое важное понятие -- понятие мембран, гиперповерхностей специального вида
внутри зоногона; мы показываем, что любую  мембрану можно вписать в кубильяж.

Комбинаторная часть статьи связывает кубильяжи с системами подмножеств
множества $[n]=\{1,2,...,n\}$, индексирующего вектора, порождающие циклический
зонотоп $Z(n,d)$.  Имеется два способа сделать это. Первый рассматривает
кубильяжи с точки зрения их \emph{спектров}. Каждый кубильяж $\mathcal Q $
зонотопа $Z(n,d)$ ($d$ --  это размерность зонотопа, тогда как $n$ -- число
`направляющих' векторов) задает систему $Sp(\mathcal Q )$ подмножеств  в
$[n]=\{1,2,...,n\}$, то есть подмножество в $2^{[n]}$. Сет-система $Sp(\mathcal
Q )$ однозначно определяет кубильяж $\mathcal Q $. Этот результат  поднимает
вопрос: какие сет-системы соответствуют кубильяжам, потому что  ответ на него
позволяет заменить геометрию комбинаторикой. Найденный Галашиным и Постниковым
\cite{GaP} ответ состоит в том, что  такие сет-системы обладают т.н. свойством
($d-1$)-\emph{разделенности} и при этом они максимальны по размеру (равному ${n
\choose \le d}$). Если максимальность по размеру можно заменить максимальностью
по включению, то говорят о \emph{чистоте} соответствующего отношения.
Фундаментальное  исследование вопросов чистоты было проведено в \cite{GaP}.
Однако чтение этой  статьи затруднено тем, что авторы работали с произвольными
(не обязательно циклическими) зонотопами, и даже с ориентированными матроидами,
и потому использовали усложненную матроидную терминологию и технику. Мы же
сознательно ограничиваемся циклическим  случаем, что делает (на наш взгляд)
изложение более простым и доступным. В случаях $d=2$ и 3 свойство чистоты было
доказано в \cite{LZ} и \cite{Ga}, соответственно. В случае $d \ge 4$ ответ
отрицателен; для этого мы подробно обсуждаем ситуацию в случае $n=d+2$, а также
переносим отсутствия чистоты на большие $n$ и $d$.

Другой переход от геометрии к комбинаторике подсказывает отображение типов. С
каждым кубом кубильяжа можно связать его \emph{тип},  некоторое подмножество
размера $d$ в множестве $[n]$. И это задает  отображение (биективное, как
несложно показать) кубильяжа в множество ${[n] \choose d}$ подмножеств размера
$d$ в $[n]$.      Пользуясь этой биекцией, можно перенести естественный порядок
$\preceq _\mathcal Q $  с $\mathcal Q $ на ${[n] \choose d}$. Оказывается, что
перенесенный порядок тоже однозначно определяет кубильяж $\mathcal Q $. Тем
самым можно вместо кубильяжей работать с классом т.н. \emph{допустимых}
порядков на `дискретном' грассманиане $Gr([n],d)={[n] \choose d}$. Так мы
возвращаемся к истоку теории -- определению высших порядков Брюа как допустимых
порядков на грассманиане ${[n] \choose d}$. На  геометрическом языке высший
порядок Брюа $B(n,d)$ -- это порядок на  множестве ${\bf Q}(n,d)$ кубильяжей
зонотопа $Z(n,d)$, или, эквивалентно, на  множестве ${\bf M}(n,d)$ мембран в
зонотопе $Z(n,d+1)$.

Кубильяжи имеют многочисленные связи с другими объектами математического мира.
Например, известна их связь с уравнениями Кадомцева-Петвиашвили, идущая пока в
основном через триангуляции циклических политопов. Уже отмечалось, что
происхождение теории было мотивировано уравнениями Замолодчикова. Кубильяжи
дают также естественный пример поликатегорий. Кубильяжи и триангуляции связаны
с представлениями колчанов {и конечномерных алгебр, см., например, работу}
\cite{OT}. О некоторых таких `внешних' связях кубильяжей мы рассказывыаем
кратко в Дополнениях. Здесь мы уже, как правило, не приводим точные
определения и результаты, указывая лишь общие контуры связи. Таких Дополнений
четыре. Это связь с поликатегориями, связь с триангуляциями (и через них с
уравнениями К-П), рассказ о слабых мембранах, обобщающих понятие мембран. Туда
же мы отнесли и доказательство теоремы об ацикличности из основного текста.

\

\

\textsc{Часть первая, геометрическая}

\

Здесь излагаются основные геометрические понятия и факты про  кубильяжи
зонотопов. Комбинаторый (теоретико-множественный) взгляд  будет дан во второй
части. Мы стараемся иллюстрировать все на 2-мерных кубильяжах (ромбических
тайлингах), иногда рисовать 3-мерные картины.

            \section{Зонотопы} \label{sec:intr}

Естественно начать с напоминаний про зонотопы\footnote{Имеется большая
литература про зонотопы, включая книги \cite{BVSWZ, Zbook}.}  и введения
терминов. Кубильяж -- это, грубо говоря, правильное заполнение (разбиение)
выпуклой фигуры `кубами', а точнее говоря -- параллелоэдрами.

Что же такое зонотоп? Фиксируем  вещественное векторное пространство $V$
(размерности $d>0$) и некоторый конечный набор ${\bf V}=(v_1,...,v_n)$ из $n$
векторов в нем. {\em Зонотопом} $Z=Z({\bf  V})$ (порожденным набором  ${\bf
V}$) называется  сумма по Минковскому $n$ отрезков $[0,v_i]$. Иначе говоря, $Z$
состоит  из точек $z$ вида $\sum _i \alpha_i v_i$, где $0\le \alpha _i\le 1$
для любого $i=1,...,n$. Можно также сказать, что зонотоп -- это проекция
`единичного'  $n$-мерного куба, когда базисные вектора пространства ${\mathbb
R}^n$ отправляются в $v_i$. Это выпуклое тело, симметричное относительно  точки
(центра) $\sum v_i/2$.

Зонотопом мы называем и любой сдвиг зонотопа. Заменяя, если  нужно, вектора
$v_i$ на $-v_i$, можно считать, что все они `смотрят в  одну сторону', если
угодно -- `вверх' по отношению к некоторой  `горизонтальной'  гиперплоскости.
Такая замена меняет зонотоп на сдвинутый. Но зато теперь наш зонотоп имеет
`нижнюю' (корневую) вершину 0 и `верхнюю' вершину $v_1+...+v_n$.

Простейший пример зонотопа -- это \emph{куб} (точнее -- параллелоэдр, но мы для
краткости именуем его кубом). Это когда система $v_1,...,v_n$ образует базис
пространства $V$ (в этом случае $n=d$). Далее мы будем заниматься заполнениями
зонотопов кубами.

Довольно легко описать грани зонотопа (которые тоже  зонотопы меньшей
размерности; {см. \cite[гл. 7]{Zbook} о кодировке граней знаковыми векторами}).
Мы еще более облегчим эту задачу. Во первых, мы будем считать, что система {\bf
V} находится в \emph{общем положении}, то есть что $n\ge d$ и что любые $d$
векторов из системы {\bf V} образуют базис пространства $V$. Далее это
предположение всюду подразумевается, в основном по той причине, что мы будем
заниматься кубильяжами $Z$. Во вторых, мы ограничимся описанием только фасет
(граней коразмерности 1) зонотопа $Z$. Понятно, что фасеты такого `общего'
зонотопа будут кубами размерности $d-1$.

Чтобы задать фасету, мы фиксируем произвольные $d-1$ вектор $w_1,...,w_{d-1}$
(подсистему {\bf W}) в нашей системе {\bf V}. Пусть $H_{\bf W}$ --
гиперплоскость в $V$, натянутая на {\bf W}. Она делит оставшиеся точки-вектора
{\bf V}-{\bf W} на две части, скажем, ${\bf V}_+$ и ${\bf V}_-$. Это  дает две
(противоположные друг другу) фасеты зонотопа $Z$. Одна образована
зонотопом-кубом $Z({\bf W})$, укорененным в точке $\sum  {\bf      V}_+$, а
другая --  тем же зонотопом, но укорененным в точке $\sum  {\bf V}_-$. И так
получаются  все фасеты.  Отсюда видно, что

\

\noindent (1.1) \ \ \ \ \emph{`общий' зонотоп $Z$ имеет  $2{n \choose d-1}$ фасет.}

\

Тут неявно предполагалось, что $d>1$; случай одномерного  зонотопа несколько
выпадает из общей картины по причине его тривиальности; это просто отрезок $[0,
v_1+...+v_n]$.

Стоит еще сказать о вершинах зонотопа и об их числе. Наверняка это хорошо
известно, но нам не удалось найти хорошую ссылку. Поэтому приведем формулу, а
позже скажем о доказательстве.

\

\noindent  (1.2)  \emph{Число вершин $V(n,d)$ зонотопа $Z$ равно
$2{n-1 \choose \le (d-1)}=2({n-1 \choose d-1}+...+{n-1 \choose 0})$.}

\

Другая полезная вещь -- это взгляд на зонотоп  вдоль некоторого направления.
Пусть направление задается некоторым  ненулевым вектором $w$ в $V$ и
проектирует (отображением $\pi _w$) пространство $V$ (и находящийся в нем
зонотоп $Z$) на (фактор)пространство $V'=V/{\mathbb R}w$. Снова тут удобно
считать, что этот  вектор $w$ находтся в общем положении относительно системы
{\bf V}. В этом  случае каждая фасета $F$ зонотопа $Z$ инъективно проектируется
в  пространство $V'$. А весь зонотоп $Z$ проектируется на зонотоп      $Z'=\pi
_w(Z)=Z(\pi _w({\bf V}))$. При этом граница $\partial Z'$ (имеющая размерность
$d-2$) является биективным образом некоторого (тоже ($d-2$)-мерного)
подкомплекса границы $Z$, который можно назвать \emph{ободом} (относительно
этой проекции $\pi _w$). Обод делит (нестрого) границу $Z$ на две `полусферы' -
`верхнюю' (мы представляем проекцию $\pi_w $ как вертикальную, снизу вверх) и
`нижнюю'.

Здесь удобно ввести терминологию видимых и невидимых фасет. Снова берем `общий'
вектор $w$ (вдоль которого мы смотрим на  зонотоп $Z$). Пусть $F$ - фасета $Z$.
Скажем, что эта фасета $F$ {\em видимая} (в направлении $w$), если для
(относительно внутренней) точки $p$ этой фасеты точка $p-\varepsilon w$ не
принадлежит $Z$ (для малого или  любого  $\varepsilon >0$); тогда для малого
$\varepsilon >0$ точка $p+\varepsilon w$  принадлежит $Z$. Иначе говоря, точка
$p$ первая в зонотопе при движении (из $-\infty $ в $+\infty $) по прямой
$p+{\mathbb R}w$.

В этой теминологии видимые фасеты покрывают нижнюю  полусферу зонотопа, а
невидимые - верхнюю. Проекция верхней  полусферы дает один кубильяж зонотопа $Z
'$, а нижнего -- другой, симметричный верхнему. Мы еще вернемся к этой теме в
следующем разделе.

Можно проектировать не только вдоль `общих' направлений, но и, скажем, вдоль
направления вектора $v_i$, обозначая эту проекцию как $\pi _i$. Проекция при
$\pi _i$ зонотопа $Z$ снова будет зонотопом $Z'=Z/ v_i$, порожденным (в
$V'=V/\mathbb{R}v_i$) образами всех векторов $v_1,...,v_n$, кроме $v_i$. Однако
теперь обод будет не ($d-2$)-мерный, но ($d-1$)-мерный {(и называется он
\emph{зоной} или \emph{поясом} зонотопа)}. Отметим также, что сам $Z$
получается как сумма зонотопа $Z({\bf V}-\{v_i\})$ и отрезка  $Z(v_i)=[0,v_i]$.

После введения этих понятий можно обратиться к доказательству формулы (1.2).
Зонотоп $Z=Z(\mathbf{V})$ представляется как сумма зонотопа
$\widetilde{Z}=Z(\mathbf{V}-\{v_n\})$ и отрезка $[0,v_n]$. При этом его вершины
делятся на видимые и невидимые. Но видимые вершины (относительно направления
$v_n$) зонотопа $\widetilde{Z}$ те же, что видимые вершины зонотопа $Z$;
аналогично для невидимых (точнее, видимых с противоположного направления).
Единственная разница состоит в том, что в пером случае эти два множества не
пересекаются, тогда как во втором -- пересекаются, причем в точности по
множеству вершин обода (зонотопа $\widetilde{Z}$ в направлении $v_n$.). Отсюда
мы получаем соотношение
$$
V(n,d)=V(n-1,d)+V(n-1,d-1).
$$
Учитывая правило Паскаля, остается только убедиться в справедливости формулы
(1.2) в случаях $d=1$ и $d=n$. В первом случае правая часть равна $2{n-1
\choose 0}=2$, что согласуется с тем, что отрезок имеет две вершины. Во втором
случае правая часть равна $2{d-1 \choose \le d-1}=2(2^{d-1})=2^d$, что
согласуется с тем, что $d$-мерный куб имеет $2^d$ вершин.

           \section{Кубильяжи }\label{sec:cubil}

Кубильяж (или кубический, гиперромбический, параллелоэдральный тайлинг) -- это правильное замощение
(разбиение, paving, tiling) зонотопа (или другого многогранника) кубами. Точнее:

{\em Кубильяж} ${\mathcal Q}$ зонотопа $Z$ -- это множество $d$-мерных кубов
(см. выше) $Q_1,...,Q_N$ (называемых плитками, тайлами или просто
\emph{кубами}), покрывающих $Z$, причем выполнены два условия:

           1) два куба могут пересекаться только по общей грани;

           2) фасеты зонотопа $Z$ являются фасетами некоторых кубов из ${\mathcal Q}$.

\noindent Точнее, так определяется кубильяж для зонотопов размерности $d>1$.
Кубильяж одномерного зонотопа по определению состоит из $n$ отрезков,
конгруэнтных отрезкам $[0,v_i]$, $i=1,...,n$.

{\em Гранью} кубильяжа считается грань некоторого куба из ${\mathcal Q}$.\medskip

\begin{figure}[h]
\begin{center}
\includegraphics[scale=.35]{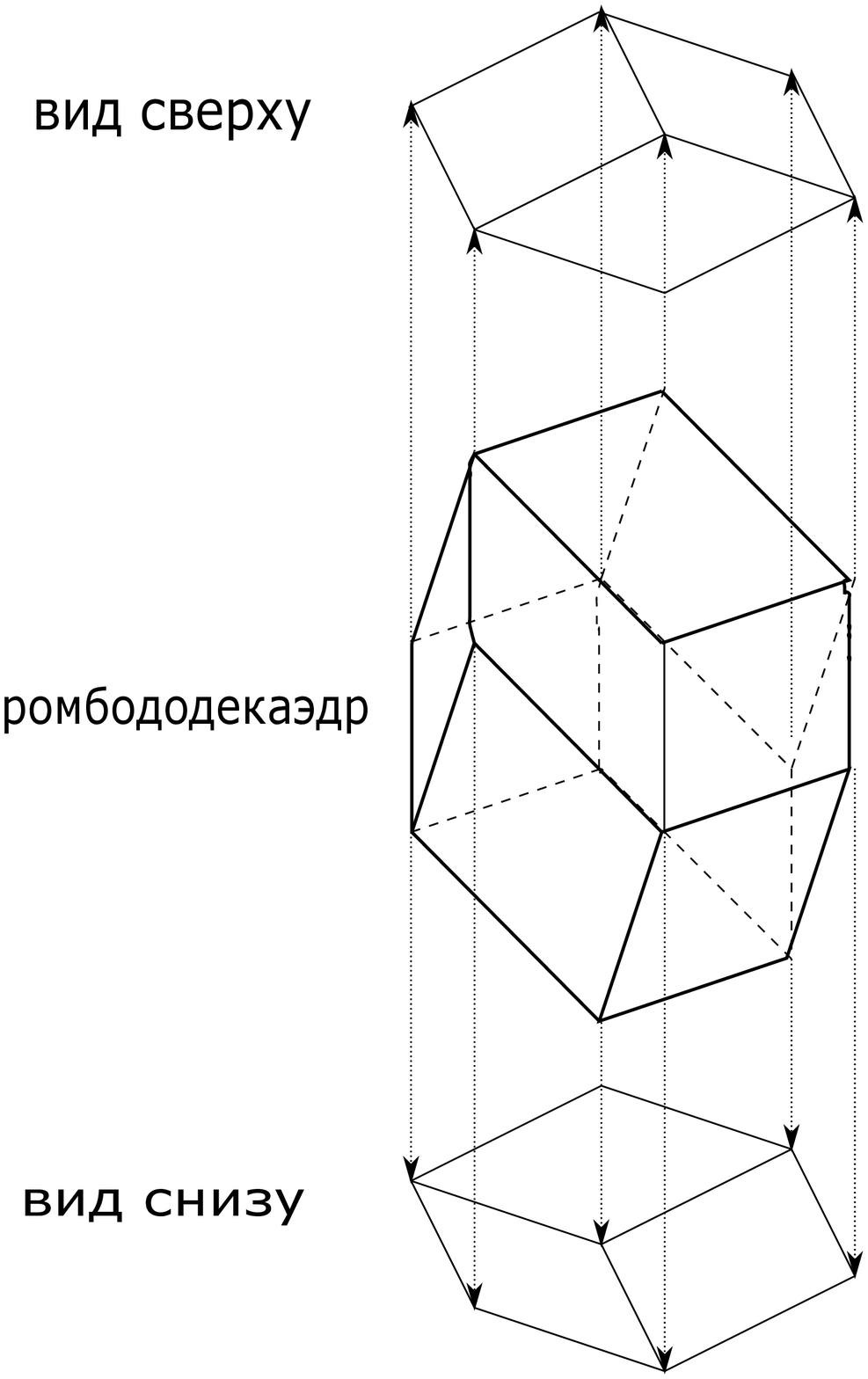}
\end{center}
 \caption{Вид сверху и снизу на ромбододекаэдр} \label{fig:1}
 \end{figure}

\textbf{Пример.} Пусть $Z$ -- зонотоп $Z(v_1,...,v_n)$, и мы проектируем  его
вдоль, скажем, направления $v_n$ (обозначая проекцию как $\pi  $, см. рис.
\ref{fig:1}). Тогда, как объяснялось в предыдущем разделе, проекция фасет
видимой  части границы $Z$ дает кубильяж зонотопа $Z'=Z(\pi (v_1),...,\pi
(v_{n-1}))$.\medskip

Часто на эту конструкцию смотрят в противоположном направлении,  говоря, что
система векторов ${\bf V}= \{v_1,...,v_n\}$ -- это `одноэлементный лифтинг'
системы ${\bf V}'=\{\pi  (v_1),...,\pi (v_{n-1})\}$ (с помощью      вектора
$v_n$). В такой терминологии одноэлементный лифтинг системы ${\bf V}'$ дает
кубильяж зонотопа $Z'=Z({\bf V}')$. Знаменитая теорема Бонэ-Дресса (см.
\cite{BVSWZ} или \cite{Zbook}) в каком-то смысле утверждает обратное. Мы  не
приводим точную формулировку, поскольку это лежит в стороне от наших
целей.\medskip

Полезным для понимания является следующий простой факт, который часто просто включают в определение кубильяжа.

\begin{lemma} \label{lm:1}
Пусть $Q$ -- куб некоторого кубильяжа ${\mathcal Q}$ зототопа $Z=Z({\bf V})$.
Тогда  любое ребро этого куба конгруэнтно некотоому отрезку $[0,v_i]$.
 \end{lemma}

\begin{proof}
Можно считать, что $d>1$. Пусть $E$ - некоторое ребро куба $Q$. Если оно не
лежит на границе зонотопа, вокруг него есть много кубов, ребра которых
конгруэтны $E$. Переходя от одного такого куба к другому, мы рано или поздно
выйдем на границу зонотопа. А тогда все следует из свойства 2) определения
кубильяжа.
\end{proof}

В частности, любой куб $Q$ кубильяжа конгруэнтен кубу-зонотопу $Z({\bf W})$,
где ${\bf W}$ -- $d$-элементное подмножество в $\mathbf{V}$. Множество индексов
(или \emph{цветов}; так мы понимаем элементы индексирующего множества
$[n]=\{1,...,n\}$) этого  подмножества ${\bf W}$ называется {\em типом} этого
куба. Иными словами, тип куба  $Q$ -- это подмножество $\{i_1,...,i_d\}$ в
$[n]$, такое что куб $Q$  есть сдвиг куба-зонотопа $Z(v_{i_1},...,v_{i_d})$.

Аналогично  понимается тип грани кубильяжа. В частности, тип ребра кубильяжа
-- одноэлементное подмножество $\{i\}$ в $[n]$.  Мы ориентируем каждое  ребро
кубильяжа в направлении вектора $v_i$. Направленный путь из нижней вершины 0 в
верхнюю вершину $v([n])$,  идущий по ребрам кубильяжа, мы называем {\em
змейкой}. Как мы увидим  далее, в змейке ребро цвета $i$ встречается ровно один
раз. Если  $c(i)$ -- цвет $i$-го ребра-стрелки некоторой змейки, то мы получаем
биекцию $[n]$ в $[n]$, то есть перестановку множества $[n]$. Эту перестановку
можно понимать также как линейный порядок на $[n]$, в котором минимальным
элементом яаляется цвет первого ребра змейки.

Сопоставление кубу $Q$ кубильяжа ${\mathcal Q}$ его типа  задает отображение
$\tau =\tau _{\mathcal Q}: {\mathcal Q} \to {{\bf V} \choose d}$ в множество
$d$-элементных подмножеств {\bf V}.

\begin{prop} \label{pr:1} Отображение $\tau $ является биекцией. В частности, число кубов в любом кубильяже равно ${n \choose d}$.\end{prop}

Мы докажем это утверждение в следующем разделе. А пока приведем простое\medskip

\textbf{Следствие.} \emph{Число вершин любого кубильяжа равно ${n \choose \le
d}$.}\medskip

В самом деле, выберем некоторое общее направление проектирования (`смотрения')
$w$. {Из формулы (1.2) видно, что для каждого куба $Q$ кубильяжа все его
вершины, кроме двух, лежат на ободе этого куба. И есть одна `внутренняя'
вершина на видимой половине границы  $Q$ и, аналогично, одна `внутренняя'
вершина на невидимой половине. Обозначим эту последнюю, самую далекую он нас
вершину куба как }$h(Q)$. Легко понять, что отображение $h$ (из $\mathcal{Q}$ в
множество вершин кубильяжа ${\mathcal Q}$) инъективно. А образ заметает все
вершины, кроме вершин на видимой части $\partial _-(Z)$ границы $Z$. При
проекции $\pi =\pi _w$ фасеты $\partial_-(Z)$ дают кубильяж зонотопа $Z'=Z(\pi
({\bf V}), d-1)$, у которого, по предположению индукции, ${n \choose d-1}+...+
{n \choose 1}+{n \choose 0}$  вершин. $\Box$

           \section{Перегородки}\label{sec:partit}

Пусть ${\mathcal Q}$ -- некоторый кубильяж зонотопа $Z=Z({\bf V})$, где ${\bf
V}=\{v_1,...,v_n\}$. В дальнейшем мы рассмотрим несколько интересных
подмножеств в ${\mathcal Q}$ -- перегородки, тоннели, гирлянды, стэки, капсиды
... Начнем с перегородок. \medskip

\textbf{Определение. } \emph{Перегородкой} цвета $i\in [n]$ в кубильяже
${\mathcal Q}$ называется множество ${\mathcal P}_i$ кубов кубильяжа ${\mathcal
Q}$, которые содержат цвет $i$ в своем типе. То есть
                                                         $$
                              {\mathcal P}_i=\{Q\in {\mathcal Q},  \    i\in \tau (Q)\}.
                                                           $$
(В некоторых работах перегородки называются также \emph{сечениями де Брюйна}).
\emph{Телом перегородки} ${\mathcal P}_i$ называется объединение кубов из
${\mathcal P}_i$; это замкнутое подмножество в зонотопе $Z=Z({\bf V})$, которое
мы обозначаем как $|{\mathcal P}_i|$ (не путать с числом элементов множества
${\mathcal P}_i$, которое равно ${n-1 \choose d-1}$). Тело перегородки
${\mathcal P}_i$ выходит на границу зонотопа по зоне цвета $i$, см. раздел 1.
\medskip

{\bf Основная теорема о перегородках.} {\em Тело перегородки $|{\mathcal P}_i|$
устроено  как произведение отрезка $[0,v_i]$ на некоторый диск,
трансверсальный направлению $v_i$ и биективно проектирующийся на  зонотоп
$Z'=\pi _i(Z)$ (где $\pi _i$ -- проекция вдоль вектора $v_i$, см.
выше).}\medskip

Здесь для наглядности удобно считать вектор $v_i$ направленным вертикально
вверх. В каждом кубе $P$ перегородки  ${\mathcal P}={\mathcal P}_i$ проведем
`среднее сечение' $MP$ (среднее по вертикальному направлению). Все такие
локальные сечения склеиваются в  ($d-1$)-мерный кубический комплекс $M{\mathcal
P}$. Так как каждый  куб $P$ представляется как сумма по Минковскому сечения
$MP$ и трансверсального к нему вертикального отрезка $[-v_i/2, v_i/2]$, то и
вся перегородка ${\mathcal P}$ есть сумма (по Минковскому)      $M{\mathcal P}$
и отрезка $[-v_i/2, v_i/2]$. Более точно, естественное отображение
$$
M{\mathcal P} \times [-v_i/2, v_i/2] \to M{\mathcal P}+[-v_i/2, v_i/2]=|{\mathcal P}|
$$
является биекцией (и гомеоморфизмом).  Отсюда легко понять, что при
(вертикальной) проекции сечение $M{\mathcal P}$ локально гомеоморфно
проектируется на $Z'$. Так как зонотоп $Z'$ стягиваем и, следовательно,
односвязен, среднее сечение представляетися как сумма нескольких копий $Z'$. На
самом деле копия всего одна, что видно из того, что край этого комплекса
$M{\mathcal P}$ образует `среднее сечение' соответствующего обода (зоны цвета
$i$) зонотопа $Z$. $\Box$\medskip

\textbf{Замечание.} {Перегородки (а еще лучше -- их средние сечения) напоминают
по своим свойствам гиперплоскости. Как и последние, они делят зонотоп на две
области (до и после); кроме того, любые $d$ перегородок пересекаются по
единственному кубу (а средние сечения -- по единственной точке). Так что
система всех перегородок выступает как некая замена арранжмента
гиперплоскостей. Недаром вся эта теория у Манина и Шехтмана начиналась с
аранжментов гиперплоскостей. }\medskip

Проекция комплекса $M{\mathcal P}_i$ (или всей  перегородки ${\mathcal P}_i$)
дает  кубильяж $\pi _i({\mathcal P}_i)$ зонотопа $Z'=\pi _i(Z)=Z/v_i$. Этот
кубильяж (размерности  $d-1$) зонотопа $Z'$ можно обозначить как ${\mathcal
Q}/v_i$. Вслед за  Атанасиадисом \cite{Athan} мы называем кубильяж ${\mathcal
Q}/v_i$ {\em сжатием (контрактацией)} кубильяжа ${\mathcal Q}$ в направлении
$v_i$, или сжатием по цвету $i$. Отметим, что при сжатии размерность кубильяжа
уменьшается на единицу, как и число цветов.

Первое следствие основной теоремы про перегородки -- это уже упоминавшееся
утверждение про змейки. Напомним, что \emph{змейка} -- это направленный путь по
ребрам-стрелкам кубильяжа из нижней (корневой) вершины зонотопа в верхнюю. Так
вот, {\em цвета ребер змейки биективны множеству $[n]$}. То есть каждый цвет
$i\in [n]$ встречается и при том один раз как цвет ребра змейки. В самом деле,
перегородка ${\mathcal P}_i$  цвета $i$ делит зонотоп на три части -- ниже
${\mathcal P}_i$, само ${\mathcal P}_i$, и выше      ${\mathcal P}_i$. Наша
змейка (при движении от корня к верхушке) обязательно пересекает перегородку,
так что обязательно содержит стрелку цвета  $i$. Причем пересекает ровно один
раз, потому что после пересечения перегородки змейка попадает в верхнюю часть
зонотопа и уже не может ее покинуть (в перегородке все стрелки цвета $i$
смотрят вверх).

Это следствие позволяет ввести важное понятие \emph{спектра} (набора цветов)
вершины кубильяжа, которое будет играть центральную роль в комбинаторной части.
А именно, рассмотрим для вершины $v$ кубильяжа путь из 0 в $v$, идущий вверх по
стрелкам-ребрам кубильяжа. Очевидно, что такой путь существует, быть может, не
один. Этот путь -- начальный кусок некоторой змейки, поэтому цвета его ребер
образуют подмножество в $[n]$. Это подмножество не зависит от выбора пути; мы
обозначаем его $sp(v)$ и называем {\em спектром} вершины $v$. Независимость от
выбора пути видна из другого  описания $sp(v)$: цвет $i$ входит в $sp(v)$ тогда
и только тогда, когда перегородка ${\mathcal P}_i$ проходит ниже вершины $v$
(то есть отделяет $v$ от  корневой вершины 0). Заметим также, что $sp(v)$ не
зависит и от кубильяжа ${\mathcal Q}$; в самом деле,
                              $$
                              v=\sum _{i\in sp(v)} v_i.
                                $$

Наряду с перегородками можно рассмотреть в каком-то смысле дуальные объекты,
называемые {\em тоннелями} в кубильяже ${\mathcal Q}$. Зафиксируем подмножество
$D$ в $[n]$ размера $d-1$ и соберем вместе множество ${\mathcal T}_D$ кубов,
типы которых содержат $D$ (`тоннель типа $D$'). У каждого куба $Q$ из
${\mathcal T}_D$ есть две фасеты типа $D$. Пусть $F$ -- такая фасета; тогда она
либо фасета самого зонотопа $Z$, либо по этой фасете наш куб $Q$ смежен другому
кубу $Q'$, тоже из ${\mathcal T}_D$. Повторяя это с кубом $Q'$, мы получаем
соседний куб $Q''$ и так далее, пока не дойдем до фасеты зонотопа $Z$. (Эта
конструкция уже встречалась при доказательстве Леммы 1.) Таким образом, тоннель
представляет набор нескольких `толстых путей' (или циклов) внутри кубильяжа. На
самом деле легко понять, что тоннель -- это связная цепочка кубов, идущих от
одной (видимой) части  границы $Z$ к другой, невидимой. Потому что каждый
тоннель типа $D$ пересекает каждую перегородку ${\mathcal P}_i$ (где $i\notin
D$) ровно по одному кубу (типа $Di$). И состоит ровно из $n-d+1$ куба.\medskip

Теперь обратимся к доказательству Предложения 1. Оно проводится индукцией по
$d$; при $d=1$ утверждение тривиально. Пусть теперь $d>1$. Возьмем некоторую
$d$-шку $D\subset [n]$ и покажем, что существует, и при том ровно один, куб $Q$
типа $D$ в нашем кубильяже ${\mathcal Q}$. Для этого выберем один из цветов
(скажем, $i$) в $D$  и рассмотрим перегородку ${\mathcal P}_i$  цвета $i$.
Пусть $\pi _i:Z \to Z'=\pi _i(Z)$ -- проекция вдоль вектора $v_i$. Как было
сказано выше, проекция перегородки ${\mathcal P}_i$ дает кубильяж
$\pi_i({\mathcal P}_i)$ зонотопа $Z'$. По индукции у него есть единственный куб
$Q'$ (размерности $d-1$) типа $D-i$. Куб $Q$ из перегородки ${\mathcal P}_i$,
проектирующийся на $Q'$, имеет тип $(D-i)\cup  \{i\}=D$. Это доказываент
существование и единственность куба типа $D$. $\Box$\medskip

Примерно так Шеппард \cite{She} получал формулу для числа кубов кубильяжа
зонотопа (см. Предложение 1).

           \section{Редукции и экспансии}

Продолжим получать следствия структуры перегородок. Снова временно (и для
наглядности) мы считаем, что вектор $v_i$ (произвольного цвета) идет
вертикально вверх, так что мы смотрим на зонотоп $Z$ снизу вверх. Перегородка
${\mathcal P}={\mathcal P}_i$ делит $Z$ на три части: собственно перегородку
${\mathcal P}$, то, что нестрого ниже ${\mathcal P}$ (включая нижнюю границу
${\mathcal P}$), что мы обозначим $Z_-$, и верхнюю половину $Z_+$.

{\em Редукцией} кубильяжа ${\mathcal Q}$ называется кубильяж  ${\mathcal
Q}_{-i}$ зонотопа  $Z({\bf V}-\{v_i\}$), полученный следующим образом: нижнюю
часть $Z_-$ мы  оставляем на месте (вместе со всеми ее кубами), а верхнюю часть
$Z_+$ сдвигаем вниз на вектор $-v_i$ (тоже вместе со всеми кубами). При  этом
перегородка ${\mathcal P}$ исчезает, а лучше сказать -- сжимается на свою
нижнюю границу, превращаясь в то, что позже мы назовем мембраной. Эту мембрану
${\mathcal M}=Z_-\cap |{\mathcal P}_i|$ мы называем {\em швом} (или рубцом),
оставшимся после операции редукции. См. рис. \ref{fig:2}.

\begin{figure}[h]
\begin{center}
\includegraphics[scale=.6]{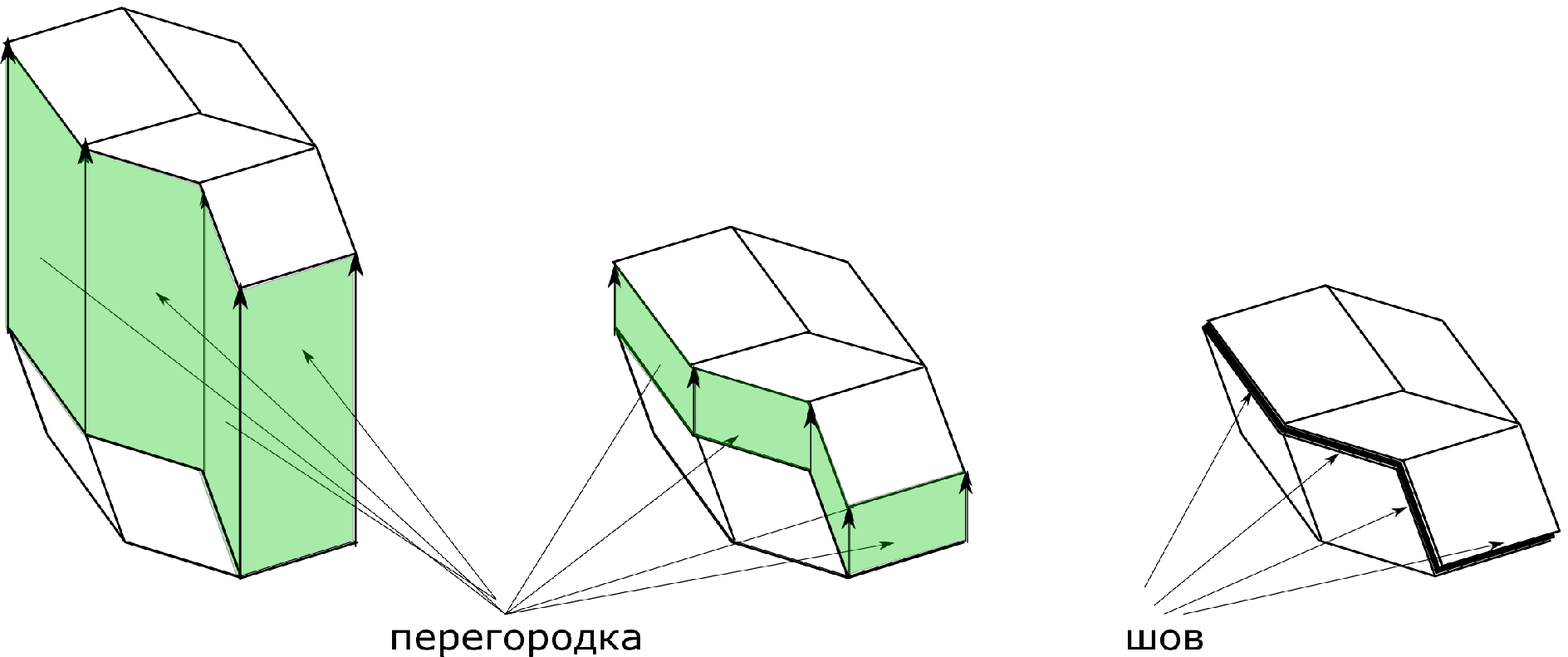}
\end{center}
\caption{\small Перегородка и ее постепенная редукция}
\label{fig:2}
\end{figure}

Эта операция редукции обратима: если мы `раздвинем' шов-мембрану ${\mathcal M}$
с помощью отрезка $[0,v_i]$ (отодвинув часть редуцированного кубильяжа,
расположенную выше шва, на вектор $v_i$), мы вернемся к исходному кубильяжу
${\mathcal Q}$. Эту обратную операцию мы   называем {\em экспансией}.

Вообще, если у нас есть кубильяж $\widetilde{{\mathcal Q}}$ зонотопа
$\widetilde{Z}=Z({\bf V}-\{v_i\})$ и в нем мембрана ${\mathcal M}$ (то есть
($d-1$)-мерный подкомплекс кубильяжа ${\mathcal Q}$, который биективно
проектируется при `вертикальной проекции'  на
$\widetilde{Z}'=\pi(\widetilde{Z})$), то мы можем `раздвинуть'  мембрану
$\mathcal M$ на вектор $v_i$ и получить некоторый кубильяж ${\mathcal Q}$
зонотопа $Z$. Эту операцию мы называем       \emph{экспансией мембраны}
${\mathcal M}$ в кубильяже $\widetilde{{\mathcal Q}}$ в направлении
$v_i$.\medskip

\textbf{Следствие.} {\emph{Кубильяж зонотопа $Z(\textbf{V})$ определяется
множеством его вершин. }}\medskip

\begin{proof} Пусть  $\mathcal{Q}$  и  $\mathcal{R}$  -- два кубильяжа зонотопа $Z(\textbf{V})$ с одним и тем же множеством вершин. Выберем какой-то цвет, скажем, $n$, и произведем редукцию этого цвета в $\mathcal{Q}$  и    $\mathcal{R}$. В результате мы получим два новых кубильяжа $\mathcal{Q}_{-n}$ и $\mathcal{R}_{-n}$ (зонотопа $Z(\mathbf{V}-\{v_n\})$) и в каждом шов $S$ и $T$. Шов $S$ составлен из `слившихся' вершин $v$ и $v'$, таких что $sp(v')=sp(v)\cup \{n\}$; аналогично для шва $T$. Отсюда видно, что совпадают как эти два шва, так и множества вершин кубильяжей $\mathcal{Q}_{-n}$ и $\mathcal{R}_{-n}$. По индукции мы получаем совпадение самих кубильяжей $\mathcal{Q}_{-n}=\mathcal{R}_{-n}$. А исходные кубильяжи $\mathcal{Q}$  и  $\mathcal{R}$ получаются $n$-экспансией общего шва $S=T$, так что они тоже совпадают.
  \end{proof}

Раз уж речь зашла о вершинах кубильяжа, то возникает вопрос -- \emph{какие
точки (или наборы точек) внутри зонотопа $Z=Z(\textbf{V})$ могут быть вершинами
некоторого кубильяжа этого зонотопа} (`вписываются в кубильяж')? Одно
требование очевидно (см. выше определение спектра): такая точки должна иметь
вид $v(S)=\sum _{i\in S} v_i$ для некоторого подмножества $S\subset [n]$. Точки
такого вида мы называем \emph{целыми}, и это требование всюду далее
предполагается выполненным. Мы еще вернемся к этому вопросу во второй части
работы, а пока приведем один простой результат в этом направлении. А именно, мы
утверждаем, что любая отдельно взятая целая точка вписывается в некоторый
кубильяж. На самом деле верно даже большее: \emph{любой подзонотоп вписывается
в некоторый кубильяж}. Более точно, множество вершин любого подзонотопа $T
\subset  Z$ вписывается в некоторый кубильяж зонотопа $Z$; здесь
предполагается, что $T$  есть сдвиг на целую точку $v(S)$ зонотопа
$Z(\textbf{W})$, где  $\textbf{W} \subset \textbf{V}$, {и $S$ не пересекается с
множеством индексов $\textbf{W}$.}\medskip

 \begin{prop} \label{pr:3}
Любой подзонотоп (и любой его кубильяж) $T$ зонотопа $Z=Z(\textbf{V})$ вписывается в некоторый кубильяж $Z$.
  \end{prop}

 \begin{proof}
Будем рассуждать индукцией по $n$ (или по размеру $\textbf{V-W}$). Пусть $t(T)$
-- `верхняя' вершина подзонотопа $T$. Предположим сначала, что эта вершина
отлична от верхней вершины всего зонотопа $Z$. В этом случаем множество
$sp(t(T))$ отлично от $[n]$; пусть $i$ -- произвольный цвет, не принадлежащий
$sp(t(T))$. Тогда подзонотоп $T$ лежит в зонотопе $Z'=Z(\textbf{V}-\{i\})$, и
по индуктивному предполоржению вписывается в некоторый кубильяж $\mathcal{Q}'$
зонотопа $Z'=Z(\textbf{V}-\{i\})$. Остается сделать экспансию по цвету $i$. См.
рис. \ref{fig:3}.

\begin{figure}[htb]
\begin{center}
\includegraphics[scale=.23]{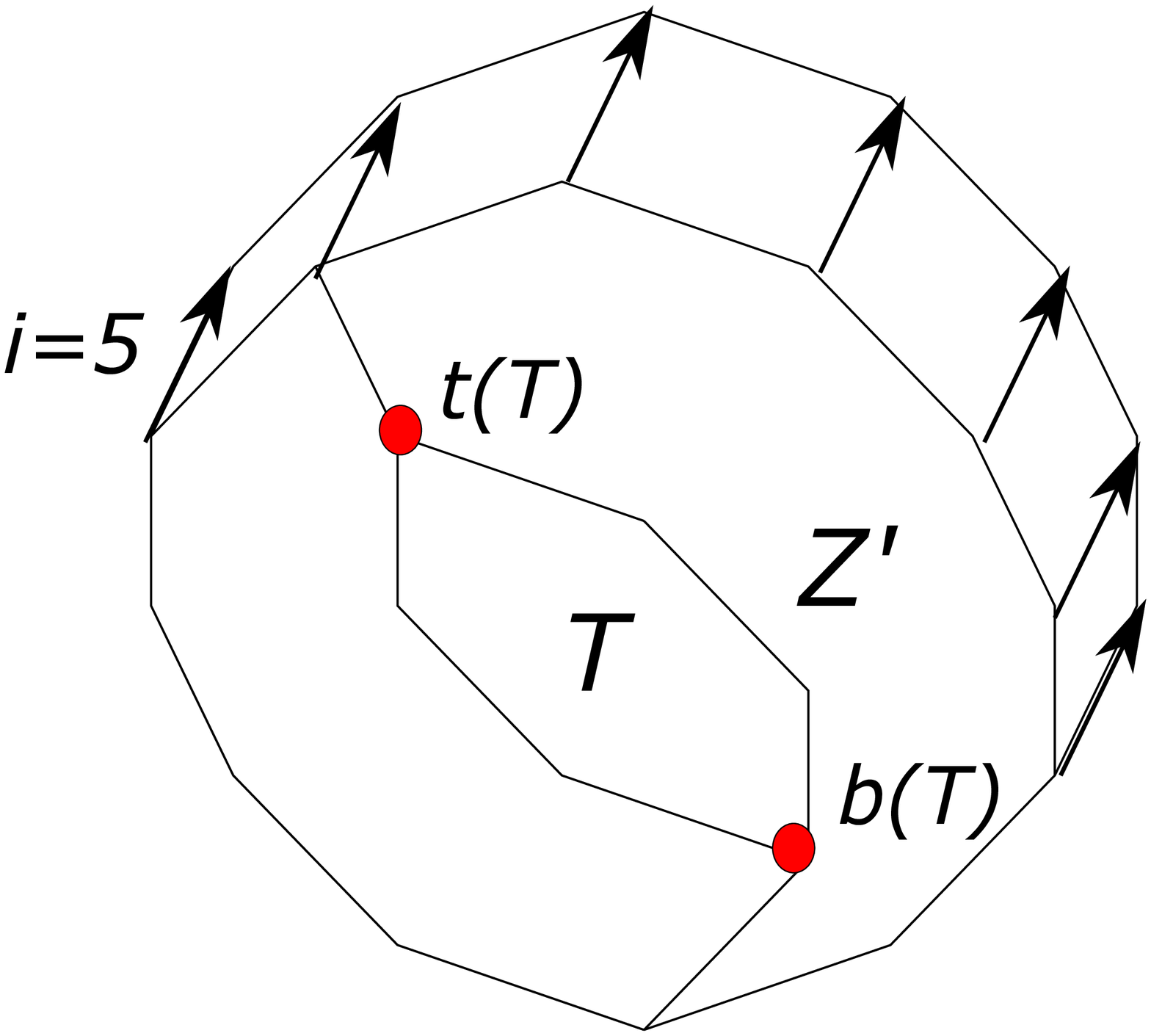}
\end{center}
\caption{подзоногон $T$ в зоногоне $Z(7,2)$}
 \label{fig:3}
  \end{figure}

Таким образом, мы можем считать, что `верхние' вершины у $T$ и $Z$ совпадают.
Рассуждая симметрично, можно считать, что совпадают и `нижние' (корневые)
вершины. Но тогда $T$ и $Z$ совпадают, и утверждение тривиально верно.
\end{proof}

Экспансию мембраны ${\mathcal M}$ можно производить не только в направлении
вектора $v_i$, но и в любом `близком' направлении $v'_i$. Точнее, тут важна не
близость, а только то, чтобы по  отношению ко всем пластинкам (граням
размерности $d-1$) мембраны  ${\mathcal M}$ вектора $v_i$ и $v'_i$ глядели в
одну и ту же сторону. Раздвигая  мембрану ${\mathcal M}$ в направлении $v'_i$,
мы получаем другой  кубильяж другого  зонотопа, но в каком-то смысле похожий,
`подобный' старому. Скажем  об этом чуть подробнее.

Говоря о зонотопе $Z({\bf V})$ и его кубильяжах, мы фиксировали  векторную
конфигурацию {\bf V}. Однако интуитивно ясно, что малые шевеления {\bf V} не
играют роли в строении кубильяжей. Скажем точнее. Пусть у нас есть две векторые
конфигурации  ${\bf V}=\{v_1,...,v_n\}$ и ${\bf W}=\{w_1,...,w_n\}$ (в одном и
том же $d$-мерном пространстве $V$ и с одним и тем же числом векторов $n$).
Скажем, что конфигурации {\bf V} и {\bf W} {\em подобны}, если для любого
подмножества $B\subset [n]$ размера $d$ определители $\det(v_B)$ и $\det(w_B)$
имеют один и тот же знак (или что  соответствующие базисы $v_B$ и $w_B$ имеют
одну и ту же ориентацию).

Если конфигурации {\bf V} и {\bf W} подобны, то для любого  кубильяжа
${\mathcal Q}$ зонотопа $ЕZ({\bf V})$ естественно строится `подобный' кубильяж
${\mathcal Q}'$ зонотопа  $Z({\bf W})$, который комбинаторно устроен так же,
как ${\mathcal  Q}$. А именно, пусть $Q$ - куб кубильяжа ${\mathcal Q}$; у него
есть корневая вершина $v$ и тип $\tau (Q)\subset [n]$. Из начальной  вершины 0
зонотопа      $Z({\bf V})$ в вершину $v$ ведет путь $P$ по стрелкам кубильяжа
$Z({\bf  V})$. Проведем соответствующий путь $P'$ в $Z({\bf W})$ и конец его
$v'$ объявим образом вершины $v$. Ясно, что это построение не зависит  от
выбора пути $P$. Образ куба $Q$ -- это куб того же типа, что и $Q$, растущий из
вершины $v'$. Так мы получаем новое множество ${\mathcal Q}'$ кубов в $Z({\bf
W})$, и нужно лишь убедиться, что эти кубы не налезают друг на  друга. А это
достаточно проверить для соседних кубов. То есть если кубы $Q$ и $R$ кубильяжа
${\mathcal Q}$ соседние (по общей фасете $F$), то их  образы $Q'$ и $R'$
(которые, очевидно, имеют общую фасету $F'$) лежат по  разные стороны от $F'$.

Пусть фасета $F$ (то есть куб размерности $d-1$) имеет тип  $\tau (F)$. Кубы
$Q$ и $R$ получаются добавлением к векторам $v_{\tau (F)}$ каких-то  векторов
$v_i$ и $v_j$. Так как они лежат по разные стороны от $F$, то      ориентации
базисов $v_{\tau (F)i}$ и $v_{\tau (F)j}$ противоположны. Но тогда и ориентации
базисов $w_{\tau (F)i}$ и $w_{\tau (F)j}$ противоположны, что и  означает, что
кубы $Q'$ и $R'$ лежат по разные стороны от $F'$.

Это показывает, что строение (и разнообразие) кубильяжей  зонотопа $Z({\bf V})$
зависит только от ориентированного матроида, порожденного векторной
конфигурацией {\bf V}. Мы будем активно  пользоваться этим простым замечанием
при рассмотрении вполне положительных конфигураций.

\section{Циклические (вполне положительные) вектор\-ные конфигурации}
\label{sec:cyclic_config}

До сих пор конфигурации {\bf V} векторов в пространстве $V$ были, фактически,
произвольными (если не считать требования общего положения). Начиная с этого
места мы ограничиваемся более  специальными конфигурациями, которые обычно
называются циклическими или вполне положительными. Здесь и далее мы будем
считать, что $V={\mathbb R}^d$ (с базисом $e_1,...,e_d$); кроме того, теперь
важную роль будет играть порядок  индексации векторов $v_i$ (то есть, мы
снабжаем множество индексов-цветов $[n]$ естественным порядком $1<2<...<n$).
\medskip

{\bf Определение.} Конфигурация векторов ${\bf V}=(v_1,...,v_n)$ называется
{\em вполне положительной} (или, не очень удачно, но более коротко, {\em
циклической}), если для любой  (возрастающей) $d$-шки $i_1<i_2<...<i_d$ из
$[n]$ определитель матрицы  $\textrm{Mat}(v_{i_1},...,v_{i_d})$, составленной
из столбцов-векторов  $v_{i_j}$, положителен.\medskip

Циклические конфигурации подобны между собой и определяются двумя числами $d$ и
$n\ge d$. По этой причине циклический  зонотоп $Z({\bf V})$ обозначается далее
просто как $Z(n,d)$, а множество его  кубильяжей -- как ${\bf Q}(n,d)$.

Удобный представитель такой конфигурации получается следующим  образом.
Отобразим прямую ${\mathbb R}$ в ${\mathbb R}^d$ с помощью  отображения ${\bf
v}_d$, ${\bf v}_d(t)=(1,t,t^2,...,t^{d-1})$. (Когда $d$ ясно из контекста, мы
опускаем инлекс $d$ и говорим просто про {\bf v}. Образ прямой ${\mathbb R}$
при таком      отображении называют \emph{кривой Веронезе} или \emph{кривой
моментов}.) Если  мы теперь возьмем конечное подмножество $\{t_1<...<t_n\}$ в
${\mathbb R}$, мы  получаем (индексированную множеством $[n]$) векторную
конфигурацию {\bf V}, состоящую из векторов $v_i={\bf v}(t_i)$, $i=1,...,n$.
Такая конфигурация вполне положительна, что видно  из формулы для определителей
Вандермонда.

Отметим два важных для дальнейшего свойства циклических конфигураций. Если мы
выбрасываем какой-то элемент $k$ из $[n]$, конфигурация остается циклической.
Так что мы можем говорить о проекциях $\pi _k:Z(n,d) \to Z([n]-\{k\},d)$ и
соответствующих редукциях кубильяжей (см. выше). И второе (применимое к
конфигурациям Веронезе): если мы рассмотрим каноническую проекцию $\pi $
пространства ${\mathbb R}^d$ на ${\mathbb R}^{d-1}$ (проекция вдоль последнего
базисного вектора $e_d$, или забывание $d$-ой координаты), то циклическая
конфигурация переходит в циклическую. В частности, мы получаем проекцию $\pi
:Z(n,d)\to Z(n,d-1)$, полезную для проведения индукции по $d$. Заметим еще, что
эту проекцию $\pi $ (вдоль $d$-ой координатной оси) можно понимать как проекцию
вдоль направления ${\bf v} (+\infty )$, потому что направление вектора ${\bf
v}(t)$ при  большом $t$ стремится к направлению вектора $e_d$. Неформальная же
причина интереса циклических конфигураций объясняется тем, что  кубильяжи
циклических зонотопов тесно связаны с высшими порядками Манина-Шехтмана, о чем
мы будем говорить в комбинаторной части  работы.

Мы уже говорили про редукцию цветов для общих зонотопов. В случае циклических
зонотопов особо важную роль играет редукция старшего цвета $n$. Дело в том, что
шов, остающийся после такой редукции, является мембраной относительно проекции
$\pi $ вдоль $d$-й  координатной оси. Расскажем более обстоятельно об этом в
следующем разделе.

           \section{Мембраны}\label{sec:membr}

{\bf Определение.} Пусть ${\mathcal Q}$ - кубильяж (циклического) зонотопа
$Z=Z(n,k)$. {\em Мембраной} этого кубильяжа, или мембраной, \emph{вписанной} в
кубильяж ${\mathcal Q}$, называется ($d-1$)-мерный подкомплекс ${\mathcal M}$ в
${\mathcal Q}$, который биективно проектируется на зонотоп $Z(n,d-1)$ (при
проекции $\pi $ вдоль последней $d$-ой координаты).

Множество мембран кубильяжа ${\mathcal Q}$ обозначается как ${\bf M}({\mathcal
Q})$.\medskip

Так что мембрана -- это некоторая $(d-1)$-мерная пленка в $Z$ (идущая по граням
кубильяжа ${\mathcal Q}$), гомеоморфная $(d-1)$-мерному диску, край которого --
в точности обод зонотопа $Z$ относительно проекции $\pi$. Мембрана делит
зонотоп на две части: нестрого до мембраны ($Z_-({\mathcal M})$) и нестрого
после нее ($Z_+({\mathcal M})$). Проекция клеток мембраны ${\mathcal M}$ (ее
$(d-1)$-мерных кубов, называемых \emph{пластинами}) дает при этом кубильяж
зонотопа $Z'=Z(n,d-1)=\pi (Z)$, который мы обозначаем $\pi ({\mathcal M})$.

Априори не ясно, существуют ли мембраны вообще и много ли их. Позже мы увидим,
что их достаточно много. А пока приведем пример двух универсальных
мембран.\medskip

{\bf Пример 1.} Мы уже говорили про это в разделе 2; давайте вернемся
применительно к циклической ситуации. Рассмотрим видимую (в направлении $d$-ой
координаты) часть границы $Z$, $\partial_-Z$. Очевидно, это мембрана (любого
кубильяжа зонотопа $Z$); ее проекция $\pi (\partial _-Z)$ дает кубильяж
зонотопа $Z'$, который мы называем {\em стандартным} (так именовались
соответствующие ромбические тайлинги двумерного зоногона в \cite{UMN}).
Невидимая  часть границы $\partial _+Z$ дает {\em анти-стандартный} кубильяж
зонотопа $Z'$.\medskip

Универсальный способ получения мембран дает редукция старшего  цвета. Пусть
${\mathcal Q}$ -- кубильяж зонотопа $Z=Z(n,d)$. И пусть  $\widetilde{{\mathcal
Q}}={\mathcal Q}_{-n}$ --  кубильяж зонотопа $\tilde Z=Z(n-1,d)$, полученный
редукцией цвета  $n$. Тогда  \emph{шов $\mathcal S $ от этой редукции является
мембраной}. В самом деле, замена  вектора $v_n$ на $d$-й базисный вектор
$e_d=(0,...,0,1)$ в ${\mathbb  R}^d$ приводит  к подобной векторной
конфигурации. Поэтому (см. раздел 4) можно  произвести экспансию шва $\mathcal
S $ в направлении $e_d$ и получить кубильяж  ${\mathcal Q}'$, подобный
${\mathcal Q}$. А это значит (см. основную теорему из раздела 3), что шов
$\mathcal S$ биективно проектируется на зонотоп $\pi (Z)=Z(n,d-1)$.

Напротив, имея мембрану ${\mathcal M}$ в кубильяже $\widetilde{{\mathcal  Q}}$
зонотопа $\widetilde Z=Z(n-1,d)$,      мы можем экспансировать эту мембрану в
направлении нового вектора  $v_n={\bf v}(t_n)$, где $t_n>t_{n-1}$. Дело в том,
что направление,  задаваемое вектором $v_n$, также трансверсально мембране
${\mathcal M}$ (как и      направление ${\bf v}(+\infty )$). Мы говорили об
этом в разделе 4 про подобные  конфигурации. В результате получаем новый
кубильяж ${\mathcal Q}$  зонотопа $Z=Z(n,d)$. `Раздвинутая' мембрана ${\mathcal
M}+[0,v_n]$ становится перегородкой старшего цвета $n$ в $\mathcal Q$.
Редуцируя по этому цвету, мы  возвращаемся к исходному кубильяжу
$\widetilde{{\mathcal Q}}$.

Таким образом, пара $(\widetilde{{\mathcal Q}},{\mathcal M})$, где
$\widetilde{{\mathcal Q}}$ -- кубильяж зонотопа      $\widetilde{Z}=Z(n-1,d)$,
а ${\mathcal M}$ -- мембрана в $\widetilde{{\mathcal Q}}$, определяет кубильяж
${\mathcal Q}$ зонотопа  $Z=Z(n,d)$. И на этом пути получаются все кубильяжи
зонотопа $Z$. В  самом деле, пусть ${\mathcal Q}$ -- произвольный кубильяж $Z$.
Возьмем в нем  перегородку ${\mathcal P}$ цвета $n$ и редуцируем ее. Шов
$\mathcal S $ от этой редукции, как уже говорилось, является мембраной для
редуцированного кубильяжа      $\widetilde{{\mathcal Q}}={\mathcal Q}_{-n}$
редуцированного зонотопа $Z(n-1,d)$. Экспансия этой мембраны $\mathcal S $
возвращает нас к исходному кубильяжу  ${\mathcal Q}$.

           Итог этого обсуждения можно подвести следующим утверждением.

\begin{prop} \label{pr:4}
Задать кубильяж зонотопа $Z(n,d)$ --  то же самое, что задать пару
$(\widetilde{{\mathcal Q}}, {\mathcal M})$, где $\widetilde{{\mathcal Q}}$ --
кубильяж зонотопа $Z(n-1,d)$, а ${\mathcal M}$ -- мембрана для
$\widetilde{{\mathcal Q}}$.
 \end{prop}

Или чуть иначе: существует естественное сюръективное  отображение ${\bf Q}(n,d)
\to {\bf Q}(n-1,d)$; слой его над любой точкой ${\mathcal Q}\in {\bf Q}(n-1,d)$
совпадает с множеством ${\bf M}({\mathcal Q})$ мембран для ${\mathcal
Q}$.\medskip

{\bf Пример 2.} Стандартный кубильяж зонотопа $Z(n,d)$ получается из
стандартного же кубильяжа зонотопа $Z'=Z(n-1,d)$ с помощью `задней'  мембраны
$\partial_+(Z')$ (невидимой части границы $Z'$). {См. рис. \ref{fig:7}, где
$d=2$.}

В самом деле, мы определяли стандартный кубильяж зонотопа $Z(n,d)$ как образ
`передней' (видимой)  мембраны $\partial_-Z(n,d+1)$ зонотопа $Z(n,d+1)$. Но как
выглядит последняя? Зонотоп  $Z(n,d+1)$ получается из зонотопа $Z(n-1,d+1)$
прибавлением отрезка  $[0,{\bf v}_{d+1}(n)]$. Соответственно, его видимая
граница состоит из видимой границы $\partial _-$ зонотопа $Z(n-1,d+1)$ плюс
фасеты вида  $F+[0,{\bf v}_{d+1}(n)]$, где $F$ пробегает множество невидимых (в
направлении координаты $d$)  фасет этой видимой границы $\partial _-$. При
проекции вдоль координаты $d+1$ первая половина превращается в стандартный
кубильяж зонотопа $Z'=Z(n-1,d)$, а фасеты $F$ -- в невидимые фасеты этого
зонотопа $Z'$, образуя в совокупности невидимую часть границы зонотопа $Z'$.
Поэтому вторая половина проектируется в точности на экспансию (по  цвету $n$)
задней (невидимой) мембраны зонотопа $Z'$.

Аналогично, антистандартный кубильяж получается экспансией по цвету $n$
передней (видимой) мембраны $\partial_-Z'$ антистандартного  кубильяжа зонотопа
$Z'=Z(n-1,d)$.\medskip

           \section{Кубы и флопы}\label{sec:cubes}

Рассмотрим более обстоятельно случай куба, зонотопа $C=Z(d,d)$. У него есть
всего две мембраны -- видимая (передняя) $\partial_-C$  и невидимая (задняя)
$\partial_+C$, которые при проекции $\pi$  дают стандартный и антистандартный
кубильяжи зонотопа $Z(d,d-1)$. Из подсчета числа вершин обода (см. (1.2))
видно, что на передней мембране имеется ровно одна `внутренняя' вершина $t$
(назовем ее \emph{хвостом}), как и на задней одна `внутренняя' вершина  $h$
(назовем ее \emph{головой}).

\begin{figure}[h]
\begin{center}
\includegraphics[scale=.25]{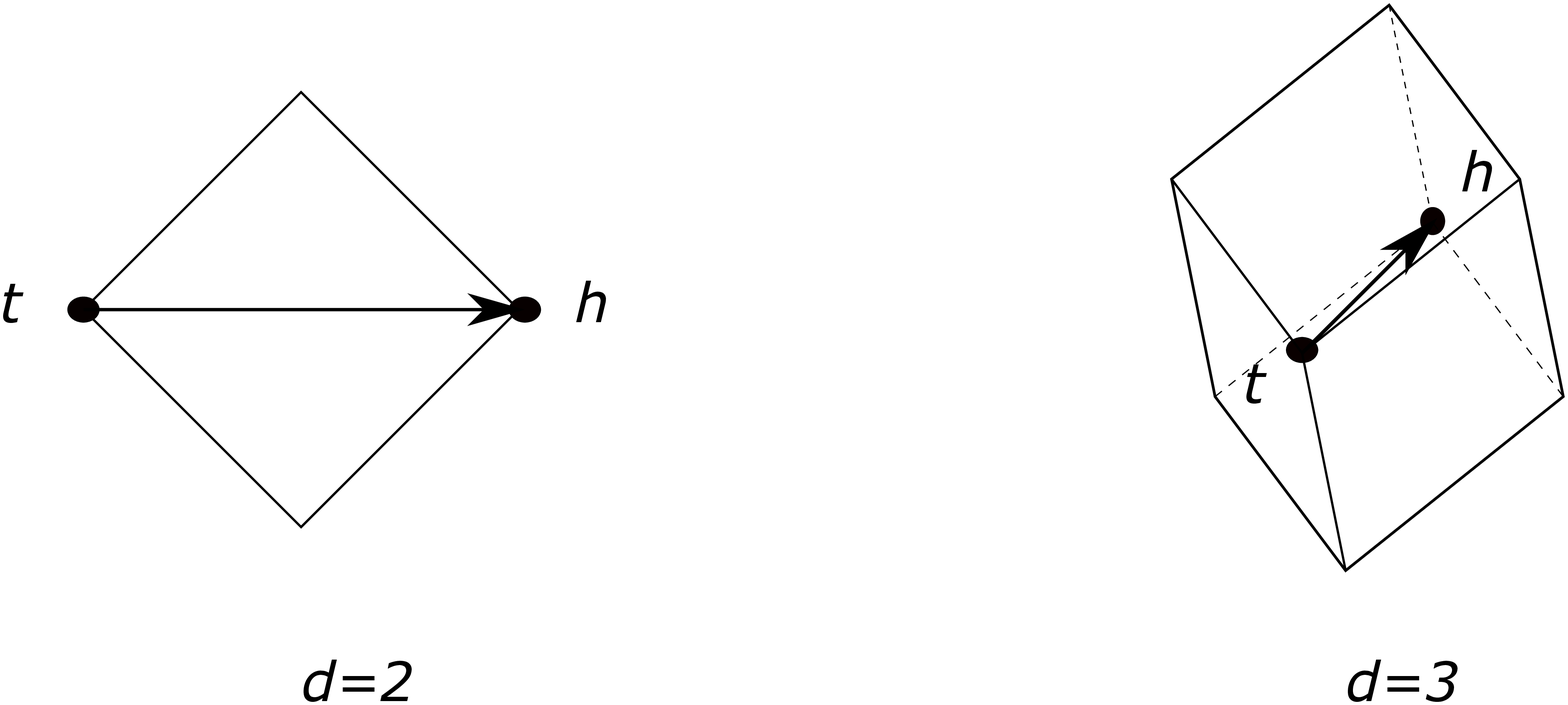}
\end{center}
  \caption{Хвост и голова}
 \label{fig:4}
  \end{figure}

Скажем более конкретно о спектрах этих вершин. Для этого вернемся к разделу 1,
где описывались фасеты зонотопа. Аналогично можно поступить и здесь, для
описания вершин $t$ и $h$. Напомним, что вектора $v_1,...,v_d$  задавались
числами $t_1<...<t_d$ и мы полагали $v_i=\textbf{v}(t_i)$. Возьмем некоторые
перемежающие их числа $s_1,...,s_{d-1}$, так что
$t_1<s_1<t_2<...<t_{d-1}<s_{d-1}<t_{d}$.  И в качестве линейного функционала на
пространстве $\mathbb{R}^{d}$ возьмем $\det(\textbf{v}(s_1), ...,
\textbf{v}(s_{d-1}), \bullet)$. Заметим, что этот функционал положителен на
векторах $v_{d}$ и $e_{d}$. Он положителен также на векторах $v_i$, когда $i$
имеет ту же четность, что и $d$, и отрицателен на векторах $v_i$, когда $i$
имеет противоположную с $d$ четность. Отсюда видно, что {головная} вершина
${h}$  есть сумма векторов $v_i$, таких что $d-i$ четно, тогда как {хвостовая}
вершина ${t}$  есть сумма векторов $v_i$, таких что $d-i$ нечетно. Так что
$sp(h)=\{d,d-2,...\}$, тогда как $sp(t)=\{d-1,d-3,...\}$. {Например, в случае
$d=3$ (см. рис. \ref{fig:4})} $sp(t)=2$, $sp(h)=13$ (как и далее, 13 обозначает
множество $\{1,3\}$). Очевидно, что в объединении эти множества-спектры дают
все $[d]$.

Стрелку, идущую из хвоста  $t$   в голову $h$, будем назвать \emph{хордой} куба
$C$.  В случае четного  $d$ хорда идет горизонтально, тогда как в при нечетном
$d$ она поднимается на 1 при каждом шаге, как видно из рис.  \ref{fig:4}.
{Здесь впервые начинает проявляться небольшое различие между четными и
нечетными $d$.}

После расмотрения изолированного куба обратимся к общей ситуации.  Пусть
${\mathcal Q}$ -- кубильяж зонотопа $Z=Z(n,d)$, и $Q$ -- некоторый куб этого
кубильяжа с хвостом $t_Q$, головой $h_Q$ и хордой $t_Q \to h_Q$. Если мы
проделаем это с каждым кубом кубильяжа, то эти хорды дадут структуру
направленного графа на множестве вершин кубильяжа ${\mathcal Q}$. Этот граф
состоит из путей-нитей, соединяющих (внутренние) вершины передней мембраны
$\partial_-Z$ с (внутренними же) вершинами задней мембраны $\partial_+Z$.
Каждый куб оказывается `нанизанным' на свою хорду, а кубы, нанизанные на пути,
выглядят как `гирлянды'. Считая, что вершины обода соединены сами с собой путем
длины 0, мы получаем интересную (и слабо изученную) `гирляндную' биекцию
множества вершин передней мембраны в множество вершин задней мембраны зонотопа
$Z$.

\begin{figure}[h]
\begin{center}
\includegraphics[scale=.25]{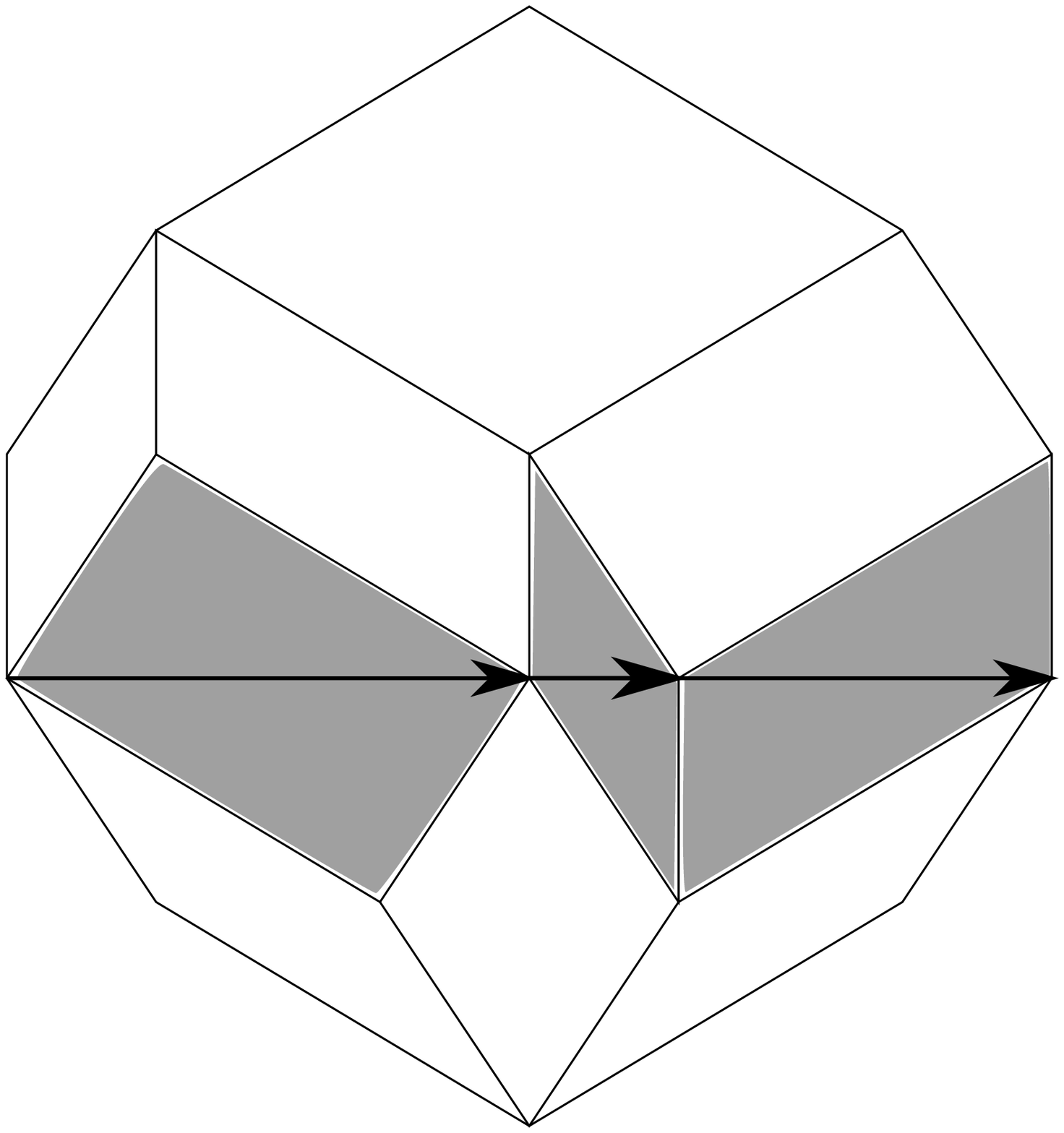}
\end{center}
  \caption{Гирлянда в зоногоне $Z(5,2)$}
 \label{fig:5}
  \end{figure}

Разумеется, можно не ограничиваться только граничными мембранами. Если
$\mathcal{M}$ и $\mathcal{M}'$ -- две произвольные мембраны (одного и того же
кубильяжа), то каждая гирлянда `протыкает' каждую из мембран ровно одит раз.
Тем самым мы получаем `гирляндную' биекцию между множеством вершин
$\mathcal{M}$ и множеством вершин $\mathcal{M}'$. Это лишний раз показывает,
что число вершин у каждой мембраны такое же, как у стандартной. На самом деле
мы получаем даже больше: при четном $d$ у всех мембран зонотопа $Z(n,d)$ на
каждой высоте одно и то же число вершин. Или: при нечетном $d$ у любого
кубильяжа зонотопа $Z(n,d)$ число вершин на каждом этаже одно и то же.\medskip

{\bf Пример.} {Проиллюстрируем гирляндную биекцию на примере стандартного
кубильяжа занотопа $Z(4,3)$. Как уже говорилось в предыдущем разделе, такой
кубильяж получается из куба $C=Z(3,3)$ экспансией по цвету 4 задней стороны
этого куба. Задняя сторона $\partial _+C$ состоит из трех пластинок-ромбов, так
что стандартный кубильяж $Z(4,3)$ состоит из четырех кубов: исходного $C$ и
трех `экспансированных' из этих пластинок. Мы изобразили все это на рис. 5, где
для наглядности вынули и слегка отодвинули первый куб $C$.}

\begin{figure}[h]
\begin{center}
\includegraphics[scale=.8]{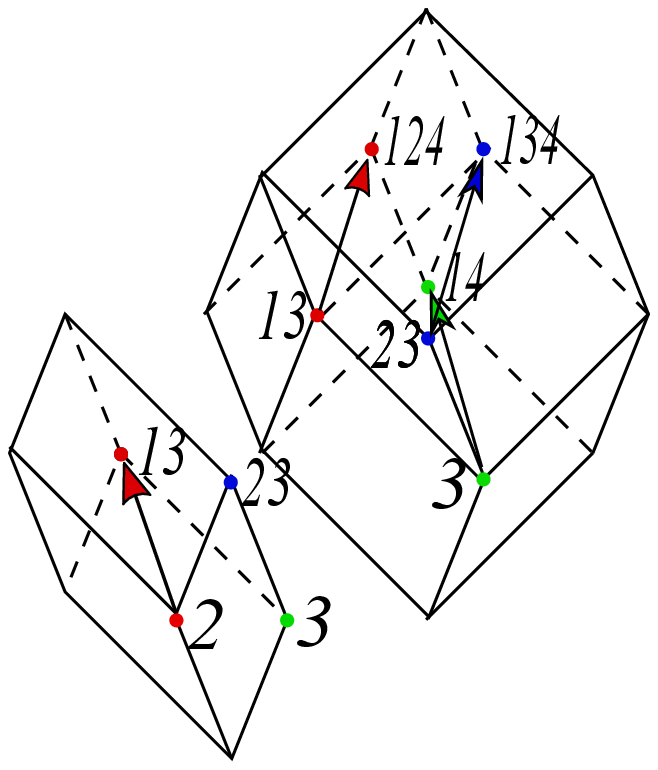}
\end{center}
  \caption{\small Гирляндное отображение для стандартного кубильяжа зонотопа $Z(4,3)$}
 \label{fig:6}
  \end{figure}

           Передняя сторона (или мембрана) зонотопа $Z(4,3)$ содержит,
      кроме периферийных (лежащих на ободе), три вершины со спектрами 2,3
      и 23. Так как гирляндное отображение $\gamma _S$ тривиально действует на
      ободе, нам нужно понять -- куда при отображении $\gamma _S$ переходят эти
      три вершины 2, 3 и 23. Глядя на рисунок, мы видим, что 2 переходит
      сначала в точку 13, которая затем переходит в точку 124. Точка 3
      переходит в 14, а точка 23 -- в 134. Окончательно,
                                                             $$
      \gamma _S:         \left\{ \begin{array}{c}
                                 2\mapsto 124 \\
                                 3 \mapsto 14 \\
                                 23 \mapsto 134
                               \end{array}\right..
      $$
Действуя аналогично, можно убедиться, что гирляндное отображение  $\gamma _A$
(для антистандартного кубильяжа $Z(4,3)$) устроено так:
                                                         $$
      \gamma _A:  \left\{ \begin{array}{c}
                                 2\mapsto 14 \\
                                 3 \mapsto 134 \\
                                 23 \mapsto 124
                               \end{array}\right..
                  $$

{Последнее, впрочем, можно увидеть и из следующего общего соображения.
Произвольный зоногон $Z$ симметричен относительно своего центра. Если
обозначить эту симметрию как $\alpha $, мы для любого кубильяжа $\mathcal Q $
можем рассмотреть симметричный ему кубильяж $\alpha \mathcal Q $ того же
зонотопа $Z$. При этом стандартный кубильяж переходит в антистандартный, и
наоборот. Пусть $S$ обозначает вершины зоногона на передней стороне $Z$, а $A$
-- вершины на задней стороне. Инволюция       $\alpha $ осуществляет биекцию
$S$ на $A$; подмножество $X$ переходит в свое дополнение $[n]-X$.}

\begin{prop} \label{pr:4} Если $\gamma _\mathcal Q :S \to A$ -- гирляндная биекция,
      связанная с кубильяжем $\mathcal Q $, то $\gamma _{\alpha
      \mathcal Q }=\alpha ^{-1}\gamma _\mathcal Q \alpha $.
\end{prop}

\begin{proof} При инволюции $\alpha $ каждая хорда куба $Q$ переходит в противоположно направленную хорду куба $\alpha Q$.
 \end{proof}

Пусть ${\mathcal M}$ -- мембрана в некотором кубильяже ${\mathcal Q}$.
Представим, что некоторый куб $Q$ кубильяжа ${\mathcal Q}$  касается мембраны
всей своей видимой границей $\partial _-Q$. (Ниже мы увидим, что такая
ситуация не исключительна и встречается постоянно.) В этом случае можно
образовать новую мембрану ${\mathcal M}'$, которая всюду совпадает с ${\mathcal
M}$, за исключением того, что этот фрагмент $\partial _-Q$ заменяется новым
фрагментом $\partial _+Q$. Иначе говоря, если мембрана ${\mathcal M}$ обходила
куб $C$      спереди, то теперь она обходит его сзади. Мы говорим, что мембрана
${\mathcal M}'$ получена из ${\mathcal M}$ с помощью {\em  повышающего флопа}
(а ${\mathcal M}$ получена из ${\mathcal M}'$ \emph{понижающим флопом}).

Это можно выразить чуть иначе. Пусть ${\mathcal Q}_-({\mathcal M})$ обозначает
набор кубов кубильяжа ${\mathcal Q}$, которые расположены до мембраны
${\mathcal M}$. Тогда ${\mathcal Q}_-({\mathcal M}')$ получается добавлением к
${\mathcal Q}_-({\mathcal M})$ ровно одного куба $Q$. Так что при повышающих
флопах мембрана движется `от нас', захватывая по очереди по одному кубу
кубильяжа, пока не дойдет до `задней' границы  зонотопа.

\section{Капсиды и флипы}\label{sec:capsid}

После обстоятельного рассмотрения куба естественно перейти к следующему по
сложности случаю зонотопа  $Z=Z(d+1,d)$. Такой зонотоп (или подзонотоп в
большем зонотопе $Z(n,d)$) будем для краткости называть {\em
капсидом}\footnote{Термин взят из вирусологии. Так называется белковая оболочка
вируса, обычно полиэдральная (типа ромбододекаэдра, то есть $Z(5,3)$).}.

Капсид представляется как проекция $(d+1)$-мерного куба
$\widehat{Z}=Z(d+1,d+1)$. Проекция передней мембраны этого куба дает
стандартный кубильяж капсида $Z$. А проекция `хвоста' $t_{\widehat{Z}}$ дает
единственную `центральную' (не лежащую на ободе) вершину стандартного
кубильяжа, которую мы обозначим как $c^{st}$. Ее спектр $sp(c^{st})=\{d,
d-2,...\}$. Эта вершина со всех сторон окружена кубами стандартного кубильяжа
капсида и является вершиной каждого куба этого кубильяжа. Симметрично невидимая
мембрана содержит `центральную' вершину $c^{an}$ антистандартного кубильяжа
капсида со спектром $\{d+1, d-1,...\}$. Например, при $d=2$ $sp(c^{st})=2$,
$sp(c^{an})=13$; при $d=3$  $sp(c^{st})=13$, $sp(c^{an})=24$.

В силу Следствия из Раздела 4 мы получаем, что капсид имеет только два кубильяжа, описанных выше. Подведем итог:

 \begin{prop} \label{pr:5}
Капсид $Z(d+1,d)$ допускает всего два кубильяжа - стандартный и
антистандартный. Стандартный кубильяж характеризуется тем, что кроме вершин
обода содержит вершину $c^{st}$ со спектром $\{d, d-2,...\}$. Симметрично,
антистандартный кубильяж характеризуется тем, что кроме вершин обода содержит
вершину $c^{an}$ со спектром $\{d+1, d-1,...\}$. \hfill $\Box$
 \end{prop}

Утверждение про два кубильяжа капсида  можно получить также из Предложения 3.
Последнее говорит, что задать кубильяж капсида $Z(d+1,d)$ -- это то же самое,
что задать мембрану в кубе  $C=Z(d,d)$. Но у куба есть всего две мембраны:
видимая часть границы $\partial_-C$ и невидимая часть границы $\partial_+C$.
Соответственно, у капсида $Z(d+1,d)$ есть два кубильяжа. Один получается
$(d+1)$-экспансией задней мембраны $\partial _+C$; это стандартный кубильяж.
Другой (антистандартный) получается путем экспансии передней мембраны
$\partial_-C$. См. рис. 7.

\begin{figure}
\begin{center}
\includegraphics[scale=.23]{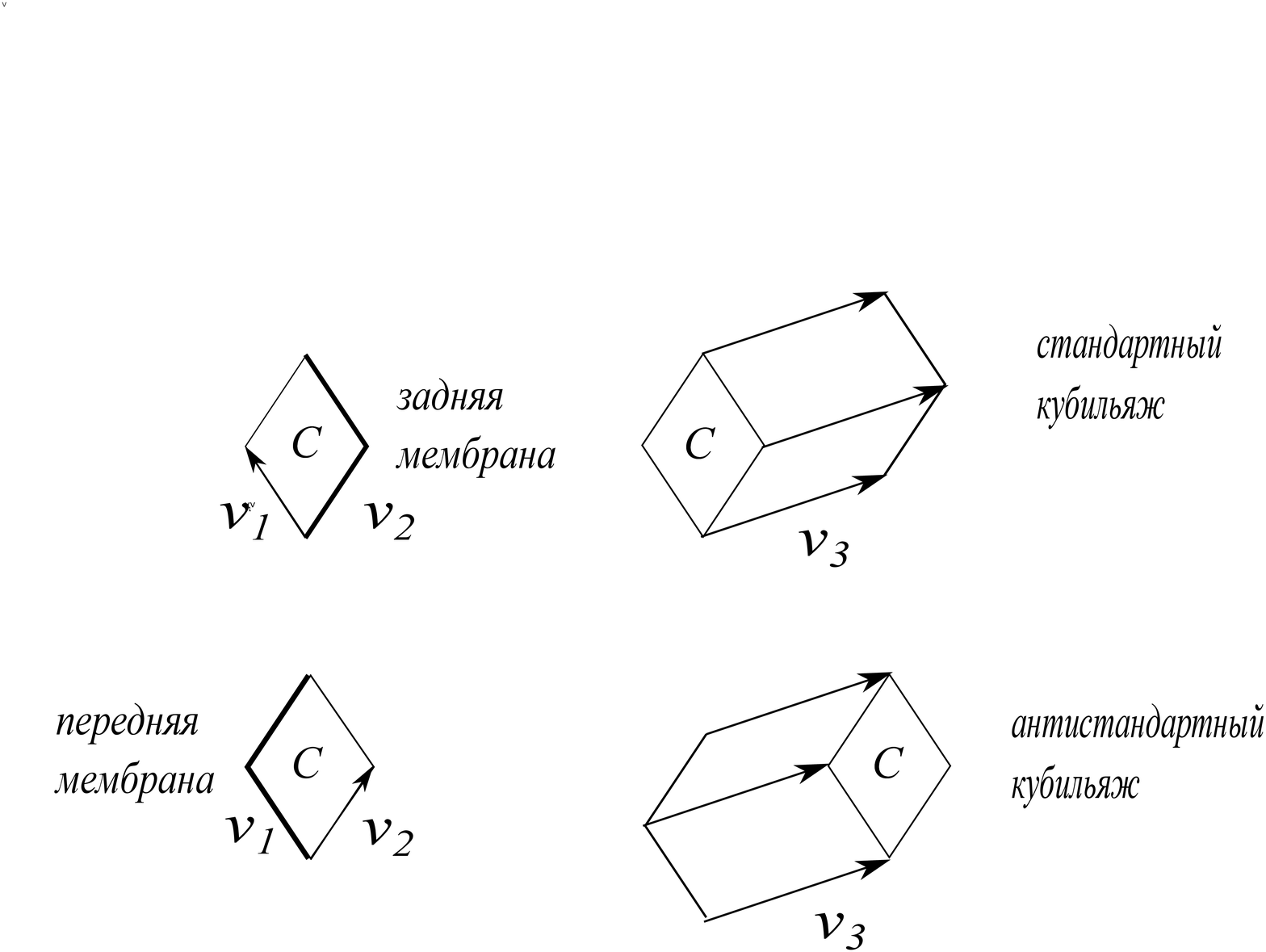}
\end{center}
  \caption{экспансия мембран куба}
 \label{fig:7}
  \end{figure}

Предположим теперь, что у нас имеется некоторый кубильяж ${\mathcal Q}$
произвольного зонотопа $Z(n,d)$, а в нем фрагмент  вида  $Z(d+1,d)$ (то есть
граница этого капсида является подкомплексом комплекса ${\mathcal Q}$).
Ограничение кубильяжа ${\mathcal Q}$ на этот капсид  $Z(d+1,d)$ является
стандартным или антистандартным кубильяжем. Замена кубильяжа этого фрагмента
противоположным кубильяжем называется {\em флипом} ({\em  повышающим}, когда
стандатрный фрагмент заменяется антистандартным, и      {\em понижающим} в
противном случае). Флип меняет кубильяж ${\mathcal Q}$ на некоторый новый
кубильяж ${\mathcal Q}'$, отличающийся от ${\mathcal Q}$ лишь локально, внутри
капсида.

Связь флипов с флопами, рассмотренными в предыдущем разделе, очевидна.
Предположим: что кубильяж ${\mathcal Q}$ зонотопа  $Z(n,d)$ получен как
проекция мембраны $\mathcal{M}$ в кубильяже $\widehat{{\mathcal Q}}$ "большего"
зонотопа $\widehat{Z}=Z(n, d+1)$. И мы делаем (скажем, повышающий) флоп с этой
мембраной, заменяя ее мембраной $\mathcal{M}'$. Тогда проекция $\mathcal{Q}'$
новой мембраны $\mathcal{M}'$ получается применением (повышающего) флипа к
кубильяжу $\mathcal{Q}'$.

Как мы увидим ниже, флипов  много, хотя бы в том смысле, что от любого
кубильяжа ${\mathcal Q}$ можно добраться до стандартного кубильяжа ${\mathcal
Q}_{st}$ с помощью понижающих флипов. Однако независимо       от этого факта
можно ввести понятие \emph{порядка} на множестве ${\bf Q}(n,d)$. А  именно,
${\mathcal Q} \le {\mathcal Q}'$, если от кубильяжа ${\mathcal Q}$ до кубильяжа
${\mathcal Q}'$ можно добраться серией повышающих флипов. Чтобы показать, что
флипов `много', полезно ввести важное понятие естественного порядка на
множестве кубов произвольного кубильяжа.

           \section{Порядок кубов в кубильяже} \label{sec:order_in_cubil}

Пусть ${\mathcal Q}$ -- кубильяж (циклического) зонотопа $Z(n,d)\subset
{\mathbb R}^d$. Оказывается, можно частично упорядочить некоторым естественным
образом кубы этого кубильяжа. Говоря грубо -- по росту  $d$-й координаты. Но
более правильно -- по тому, располагается ли  один куб за другим при смотрении
в направлении $d$-й координаты.

Пусть $Q$ -- некоторый куб кубильяжа ${\mathcal Q}$, и $F$ -- его фасета.
Скажем, что фасета $F$ куба $Q$ \emph{видимая} (или освещенная), если
существует прямая (в ${\mathbb R}^d$), параллельная координатному       вектору
$e_d$,  которая сначала пересекает $F$ и лишь после этого входит в куб $Q$.
Противоположная (к $F$ в $Q$) фасета называется \emph{невидимой} или
затененной. Очевидно, что из $2d$ фасет куба половина видимые, а  половина
затенненные.

Пусть теперь $Q$ и $Q'$ -- два куба кубильяжа ${\mathcal Q}$.  Скажем, что куб
$Q$ {\em непосредственно предшествует} кубу $Q'$ (и пишем  $Q\prec Q'$), если
$Q$ и $Q'$ соседние по фасете $F$, которая невидима для $Q$ и видима  для $Q'$.
Так на ${\mathcal Q}$ возникает бинарное отношение $\prec $ (или $\prec
_{\mathcal Q}$, если быть точнее).

\begin{figure}[h]
\begin{center}
\includegraphics[scale=0.4]{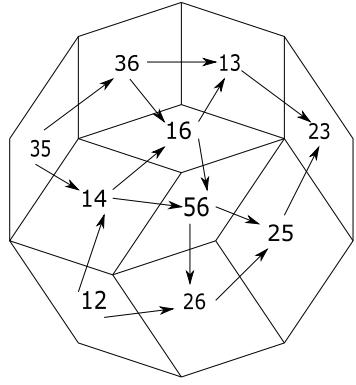}
\end{center}
\caption{Отношение $\prec$ на кубильяже зоногона $Z(5,2)$.}
 \label{fig:8}
  \end{figure}

 Рисунок \ref{fig:8} подсказывает следующее обобщение:

 \begin{prop} \label{pr:6} Отношение $\prec $ ациклично.
 \end{prop}

Это утверждение, как и следующая Лемма о реверсе доказываются в Дополнении 4.
Лемма относится к следующей ситуации.  Пусть ${\mathcal P}={\mathcal P}_n$ -
перегородка      старшего цвета $n$ в      кубильяже ${\mathcal Q}$ зонотопа
$Z=Z(n,d)$. Редукция этой перегородки дает      кубильяж $\widetilde{{\mathcal
Q}}$ зонотопа $\widetilde{Z}=Z(n-1,d)$ и шов-мембрану       ${\mathcal M}$ в
$\widetilde{{\mathcal Q}}$. Проектируя      ${\mathcal M}$ вдоль $d$-ой
координатной оси $e_d$, мы получаем      кубильяж ${\mathcal Q}'=\pi ({\mathcal
M})$ зонотопа $       Z'=Z(n,d-1)$.

Пусть теперь $Q$ и $R$ -- два куба из перегородки ${\mathcal P}$, связанные
отношением непосредственного предшествования, $R\prec Q$. При редукции  кубы
$Q$ и $R$ превращаются в пластинки-фасеты $\gamma (Q)$ и       $\gamma (R)$
мембраны ${\mathcal M}$; их образы при проекции $\pi $ дают кубы $Q'$ и $R'$
кубильяжа ${\mathcal Q}'$. Очевидно, что $ Q'$ и $R'$ тоже соседние в кубильяже
${\mathcal Q}'$, и мы  можем сравнить их      отношением непосредственного
предшествования $\prec '$ в кубильяже  ${\mathcal Q}'$.

 \begin{lemma} \label{lm:2} \rm{(о реверсе).}
\emph{Если кубы  $R$ и $Q$ лежат в перегородке ${\mathcal P}$ и $R\prec Q$ , то $Q' \prec' R'$.}
 \end{lemma}

\begin{figure}[h]
\begin{center}
\includegraphics[scale=0.2]{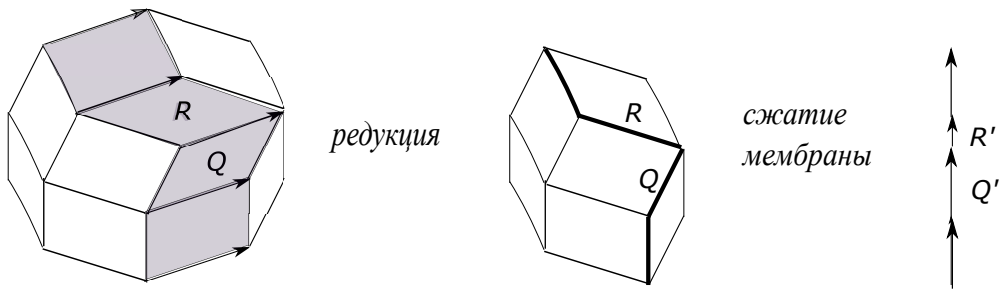}
\end{center}
\caption{\small На рисунке слева для кубов $R$ и $Q$ из перегородки выполнено $R\prec Q$.
Справа нарисован кубильяж $\mathcal{Q'}$, полученный сжатием перегородки.
Для соотвествующих кубов $R'$ и $Q'$ выполнено противоположное соотношение $Q' \prec R'$.}
 \label{fig:9}
  \end{figure}

Непосредственным следствием Леммы о реверсе является то, что \emph{ограничение
отношения $\prec $ на перегородку ${\mathcal P}$  (цвета $n$) противоположно
отношению $\prec '$ для кубильяжа} ${\mathcal Q}'$ (здесь мы отождествляем
множество ${\mathcal P}$ с ${\mathcal Q}'$ через $P \mapsto \gamma (P) \mapsto
P'$).\medskip

{\bf Определение.} {\em  Естественным порядком} на кубильяже ${\mathcal Q}$
называется рефлек\-сивно-транзитивное замыкание отношения $\prec $; мы
обозначаем его $\preceq $ или $\preceq_{\mathcal Q}$.\medskip

По определению это отношение $\preceq $ является предпорядком;  нетривиальность
в том, что оно является \emph{порядком}, то есть антисимметрично ($Q\preceq Q'$
и $Q'\preceq Q$ влекут  $Q=Q'$). Именно это и дает Предложение 6.\medskip

{\bf Пример 1.}  Пусть $Q$ и $Q'$ -- два куба из одного туннеля, и $Q'$ идет
после $Q$ при движении от видимой границы зонотопа к невидимой. Тогда $Q\preceq
Q'$.\medskip

{\bf Пример 2.} Пусть куб $Q$ `частично затеняет' $Q'$, то есть имеется прямая,
параллельная координатному вектору $e_n$, которая протыкает (пересекает по
внутренней точке) куб $Q$ раньше, чем $Q'$. Тогда $Q\preceq Q'$. В самом деле,
надо рассмотреть цепочку кубов $Q_1,...,Q_k$, которые эта прямая пересекает по
пути от $Q$ к $Q'$. Тогда $Q\prec Q_1\prec  ...\prec Q_k\prec Q'$.

Отметим, что мы могли бы начать с этого более сильного  отношения затенения и
(после транзитивного замыкания) получить то же отношение $\preceq $.\medskip

{\textbf{Пример 3.} Пусть куб $Q'$ располагается непосредственно после куба $Q$
в некоторой гирлянде (то есть голова $Q$ совпадает с хвостом $Q'$, см. раздел
7). Тогда $Q$ затеняет $Q'$ и по предыдущему примеру $Q\preceq Q'$. }\medskip

Пусть ${\mathcal Q}$ -- кубильяж, а ${\mathcal Q}_{-i}$ -- редукция ${\mathcal
Q}$ по цвету $i$. Как множество кубов  ${\mathcal Q}_{-i}$ отождствляется с
${\mathcal Q}-{\mathcal P}_i$, так что можно рассматривать ${\mathcal Q}_{-i}$
как подмножество в ${\mathcal Q}$.

 \begin{prop} \label{pr:7}
Ограничение порядка $\preceq $ на ${\mathcal Q}_{-i}$ сильнее, чем порядок
$\preceq _{-i}$ на  ${\mathcal Q}_{-i}$.
 \end{prop}

Иначе говоря, если для кубов $Q$ и $R$ из  редуцированного кубильяжа ${\mathcal
Q}_{-i}$ выполнено соотношение $Q\preceq _{-i}R$, то $Q\preceq R$.

Для доказательства мы должны проверить, что если $Q$ непосредственно
предшествует $R$ в ${\mathcal Q}_{-i}$, то $Q\preceq R$. Соотношение $Q\prec
_{-i} R$ означает, что кубы $Q$ и $R$ соседние в ${\mathcal Q}_{-i}$  по фасете
$F$, причем $Q$ располагается раньше. Возможны два  случая. Первый -- фасета
$F$ не является пластиной шва $\mathcal S$ редукции. Тогда $Q$ и $R$ соседние
уже в ${\mathcal Q}$ и все очевидно. Второй -- фасета $F$ принадлежит шву. Но
тогда она получилась редукцией цвета $i$ в кубе $S=F+[0,v_i]$. И очевидно, что
(в кубильяже ${\mathcal Q}$) мы имеем соотношения $Q\prec S\prec R$. См.
рисунок \ref{fig:10}.

  \begin{figure}[htb]
\unitlength=1mm
\special{em:linewidth 0.4pt}
\linethickness{0.4pt}
\begin{picture}(120.00,35)
\put(25,10){\line(-1,1){10}}
\put(35,20){\line(-1,1){10}}
\put(50,25){\line(-1,1){10}}
\put(25,30){\line(-1,-1){10}}
\put(35,20){\line(-1,-1){10}}
\put(35,20){\line(3,1){15}}
\put(25,30){\line(3,1){15}}
\put(24.00,20.00){\makebox(0,0)[cc]{$Q$}}
\put(30.00,24.00){\makebox(0,0)[cc]{$F$}}
\put(37.00,28.00){\makebox(0,0)[cc]{$R$}}
\put(80,30){\line(-1,-1){10}}
\put(90,20){\line(-1,-1){10}}
\put(80,10){\line(-1,1){10}}
\put(120,20){\line(-1,1){10}}
\put(90,20){\line(-1,1){10}}
\put(105,15){\line(-1,1){10}}
\put(95,25){\line(3,1){15}}
\put(105,15){\line(3,1){15}}
\put(105.00,15.00){\vector(-3,1){15.00}}
\put(95.00,25.00){\vector(-3,1){15.00}}
\put(38.00,5.00){\makebox(0,0)[cc]{$\mathcal{Q}_{-i}$}}
\put(96.00,5.00){\makebox(0,0)[cc]{$\mathcal{Q}$}}
\put(96.00,15.00){\makebox(0,0)[cc]{$v_i$}}
\put(80.00,20.00){\makebox(0,0)[cc]{$Q$}}
\put(107.00,23.00){\makebox(0,0)[cc]{$R$}}
\put(93.00,22.5){\makebox(0,0)[cc]{$S$}}
\end{picture}
  \caption{}
 \label{fig:10}
  \end{figure}

           \section{Порядок на кубильяжах капсида} \label{sec:order_in_kapsid}

Наиболее просто Предложение 7 работает для стандартного кубильяжа. Дело в том,
что редукция $\widetilde{\mathcal Q}={\mathcal Q}_{-n}$ стандартного кубильяжа
${\mathcal Q}={\mathcal Q}_{st}$ зонотопа $Z(n,d)$ по цвету $n$ является
стандартным кубильяжем (зонотопа $Z(n-1,d)$). Более того, сжатие перегородки
${\mathcal P}={\mathcal P}_n$ цвета  $n$ в ${\mathcal Q}$ дает антистандартный
кубильяж зонотопа      $Z'=Z(n-1,d-1)$. Этими  соображениями можно
воспользоваться, чтобы полностью описать  естественный порядок на стандартном
кубильяже. Мы проделаем это упражнение для случая $n=d+1$, то есть для капсида.

Пусть ${\mathcal Q}$ -- стандартный кубильяж капсида  $Z=Z(d+1,d)$.
Редуцированный кубильяж $\widetilde{{\mathcal Q}}$ -- это тривиальный кубильяж
куба $\widetilde{Z}=Z(d,d)$. Перегородка ${\mathcal P}$ (цвета $n=d+1$)
примыкает к невидимой (задней) части $\partial _-(\widetilde{Z})$ границы куба
$\widetilde{Z}$; более точно, она получается суммированием этой задней части
границы с отрезком $[0,v_n]$. Поэтому кубы из перегородки имеют вид
$F+[0,v_n]$, где $F$ пробегает  невидивые фасеты куба $\widetilde{Z}$.
{(Посмотрите на рис. 6 и 7.)} Видно, что все они $\succ$  куба $\widetilde{Z}$.
И остается разобраться с порядком на кубах перегородки ${\mathcal P}$. Но, как
мы знаем из Леммы о реверсе, этот порядок противоположен порядку кубильяжа,
полученного проекцией $\pi $ (сжатием) задней мембраны $\partial
_+(\widetilde{Z})$, то есть антистандартного кубильяжа `меньшего' капсида
$Z'=Z(d,d-1)$. Так что порядок на кубах из ${\mathcal P}$ -- это в точности
порядок на стандартном кубильяже капсида $Z(d,d-1)$. Индукцией по $d$ мы
получаем

 \begin{prop} \label{pr:8}
Естественный порядок на стандартном кубильяже капсида $Z(d+1,d)$ -- полный (или линейный).
 \end{prop}

\begin{figure}[htb]
\unitlength=.6mm
\special{em:linewidth 0.4pt}
\linethickness{0.4pt}
\begin{picture}(80.00,42.00)(-50,0)
\put(60,45){\line(2,-1){20}}
\put(40,35){\line(2,-1){20}}
\put(40,15){\line(2,-1){20}}
\put(60,5){\line(0,1){20}}
\put(40,15){\line(0,1){20}}
\put(80,15){\line(0,1){20}}

\put(60,5){\line(2,1){20}}
\put(40,35){\line(2,1){20}}
\put(60,25){\line(2,1){20}}
\put(50.00,20.00){\makebox(0,0)[cc]{12}}
\put(70.00,20.00){\makebox(0,0)[cc]{23}}
\put(60.00,35.00){\makebox(0,0)[cc]{13}}
\put(52.00,23.00){\vector(2,3){6.00}}
\put(62.00,32.00){\vector(3,-4){6.67}}
\put(53.00,19.00){\vector(1,0){13.00}}
\end{picture}
 \caption{Порядок на стандартном тайлинге зоногона Z(3,2).}
 \label{fig:11}
  \end{figure}

Иначе говоря, это цепь из $n=d+1$ элемента. Более того, мы      видим, что уже
отношение $\prec $ в этом случае транзитивно. Куб $\widetilde{Z}$ --
минимальный элемент этого порядка. Заметим, что его тип равен  $[d] = [n]-n$.
Та же индукция, что и выше, показывает, что в терминах типов кубов порядок на
стандарном кубильяже капсида $Z(d+1,d)$  совпадает с лексикографическим
порядком $\prec _{lex}$ на множестве      $Gr([d+1],d)$ подмножеств размера $d$
в $[d+1]$. Например, при  $d=6$ этот порядок выглядит так:
           $$
                12345 \prec  12346 \prec  12356 \prec  12456 \prec       13456 \prec  23456.
             $$

Для антистандартного кубильяжа капсида  порядок $\prec $ тоже линейный, но в
терминах типов меняется на  противоположный, антилексикографический, $\prec
_{antilex}$.

Мы оставляем как упражнение разбор следующего по сложности случая: стандартного
кубильяжа зонотопа $Z(d+2,d)$. Ответ для стандартного кубильяжа зонотопа
$Z(6,4)$ изображен на рис. 12.\footnote{Отметим аналогию с рисунками колчанов в
\cite[Table 1]{OT}.}

\begin{figure}[h]
\unitlength=.75mm
\special{em:linewidth 0.4pt}
\linethickness{0.4pt}
\begin{picture}(110.00,55)(-35,0)
\put(31.00,10.00){\makebox(0,0)[cc]{1234}} \put(50.00,10.00){\makebox(0,0)[cc]{1236}}
\put(70.00,10.00){\makebox(0,0)[cc]{1256}} \put(90.00,10.00){\makebox(0,0)[cc]{1456}}
\put(110.00,10.00){\makebox(0,0)[cc]{3456}} \put(40.00,20.00){\makebox(0,0)[cc]{1235}}
\put(60.00,20.00){\makebox(0,0)[cc]{1246}} \put(80.00,20.00){\makebox(0,0)[cc]{1356}}
\put(100.00,20.00){\makebox(0,0)[cc]{2456}} \put(50.00,30.00){\makebox(0,0)[cc]{1245}}
\put(70.00,30.00){\makebox(0,0)[cc]{1346}} \put(90.00,30.00){\makebox(0,0)[cc]{2356}}
\put(60.00,40.00){\makebox(0,0)[cc]{1345}} \put(80.00,40.00){\makebox(0,0)[cc]{2346}}
\put(70.00,50.00){\makebox(0,0)[cc]{2345}}

\put(32.00,12.00){\vector(1,1){6.00}} \put(42.00,22.00){\vector(1,1){6.00}}
\put(52.00,32.00){\vector(1,1){6.00}} \put(62.00,42.00){\vector(1,1){6.00}}
\put(52.00,12.00){\vector(1,1){6.00}} \put(62.00,22.00){\vector(1,1){6.00}}
\put(72.00,32.00){\vector(1,1){6.00}} \put(72.00,12.00){\vector(1,1){6.00}}
\put(82.00,22.00){\vector(1,1){6.00}}
\put(92.00,12.00){\vector(1,1){6.00}}

\put(42.00,18.00){\vector(1,-1){6.00}} \put(62.00,18.00){\vector(1,-1){6.00}}
\put(52.00,28.00){\vector(1,-1){6.00}}
\put(82.00,18.00){\vector(1,-1){6.00}} \put(72.00,28.00){\vector(1,-1){6.00}}
\put(62.00,38.00){\vector(1,-1){6.00}} \put(102.00,18.00){\vector(1,-1){6.00}}
\put(92.00,28.00){\vector(1,-1){6.00}} \put(82.00,38.00){\vector(1,-1){6.00}}
\put(72.00,48.00){\vector(1,-1){6.00}}
\end{picture}
 \caption{Горка}
 \label{fig:12}
  \end{figure}

При рассмотрении естественного порядка на общем кубильяже важную роль играет
следующее соображение. Каждый капсид (то есть подзоногон $Z(d+1,d)$ в кубильяже
${\mathcal Q}$) задает свой тип $K$, подмножество размера $d+1$ в $[n]$. В свою
очередь, пусть $K\subset [n]$ -- произвольное подмножество размера  $d+1$.
Рассмотрим в ${\mathcal Q}$ множество $\mathcal F (K)$, состоящее из кубов
$Q\in {\mathcal Q}$, типы которых содержатся в $K$. Число таких кубов равно
$d+1$, и они  параметризуются своими типами $K-i$, где $i$ пробегает $K$. В
совокупности эти кубы представляют как бы `разбежавшийся'  капсид (настоящий
капсид состоит из кубов, тесно прижатых друг к другу). Редуцируя цвета, не
входящие в $K$, мы `прижмем' эти кубы друг к другу. Более точно, пусть
${\mathcal Q}_K$ обозначает кубильяж, полученный из ${\mathcal Q}$  редукцией
всех цветов, не принадлежащих $K$. Как множество, ${\mathcal Q}_K$
отождествляется с $\mathcal F (K)$. Кубильяж ${\mathcal Q}_K$ является
кубильяжем капсида  $Z(K,d)$ и поэтому либо стандартный, либо антистандартный.
И так как  (см. Пример 3 выше) естественный порядок $\preceq $ на $\mathcal F
(K)$ совпадает с порядком на ${\mathcal Q}_K$, то он во-первых, линейный (то
есть $\mathcal F (K)$ - это  цепь для $\preceq $), и во-вторых, в терминах
типов он либо      лексикографический, либо антилексикографический. Причем
какой --  легко понять. В множестве $K$ как подмножестве $[n]$ есть
максимальный элемент, пусть это $k$. Тогда все определяется тем -- будет ли куб
$Q$ с типом $K-k$  расположен раньше перегородки ${\mathcal P}_k$ цвета $k$ в
${\mathcal  Q}$ (и тогда $Q$ -- минимальный элемент в $\mathcal F(K)$, а
порядок лексикографический), либо позже перегородки (и тогда куб $Q$
максимален, а порядок $\preceq $ на подмножестве $\mathcal F (K)$
антилексикографический).

Итак, ограничения $\preceq $ на такие `капсидо-подобные'  системы $\mathcal F
(K)$  в ${\mathcal Q}$ являются линейными порядками (цепями). В свою очередь,
эти  цепи однозначно определяют $\preceq $. Более точно,

 \begin{prop} \label{pr:9}
Естественный порядок $\preceq $ на кубильяже ${\mathcal Q}$ -- это
транзитивное замыкание цепей $(\mathcal F (K),\preceq _{\mathcal F (K)})$ по
всем $K\subset [n]$ размера $d+1$.
 \end{prop}

Для доказательства достаточно проверить, что если $Q\prec R$ для кубов $Q$ и
$R$ кубильяжа ${\mathcal Q}$, то $Q$ и $R$ встечаются в некотором
`разбросанном' капсиде $\mathcal F (K)$. В самом деле, надо взять в качестве
$K$ объединение типов $Q$ и $R$; так как эти кубы  соседние, размер $K$ равен
$d+1$. \hfill $\Box$ \medskip

Таким образом, естественный порядок $\preceq $ на кубильяже ${\mathcal Q}$
можно было бы определять не через отношение $\prec $ `непосредственного
следования', но через цепи $\mathcal F (K)$ размера $d+1$. Мы вернемся к этому
обстоятельству в комбинаторной части.

           \section{Стэки и мембраны} \label{sec:stack-membr}

{\bf Определение.} Подмножество $\mathcal S $ кубов в  ${\mathcal Q}$
называется {\em стэком}, если с каждым кубом $Q$  подмножество $\mathcal S $
содержит и меньшие (относительно $\prec$ или $\preceq $) кубы. Иными словами,
это порядковый идеал в посете $({\mathcal Q}, \preceq)$. \emph{Телом} стэка
$\mathcal S $ называется подмножество в $Z=Z(n,k)$, являющееся объединением
кубов, входящих в $\mathcal S $, плюс множество точек видимой границы зонотопа
$Z$.\medskip

Интерес стэков вызван тем, что они задают мембраны кубильяжа ${\mathcal Q}$. В
самом деле, любая прямая, параллельная $d$-ой оси координат, проходя внутри
зонотопа $Z$, сначала идет по телу стэка $\mathcal S$, потом в какой-то момент
пересекает границу стэка и уже не возвращается в стэк. Множество этих
`последних' точек стэка образует мембрану ${\mathcal M}={\mathcal M}(\mathcal S
)$ в кубильяже ${\mathcal Q}$. В самом деле, оно состоит из фасет кубильяжа
${\mathcal Q}$ и биективно  проектируется на $Z'=\pi (Z)$. Обратно, если
${\mathcal M}$ --  мембрана кубильяжа ${\mathcal Q}$, и мы соберем все кубы,
расположенные до мембраны, мы получим стэк $\mathcal S ({\mathcal M})$.
Очевидно, что эти операции взаимно обратны, и мы получаем

\begin{prop} \label{pr:10}
Множество ${\bf S}({\mathcal Q})$ стэков в кубильяже ${\mathcal Q}$ и множество
мембран ${\bf M}({\mathcal Q})$ естественно биективны. \hfill$\Box$
 \end{prop}

\noindent {\bf Следствие.} {\em Зонотоп $Z(d+2,d)$ имеет $2(d+2)$ различных кубильяжей.}\medskip

 \begin{proof}
Согласно Предложению~\ref{pr:3} из раздела 6, задать кубильяж зонотопа
$Z(d+2,d)$ -- это то же самое, что задать мембрану в некотором кубильяже
капсида $Z(d+1,d)$. У последнего имеется всего два кубильяжа -- стандартный и
антистандартный.  И на каждом из них естественный порядок является цепью (длины
$d+1$).  Соответственно, имеется $d+2$ стэка, а в силу предыдущего предложения
-- $d+2$ мембраны.
 \end{proof}

Легко понять, что {посет $\textbf{Q}(d+2,d)$ состоит из двух ветвей кольца размера $2(d+2)$, как на рисунке \ref{fig:13}.}

\begin{figure}[h]
\begin{center}
\includegraphics[scale=0.3]{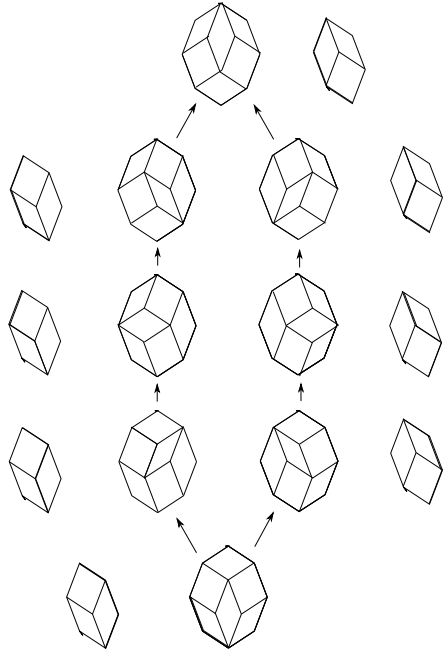}
\end{center}
  \caption{\small "Кольцо". Изображены восемь кубильяжей зоногона $Z(4,2)$. Стрелки
указывают повышающие флипы. Сбоку от каждого кубильяжа помещен кубильяж
$Z(3,2)$ и мембрана в нем, задающие (см. Предложение~\ref{pr:4})
соответствующий кубильяж $Z(4,2)$.}
 \label{fig:13}
  \end{figure}

Мембраны в кубильяже ${\mathcal Q}$ естественно упорядочены: мы говорим, что
мембрана ${\mathcal M}$ кубильяжа ${\mathcal Q}$ встречается \emph{раньше}, чем
мембрана ${\mathcal M}'$, если любая прямая, параллельная $e_d$, пересекает
${\mathcal M}$ не позже, чем ${\mathcal M}'$. В терминах стэков это в точности
отношение включения соответствующих подмножеств в ${\mathcal Q}$). Мы получаем
отсюда следующие  утверждения.

 \begin{prop} \label{pr:11}
{\rm a)} Множество мембран ${\bf M}({\mathcal Q})$ кубильяжа ${\mathcal Q}$ образует дистрибутивную решетку.

{\rm b)} Эта решетка градуирована (как и любая дистрибутивная, см.
\cite{Sta-1}, 3.4); ранг мембраны ${\mathcal M}$ равен числу элементов в стэке
$\mathcal S ({\mathcal M})$.

{\rm c)} Для любой мембраны ${\mathcal M}$, если она не максимальна, существует
куб $Q\in {\mathcal Q}$, который всеми своими видимыми фасетами входит в
мембрану ${\mathcal M}$. Добавление этого куба приводит к мембране ${\mathcal
M}'$, которая отличается от ${\mathcal M}$ повышающим флопом. \hfill$\Box$
  \end{prop}

Таким образом, мы видим, что  любую мембрану ${\mathcal M}$ в ${\mathcal Q}$
можно соединить с минимальной мембраной (видимой  частью границы зонотопа)
понижающими флопами (а с максимальной мембраной -- повышающими флопами). А это
означает также, что кубильяж $\pi ({\mathcal M})$ зонотопа $Z'=Z(n,d-1)$ можно
превратить в стандартный кубильяж зонотопа $Z'$ серией понижающих флипов.

          \section{Существование флипов}\label{sec:exist_flip}

Понятие флипа уже вводилось в разделе 8 как замена одного кубильяжа капсида
$Z(d+1,d)$ в $Z$ другим кубильяжем. Здесь  мы покажем существование флипов.

 \begin{theorem} \label{tm:1}
Пусть ${\mathcal Q}$ -- кубильяж  зонотопа $Z=Z(n,k)$. Если кубильяж ${\mathcal
Q}$ отличен от стандартного, то в нем можно сделать понижающий флип.
  \end{theorem}

Иными словами, если в кубильяже нельзя сделать понижающий флип, то он
стандартный. Верно, конечно, и обратное: стандартный  кубильяж не допускает
понижающих флипов. Это видно, например, из  того, что стандартный кубильяж
реализуется видимой, или передней, мембраной в (любом) кубильяже зонотопа $\hat
Z^=Z(n,d+1)$, и  `отступать назад' от такой мембраны уже некуда. \medskip

\begin{proof}
Рассмотрим перегородку ${\mathcal P}$ цвета $n$ в ${\mathcal Q}$. Если она
проходит вплотную к невидимой стороне $Z$, мы просто `выбросим' ее как из
${\mathcal Q}$, так и из $Z$, и получим кубильяж ${\mathcal Q}-{\mathcal P}$
зонотопа  $Z(n-1,d)$. Если этот кубильяж ${\mathcal Q}-{\mathcal P}$
стандартный, тогда и      ${\mathcal Q}$ - стандартный, вопреки  предположению.
Поэтому кубильяж ${\mathcal Q}-{\mathcal P}$  зонотопа $Z(n-1,d)$ не
стандарный, и по индукции мы можем сделать в нем (а поэтому и в  $Z$)
понижающий флип.

Таким образом мы можем считать, что перегородка ${\mathcal P}$ не примыкает  к
невидимой части границы зонотопа $Z$, и за этой перегородкой  реально есть
кубы. Возьмем среди таких кубов (строго за  перегородкой ${\mathcal P}$)
минимальный куб $Q$ относительно естественного  порядка $\preceq $ на
${\mathcal Q}$. Этот куб всей видимой  стороной примыкает к перегородке
${\mathcal P}$. И если мы добавим к этому кубу $Q$ все кубы из      перегородки
${\mathcal P}$, которые примыкают к $Q$, мы получим желаемый подзонотоп $Z_0$
(капсид). В самом деле, эти кубы устроены как суммы отрезка $[-v_n,0]$ на
видимые фасеты $Q$. И мы получаем      подзонотоп типа $Bn$, где $B$ - тип куба
$Q$. Это как раз такой набор кубов, с которым можно делать понижающий флип, см.
раздел 8.
\end{proof}

В результате такого понижающего флипа с участием цвета $n$ мы передвигаем
(отодвигаем) перегородку цвета $n$ `на один куб' по направлению к  невидимой
границе зонотопа $Z$. Двигаясь так шаг за шагом, мы  передвинем эту перегородку
вплотную к невидимой границе $Z$. Однако  мы могли бы сделать это сразу, за
один `большой' шаг. Для этого обозначим через ${\mathcal Q}_+$ множество кубов
нашего кубильяжа  ${\mathcal Q}$, расположенных ЗА перегородкой ${\mathcal
P}={\mathcal P}_n$. После этого `сдинем' весь этот  набор кубов ${\mathcal
Q}_+$ на вектор $-v_n$. В результате это множество  передвинется вплотную к
передней (видимой) границе перегородки  ${\mathcal P}$ и      образует кубильяж
зонотопа $Z(n-1,d)$. Если мы добавим к этому  зонотопу произведение (сумму) его
невидимой границы на отрезок  $[0,v_n]$ (иными словами, сделаем эспансию этого
кубильяжа вдоль `задней'  мембраны зонотопа $Z(n-1,k)$ по цвету $n$), мы
получим  исходной зонотоп $Z$ и новый его кубильяж, в котором новая
перегородка ${\mathcal P}'$ цвета $n$ примыкает к невидимой  стороне $Z$. Такую
`крупную'  перестройку кубильяжа можно было бы назвать `большим  флипом', или
\emph{обвалом}, см. рис. 14.

 \begin{figure}[htb]
\unitlength=.6mm
\linethickness{0.4pt}
\begin{picture}(150.00,75.00)(-46,0)

\bezier{328}(45.00,5.00)(0.00,25.00)(35.00,60.00)
\put(45.00,5.00){\vector(2,1){10.00}}
\put(35.00,60.00){\vector(2,1){10.00}}
\bezier{120}(35.00,60.00)(45.00,50.00)(40.00,35.00)
\bezier{156}(40.00,35.00)(30.00,20.00)(45.00,5.00)

\put(27,33){$\mathcal{Q}_-$}
\put(43,33){$\mathcal{P}$}
\put(55,30){$\mathcal{Q}_+$}

\bezier{328}(110.00,5.00)(65.00,25.00)(100.00,60.00)
\put(110.00,5.00){\vector(2,1){10.00}}
\put(100.00,60.00){\vector(2,1){10.00}}
\bezier{120}(100.00,60.00)(110.00,50.00)(105.00,35.00)
\bezier{156}(105.00,35.00)(95.00,20.00)(110.00,5.00)

\put(40.00,10.00){\vector(2,1){10.00}}
\put(37.00,16.00){\vector(2,1){10.00}}
\put(36.00,22.00){\vector(2,1){10.00}}
\put(37.00,29.00){\vector(2,1){10.00}}
\put(40.00,35.00){\vector(2,1){10.00}}
\put(42.00,42.00){\vector(2,1){10.00}}
\put(42.00,49.00){\vector(2,1){10.00}}
\put(39.00,55.00){\vector(2,1){10.00}}

\bezier{136}(45.00,65.00)(55.00,53.00)(50.00,40.00)
\bezier{148}(50.00,40.00)(40.00,26.00)(55.00,10.00)
\bezier{356}(45.00,65.00)(90.00,40.00)(55.00,10.00)

\put(92,33){$\mathcal{Q}_-$}
\put(106,28){$\mathcal{Q}_+\!\!-\!\! v_n$}
\put(129,33){$\mathcal{P}'$}

\put(49,4){$v_n$} \put(114,4){$v_n$}
\put(73,33){$\longrightarrow$}

\put(115.00,10.00){\vector(2,1){10.00}}
\put(120.00,16.00){\vector(2,1){10.00}}
\put(124.00,22.00){\vector(2,1){10.00}}
\put(126.00,29.00){\vector(2,1){10.00}}
\put(125.00,35.00){\vector(2,1){10.00}}
\put(121.00,42.00){\vector(2,1){10.00}}
\put(116.00,49.00){\vector(2,1){10.00}}
\put(108.00,55.00){\vector(2,1){10.00}}

\bezier{356}(100.00,60.00)(145.00,34.00)(110.00,5.00)
\bezier{356}(110.00,65.00)(155.00,40.00)(120.00,10.00)
\end{picture}
 \caption{обвал}
 \label{fig:14}
  \end{figure}

В результате такого обвала мы  получаем кубильяж $\widetilde{{\mathcal Q}}$
зонотопа $\widetilde{Z}=Z(n-1,d)$. С ним можно проделать  такую же операцию
подъема перегородки цвета $n-1$, и так далее. Эту  последовательность из $n$
обвалов можно назвать {\em стандартизацией} кубильяжа ${\mathcal Q}$, потому
что после каждого шага (обвала) кубильяж  становится `все более стандартным', а
в конце процесса --  полностью стандартным. Главное преимущество стандартизации
-- ее  каноничность; если понижающие флипы мы могли делать в разных  местах и в
разной последовательности, то обвалы определены  однозначно.

Симметрично можно делать `антистандартизацию'  кубильяжа,  прижимая (на первом
шаге) перегородку ${\mathcal P}$ к видимой  границе зонотопа  $Z(n,d)$ и
продолжая в таком же духе; в результате через $n$ шагов мы  получим
антистандартный кубильяж.

Уже говорилось, как с помощью флипов вводится порядок $\le$ на множестве
$\textbf{Q}(n,d)$ кубильяжей зонотопа $Z(n,d)$. А именно, мы полагаем $Q\le
Q'$, если от $Q$ до $Q'$ можно добраться с помощью последовательности
повышающих флипов. Циклов при этом быть не может, так что это отношение $\le$
действительно является отношением частичного порядка на множестве
$\textbf{Q}(n,d)$. Фактически, именно этот посет Манин и Шехтман называли
\emph{высшим порядком Брюа} и обозначали $B(n,d)$. Теорема 1 говорит, что этот
посет обладает минимальным (стандартный кубильяж) и максимальным
(антистандартный кубильяж) элементами. Подобные посеты часто оказываются
решетками. И действительно, первый посет Брюа $B(n,1)$ (т. н. \emph{слабый
порядок Брюа}) является решеткой. При малых $n$ посет $B(n,2)$ тоже решетка
(см. рисунок посета $B(5,2)$ в \cite{FZ}), но уже $B(6,2)$ не является
решеткой, \cite{Z}. Более того, этот посет в общем случае не есть посет вершин
многогранника, \cite{FZ}. И все же это градуированный (ранговый) посет.

           \section{Мембраны в зонотопе}\label{sec:membr_in_zonotop}

До сих пор мы занимались мембранами в кубильяже.  Теперь мы введем более общее
понятие  {\em мембраны в зонотопе} $Z=Z(n,k)$. Это снова $(d-1)$-мерный
кубический комплекс ${\mathcal M}$, расположенный в $Z(n,k)$, причем выполнены
два свойства (заимствованные у мембран в кубильяжах):

1) ${\mathcal M}$ (точнее, тело ${\mathcal M}$) проектируется биективно на $Z'=Z(n,d-1)$,

2) ребра $M$ конгруэнтны векторам $v_i$ (впрочем, при $d>2$ это выводится из свойства 1) примерно как в Лемме 1).

Проекции кубов из ${\mathcal M}$ дают $(d-1)$-мерный кубильяж  зонотопа $Z'$,
который мы снова обозначаем как ${\mathcal Q}'=\pi (M)$. Обратно, если
${\mathcal Q}'$ --  кубильяж зонотопа $Z'$, мы можем (единственным способом)
построить  мембрану ${\mathcal M}$, для которой ${\mathcal Q}'=\pi ({\mathcal
M}))$. Делается это так. Для любой  вершины $v'$ кубильяжа ${\mathcal Q}'$ мы
задаем вершину $v$ как $\sum _{i\in sp(v')}v_i$. После этого организуем эти
вершины в кубильяж, повторяя организацию в кубы вершин $v'$ кубильяжа
${\mathcal Q}'$. Или чуть иначе: мы строим мембрану ${\mathcal M}$ (точнее --
тело этой мембраны) как      некоторое естественное (кусочно линейное) сечение
зонотопа $Z'$  (сечение относительно проекции $\pi $). Как уже говорилось,
каждая  точки кубильяжа ${\mathcal Q}'$ однозначно записывается через  вектора
$v'_i$; остается переписать это через $v_i$. Тем самым,

      \

{\em  Множество ${\bf M}(n,d)$ мембран в зонотопе $Z=Z(n,d)$  отождествляется с
множеством ${\bf Q}(n,d-1)$ кубильяжей зонотопа $Z(n,d-1)$.}

      \

Поясним на примере $d=3$. Кубильяж зоногона $Z'=Z(n,2)$ - это ромбической
тайлинг, чисто двумерная  фигура. Переход к мембране в $Z=Z(n,3)$ делает ее
более рельефной: это  как переход от плоского чертежа к трехмерному макету.
Флипы,  выглядящие как искусственные перетасовки ромбов, превращаются в  более
наглядные изгибания пленок-мембран. Однако пока наши  мембраны `висят в
воздухе'. Следующее утверждение исправляет этот дефект.

 \begin{theorem} \label{tm:2}
 Для любой мембраны ${\mathcal M}$ зонотопа $Z=Z(n,d)$
существует кубильяж ${\mathcal Q}$ этого зонотопа $Z$, для которого она
является мембраной.
 \end{theorem}

           Иными словами, мембрана ${\mathcal M}$ вписывается в некоторый кубильяж ${\mathcal Q}$.\medskip

\begin{proof} Обозначим через $Z_-({\mathcal M})$ область в $Z$ перед  мембраной ${\mathcal M}$.
Воспользуемся Теоремой 1. Делая с мембраной ${\mathcal M}$ понижающие флопы (точнее, мы делаем флипы с кубильяжем $\pi ({\mathcal M}$), а потом поднимаем их на мембраны), мы получаем серию мембран, идущую от ${\mathcal M}$  до `передней' мембраны, то есть до видимой части границы $Z$. Тем самым мы получаем кубильяж области $Z_-({\mathcal M})$.

Аналогично можно поступить с областью $Z_+({\mathcal M})$ за мембраной; вместе
они дадут кубильяж: в который вписана мембрана ${\mathcal M}$.  \end{proof}

Более  того, если делать каноническую стандартизацию кубильяжа $\pi ({\mathcal
M})$, мы получим канонический (или `стандартный')  кубильяж области
$Z_-({\mathcal M})$ в зоногоне Z. `Стандартность' этого кубильяжа проявляется в
том, что в нем нельзя сделать понижающий флип.

\begin{figure}[h]
\begin{center}
\includegraphics[scale=0.75]{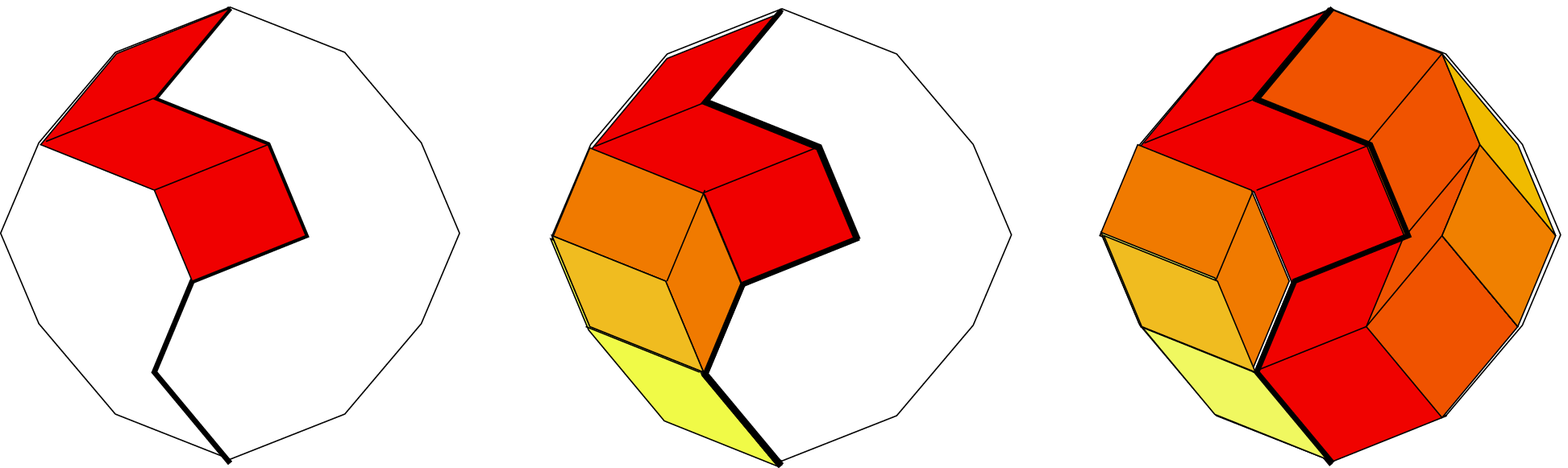}
\end{center}
 \caption{\small Стадии канонизации. Дан ромбический тайлинг и мембрана-змейка
$\mathcal{M}$ в нем. На рис. слева изображен первый шаг стандартизации области
перед мембраной -- построение части перегородки цвета 6. На рис. в центре уже
завершено построение всех перегородок в области перед мембраной, то есть
показана стандартизация этой области. На рис. справа показана
(анти)стандартизация области за мембраной.}
 \label{fig:15}
  \end{figure}

Симметрично можно поступить с областью $Z_+({\mathcal M})$, расположенной в $Z$
после мембраны ${\mathcal M}$. Мы получим каноническое (`антистандартное')
заполнение этой области. Объединяя эти два кубильяжа, мы получаем  кубильяж
${\mathcal Q}$ уже всего зонотопа $Z$, в котором  ${\mathcal M}$ является
мембраной. Мы называем этот стандартно-антистандартный кубильяж зонотопа $Z$
{\em каноническим  расширением} мембраны ${\mathcal M}$ (или {\em каноническим
подъемом} кубильяжа $\pi ({\mathcal M})$).\medskip

{\bf Замечание.} {Рассуждая как при доказательстве Теоремы 1, можно показать,
что если кубильяж области $Z_-({\mathcal M})$ отличен от стандартного (см.
выше), то в нем можно сделать понижающий флип. Это показывает, что кубильяжи
области $Z_-({\mathcal M})$ (как и кубильяжи всего зонотопа) связаны флипами.
Аналогичное замечание верно и для кубильяж области $Z_+({\mathcal
M})$.}\medskip

Мембраны служат как бы средством, связывающим кубильяжи  соседних размерностей.
Кубильяж ${\mathcal Q}$ представляется `сетью'  мембран, а тем самым `сетью'
кубильяжей на 1 меньшей размерности.  Напомним, что ${\bf Q}(n,d)$ -- это
множество кубильяжей  зонотопа $Z(n,d)$. Скажем, что кубильяжи ${\mathcal Q}$
из ${\bf  Q}(n,d)$ и ${\mathcal Q}'$ из ${\bf Q}(n,d-1)$ {\em согласованы},
если $Q'$ поднимается до мембраны ${\mathcal M}$ в ${\mathcal Q}$ (так что
${\mathcal Q}'=\pi ({\mathcal M})$). Отношение согласованности можно
рассматривать как соответствие $c$ из ${\bf Q}(n,d)$ в ${\bf Q}(n,d-1)$.
Описанный выше канонический подъем кубильяжей задает сечение $can: {\bf
Q}(n,d-1) \to {\bf Q}(n,d)$ этого соответствия $c$.

Обилие связей между кубильяжами побуждает взглянуть на них с  категорной точки
зрения. Мы поговорим об этом в Дополнении 1 про  поликатегории.

\

      \

                        \textsc{Часть вторая, комбинаторная}

\

Под комбинаторным подходом к кубильяжам мы понимаем изучение систем подмножеств
базисного множества $[n]$, генерируемых кубильяжами. Мы уже встречались с этим,
когда рассматривали типы кубов кубильяжа. Однако если  брать только типы кубов,
то любой кубильяж  дает все множество-грассманиан ${[n] \choose d}=Gr([n],d)$.
Чтобы отразить специфику конкретного кубильяжа $\mathcal Q $, нужно перенести
на $Gr([n],d)$ естественный порядок $\preceq _\mathcal Q$. Это перебрасывает
мост между геометрическим подходом  и подходом, который использовали Манин с
Шехтманом \cite{MSch-1, MSch-3} и Циглер \cite{Z}. Напомним, что Манин и
Шехтман, создатели высших  порядков Брюа, работали именно в терминах порядков
на $Gr([n],d)$.

Другой выход в комбинаторику открывает понятие спектра. Если  мы соберем в одну
систему $Sp(\mathcal Q )$ спектры всех вершин кубильяжа $\mathcal Q $,  то
получим интересную систему подмножеств в $[n]$. Главное свойство таких систем
--  свойство разделенности, впервые обнаруженное Леклерком и Зелевинским
\cite{LZ} в случае $d=2$ и обобщенное и исследованное затем Галашиним и
Постниковым \cite{Ga, GaP}.

В процессе работы нам постоянно придется иметь дело с подмножествами и
системами подмножеств в $[n]$. Само базисное  множество $[n]$ можно понимать
как аналог $n$-мерного векторного  пространства (над произвольным полем, или
над `полем' ${\bf F}_1$ из одного  элемента). С этой точки зрения подмножество
в $[n]$ размера $k$ надо  понимать как аналог векторного подпространства
размерности $k$.  Из-за этого множество всех подмножеств размера $k$ мы
называем  `дискретным грассманианом' и обозначаем $Gr([n],k)$ (вместо обычного
обозначения ${[n] \choose k}$). Аналогом полного флага подпространств служит
неуплотняемая цепь подмножеств в $[n]$, то есть фактически линейный  порядок на
$[n]$. Напомним, что мы рассматриваем $[n]$ как множество (цепь) с естественным
порядком $(1<2<...<n)$, что означает, что мы как бы фиксируем некоторый полный
флаг.

Для удобства идентификации элементы базисного множества $[n]$ мы называем
\emph{цветами} и используем обычно буквы $i,j$ и т.д.; подмножества в $[n]$
обозначаются буквами $X,Y$ и т.д. и называются  часто просто множествами. Для
краткости множество вида $X\cup \{i\}$ обозначается обычно как $Xi$.
Подмножества в $2^{[n]}$ называются {\em системами множеств} и обозначаются
рукописными буквами (типа $\mathcal S$ или $\mathcal X$).

           \section{Допустимые порядки}\label{sec:admiss_orders}

Если $Q$ -- куб некоторого кубильяжа $\mathcal Q $ зонотопа $Z(n,d)$, то он
определяет (и определяется) двумя множествами: спектром нижней (корневой)
вершины $b(Q)$, $sp(b(Q))\subset [n]$, и типом $\tau (Q)\subset [n]$. Эти два
подмножества не пересекаются, и второе имеет размер $d$. {Фактически это уже
упоминавшаяся кодировка знаковыми векторами; $Q \mapsto
\frac{\tau(Q)}{sp(b(Q))}$.} Согласно Предложению 1, отображение типа $\tau $
задает биекцию множества $\mathcal Q $ всех кубов кубильяжа $\mathcal Q $ с
грассманианом  $Gr([n],d)$. Если  перенести с помощью этой биекции $\tau $
естественный порядок $\preceq $ на $\mathcal Q $ (см. раздел 9), мы получим
некоторый порядок $\preceq _\mathcal Q$ на грассманиане $Gr([n],d)$.

 \begin{prop} \label{pr:12}
Порядок $\preceq _\mathcal Q $  определяет кубильяж $\mathcal Q $.
 \end{prop}

           Иначе говоря, он позволяет однозначно восстановить кубильяж.\medskip

  \begin{proof}
Пусть есть два кубильяжа $\mathcal Q _1$ и $\mathcal Q _2$ зонотопа $Z=Z(n,d)$
с одним и тем же порядком на  $Gr([n],d)$. И пусть $\mathcal S $ --
(порядковый) идеал в $Gr([n],d)$; покажем  индукцией по размеру $\mathcal S $,
что кубы (в $\mathcal Q _1$ и $\mathcal Q _2)$, соответствующие $\mathcal S $
(то есть $\tau ^{-1}(\mathcal S )$), расположены в $Z$ одинаково. Пусть $K\in
Gr([n],d)$ -- максимальный элемент в $\mathcal S $, и  $\mathcal S '=\mathcal S
-\{K\}$. По индуктивному предположению кубы, соотвествующие $\mathcal S '$,
расположены в $Z$ одинаково. И лежат в `передней' части относительно некоторой
мембраны $\mathcal M '$. Тогда куб $Q_1$  (в $\mathcal Q _1$ с типом $K$)
примыкает всей своей видимой границей  к определенному участку (капсиду) этой
мембраны $\mathcal M '$, точно так же, как и куб $Q_2$ (из $\mathcal Q _2$ с
типом $K$). То есть они примыкают  к одному и тому же участку $\mathcal M '$ и
потому совпадают.
      \end{proof}

Можно пояснить это и чуть иначе. Пусть $Q_1$ (соотв., $Q_2$) -- кубы в
кубильяжах $\mathcal Q _1$ (соотв., $\mathcal Q _2$) с одним и тем же типом
(скажем, $K$). И пусть $b_1=b(Q_1)$ и $b_2=b(Q_2)$ -- их корневые вершины.
Достаточно показать, что $b_1=b_2$. Возьмем какую-нибудь видимую фасету $F_1$
куба $Q_1$ (с типом $K-i$ для некоторого $i\in K$), и рассмотрим часть тоннеля
типа $K-i$, идущую от $F_1$ и передней границе зонотопа  $Z$. Мы не знаем, как
проходит этот тоннель, но мы знаем, что он составлен из кубов $Q$, которые a)
лежат в идеале $\mathcal S$, и b) содержат $K-i$ в своем типе $\tau (Q)$. Этот
тоннель выходит к передней части границы $Z$ по единственной фасете
$\widetilde{F}$ типа $K-i$. Обозначим его корневую вершину как $\widetilde{b}$.
Тогда мы получаем, что $b_1= \widetilde{b}+\sum_j (\pm) v_j$, где  сумма
берется по тем $j$, что $K-i+j$ принадлежит $\mathcal S $. Знаки же
определяются как знаки определителя $\det(...v_k...v_j)$, где $k$
пробегает $K-i$. Так как это выражение зависит от $K-i$ и от $\mathcal S $, но
не от $\mathcal Q $, то $b_2=b_1$.\medskip

Полученное утверждение поднимает вопрос о том, какими      условиями выделяются
порядки вида $\preceq _\mathcal Q $. Одно такое условие подсказывает
Предложение 9 раздела 10. Назовем {\em родителем} подмножество $K$ в $[n]$
размера $d+1$, а его {\em семьей} -- систему $\mathcal F (K)$ подмножеств
размера $d$ в $K$. (Иначе говоря, это образ $Gr(K,d)$ в  $Gr([n],d)$ при
естественном вложении. Манин и Шехтман называли это \emph{пакетом}.) У семьи
есть два выделенных линейных порядка, а именно, лексикографический $\preceq
_{lex}$ и противоположный к нему  антилексикографический $\preceq_{alex}$. Если
$K=\{i_1<i_2<...<i_{d+1}\}$, то лексикографический порядок $\preceq _{lex}$ на
$\mathcal M (k)$ имеет вид
                                                    $$
                          K-i_{d+1}<K-i_d<...<K-i_1,
                                                      $$
а антилексикографический -- противоположный. \medskip

{\bf Определение.} Частичный порядок $\preceq $ на $Gr([n],d)$ называется {\em
допустимым}, если его ограничение на любую семью $\mathcal F(K)$ (где $K$ --
произвольное множество размера $d+1$) есть       лексикографический или
антилексикографический порядок.\medskip

Разумеется, любое усиление допустимого порядка также допустимое. В частности,
любое линейное расширение. Манин и Шехтман работали именно с линейными
допустимыми порядками.

В этих терминах Предложение~\ref{pr:9} утверждает, что порядок на $Gr([n],d)$,
индуцированный некоторым кубильяжем, допустим. Оказывается, это условие не
только необходимое, но и в каком-то смысле достаточное.

\begin{theorem} \label{tm:3}  Пусть $\preceq $ -- допустимый порядок на  $Gr([n],d)$.
Тогда  существует (и притом единственный, см. Предложение~\ref{pr:12}) кубильяж
$\mathcal Q $, такой что $\preceq $ -- усиление $\preceq _\mathcal Q$.
  \end{theorem}

Чтобы сделать идею доказательства более понятной, представим,  что порядок
$\preceq $ уже происходит из кубильяжа $\mathcal Q $. Пусть $\mathcal P $ -
перегородка цвета $n$; произведем редукцию цвета $n$. В результате перегородка
сожмется на мембрану $\mathcal M $, а кубильяж $\mathcal Q $  редуцируется в
кубильяж      $\mathcal Q '=\mathcal Q _{-n}$ зонотопа $Z'=Z(n-1,d)$. $\mathcal
Q '$ как множество кубов естественно      вкладывается в $\mathcal Q $. И
ограничение естественного порядка кубильяжа $\mathcal Q $ на $\mathcal Q '$
сильнее естественного порядка на $\mathcal Q '$ (Предложение 7 из раздела 9).
Так что  ограничение $\preceq $      на $Gr([n-1],d)$ -- допустимый порядок, и
по индукции мы можем  восстановить $\mathcal Q '$. Остается восстановить еще и
мембрану $\mathcal M $. Но она задается множеством (типов) кубов, расположенных
в $\mathcal Q '$ до мембраны. Или, что то же самое, множеством кубов,
расположенных в $\mathcal Q $ до перегородки. Пусть $Q$ -- такой куб, и
$T\subset [n-1]$ -- его тип. Рассмотрим родителя $K=Tn$ и его семью как набор
кубов в $\mathcal Q$. Эта семья  `начинается' с $Q$ и затем идут кубы из
перегородки $\mathcal P $, так как их типы содержат цвет $n$. Из этого
следует, что $Q$ -- минимальный член этой семьи. Напротив, если куб $Q$ взят за
перегородкой, то $Q$ -- максимальный член семьи. Таким  образом, кубы (в
$\mathcal Q '$) до мембраны $\mathcal M $ -- это те кубы,  тип которых $T$
минимален в семье с родителем $Tn$. (Все это пояснение можно рассматривать как
третье доказательство Предложения 12.)

Теперь перейдем к доказательству, которое представляет собой обращение
пре\-дыдущих аргументов. Мы начинаем с допустимого порядка $\preceq $ на
$Gr([n],d)$. Ограничиваем его на $Gr([n-1],d)\subset Gr([n],d)$; очевидно, что
это снова допустимый порядок на $Gr([n-1],d)$. По  индукции он происходит из
некоторого кубильяжа $\mathcal Q '$ зонотопа $Z'=Z(n-1,d)$. Теперь образуем
подмножество $\mathcal T $ в $Gr([n-1],d)$, состоящее из таких $T$, что
ограничение $\preceq$ на семью  $\mathcal F(Tn)$ родителя $Tn$ --
лексикографический порядок.

Мы утверждаем, что $\mathcal T $ -- \emph{идеал относительно естественного
порядка $\preceq '$ на} $\mathcal Q '$. В самом деле, пусть $Q\in \mathcal T $
и $R\prec Q$, то  есть $R$ непосредственно предшестует кубу $Q$ в кубильяже
$\mathcal Q '$ (см. раздел 9). Нужно показать, что $R$ тоже принадлежит
$\mathcal T $. Предположим  противное, что ограничение $\preceq $ на семью
родителя $\tau (R)n$ --  антилексикографический порядок. В частности, $\tau
(R)$ -- максимальный  элемент в семье $\mathcal F (\tau (R)n)$. В свою очередь,
$\tau (Q)$ -- минимальный элемент в семье $\mathcal F (\tau (Q)n)$. Кубы $R$ и
$Q$ имеют общую фасету $F$, так что тип $F$, $\tau (F)$, содержится как в $\tau
(R)$, так и в $\tau (Q)$. Поэтому множество  $\tau (F)n$ лежит как в семье
$\mathcal F(\tau (R)n)$, так и в семье $\mathcal F (\tau (Q)n)$. Отсюда
следует, что $\tau (Q)\preceq \tau (F)n\preceq \tau (R)$ и по транзитивности
$\tau (Q)\preceq \tau (R)$. Но это  противоречит тому, что $R\prec Q$ и,
следовательно, $\tau (R)\prec \tau (Q)$.

Итак, $\mathcal T $ -- идеал в $\mathcal Q '$. Этот идеал определяет мембрану
$\mathcal M $ в $\mathcal Q '$; делая экспансию цвета $n$ в этой мембране, мы
получаем кубильяж $\mathcal Q $ (см. Предложение 4 раздела 6). Остается
убедиться, что исходное отношение $\preceq $ сильнее, чем отношение $\preceq
_\mathcal Q $. То есть если куб $R$ непосредственно предшествовал кубу $S$ в
кубильяже  $\mathcal Q $, $R\prec S$, то $\tau (R)\preceq \tau (S)$.

Это очевидно, если $R$ и $S$ лежат в $\mathcal Q '$. Это почти очевидно, если
один из кубов лежит в перегородке $\mathcal P $ (полученной экспансией
мембраны $\mathcal M $), а другой -- нет. Пусть, к примеру, $S$ лежит в
перегородке, а $R$ -- нет. Тогда куб $R$ расположен до перегородки. По
построению идеала $\mathcal T $ семья $\mathcal F (\tau (R) \cup \tau(S))$
`лексикографическая', и  $\tau (R)$ -- минимальный член этой семьи, откуда
$\tau (R)\preceq  \tau (S)$.

Осталось рассмотреть случай, когда $R$ и $S$ лежат в перегородке. Делая
редукцию всех цветов, не вошедших в $\tau (R)\cup \tau (S)$, мы можем считать,
что наш кубильяж $\mathcal Q $ -- один из двух кубильяжей капсида, где
утверждение очевидно в силу результатов раздела 10. Это завершает
доказательство теоремы 3.  $\Box$\medskip

Теорема 3 перебрасывает мостик между геометрическим подходом (кубильяжами) и
подходом Манина-Шехтмана, основанным на рассмотрении допустимых порядков на
грассманианах $Gr([n],d)$. Существование такой связи было анонсировано в работе
Воеводского и Капранова \cite{VK} и более обстоятельно изложено в \cite[Theorem
2.1]{T}.

Некоторое неудобство описания кубильяжей через допустимые  порядки состоит в
том, что разные порядки могут приводить к одному  кубильяжу. Мы уже отмечали,
что если $\preceq $ -- допустимый порядок, то  любое его усиление $\preceq '$
тоже допустимый порядок и задает тот же  кубильяж, что и $\preceq $. Чтобы
восстановить однозначность, можно  работать с минимальными допустимыми
порядками. Конечно, это в  точности порядки, порожденные (в смысле
транзитивного  замыкания) ограничениями на всевозможные семьи в $Gr([n],d)$.
Такие минимальные допустимые порядки задаются указанием для каждого родителя
(множества $K$ размера $d+1$), лексикографически или антилексикографически
упорядочена семья $\mathcal F (K)$ этого родителя. Иначе говоря, отображением
$\sigma :Gr([n],d+1) \to  \{+,-\}$, где + соответствует лексикографии. Конечно,
это назначение $\sigma $ не может быть произвольным; надо, чтобы в  результате
получалось ацикличное отношение (ср. с Предложением 6). Вопрос о том, какие
назначения дают ацикличность (и, следовательно, приводят к кубильяжам) мы
обсудим в следующем разделе.

           \section{Инверсии}\label{sec:invertions}

Раздел 13 позволяет рассматривать кубильяжи зонотопа $Z(n,d-1)$ как мембраны в
зонотопе $Z(n,d)$ на единицу большей размерности. Это  открывает еще один
способ задания кубильяжей.

Пусть $\mathcal Q'$ -- кубильяж зонотопа $Z'=Z(n,d-1)$. Реализуем его как
мембрану $\mathcal M =\mathcal M (\mathcal Q')$ в зонотопе  $Z=Z(n,d)$.
Согласно Теореме 2, существует (и как правило, не один) кубильяж $\mathcal Q $
зонотопа $Z$, в который вписана мембрана $\mathcal M $. Эта мембрана делит кубы
кубильяжа  $\mathcal Q $ на  расположенные до $\mathcal M $ и после $\mathcal M
$. Кубы, расположенные до мембраны  $\mathcal M $, соберем в множество
$\mathcal Q _-(\mathcal M )$. Это множество зависит, конечно, от кубильяжа
$\mathcal Q $, содержащего мембрану $\mathcal M $. Однако множество
\emph{типов} кубов из $\mathcal Q _-(\mathcal M)$  уже зависит только от
мембраны $\mathcal M$, но не от $\mathcal{Q}$. Это видно из следующей
леммы.\medskip

 \begin{lemma} \label{lm:3}
Пусть $Q$ -- куб кубильяжа $\mathcal Q$, расположенный до (соотв., после)
мембраны $\mathcal M$, и $K=\{k_1<k_2<...<k_d\}$ -- тип куба $Q$. Тогда
существует вершина $v$ мембраны $\mathcal M$, такая что $K\cap sp(v)=\{k_d,
k_{d-2},...\}$ (соотв., $K\cap sp(v)=\{k_{d-1}, k_{d-3},...\}$).
 \end{lemma}

\begin{proof}
Произведем редукцию всех цветов, отличных от цветов $K$. В результвте мы
получим зонотоп $Z'=Z(K,d)$ (состоящий из единственного куба $Q=Z'$) и мембрану
$\mathcal M'$ в $Z'$. Пусть $Q$ располагался до мембраны $\mathcal M$; тогда
куб $Q=Z'$ тоже располагается до мембраны      $\mathcal M'$ , то есть мембрана
$\mathcal M'$ -- это невидимая часть границы  $Q$. Пусть теперь $v'$ -- это
`голова' $h(Q)$ куба $Q$ (см. раздел 7) ; как мы знаем, ее спектр равен $\{k_d,
k_{d-2},...\}$. Пусть $v$ -- `прообраз' $v'$ в мембране $\mathcal M$, то есть
та вершина мембраны $\mathcal M$, которая редуцировалась в $v'$. Ясно, что
$sp(v)$ отличается от $sp(v')$  только цветами, не входящими в $K$, откуда
$sp(v)\cap K=\{k_d, k_{d-2},...\}$.

Аналогично рассуждаем в случае, когда  $K$ располагается после мембраны.
\end{proof}

\textbf{Замечание.} Менее техническое рассуждение такое. Кубильяжи области в
зонотопе $Z(n,d)$, расположенной перед мембраной  $\mathcal M$, могут быть
получены друг из друга с помощью флипов (см. замечание из раздела 13). А флипы
не меняют типы кубов, но лишь переставляют их.\medskip

Пусть $\mathcal Q$ -- кубильяж зонотопа $Z(n,d)$, $\mathcal M$ --
соответствующая мембрана в зонотопе $Z(n,d+1)$ (так что $\mathcal
Q=\pi(\mathcal M)$), а $\mathcal{Q}'$ -- некоторый кубильяж $Z(n,d+1)$,
содержащий мембрану $\mathcal M$ (такой существует по Теореме 2).\medskip

{\bf Определение.} {\em  Системой инверсий} $Inv(\mathcal Q)$  кубильяжа
$\mathcal Q$ называется система $\tau (\mathcal Q' _-(\mathcal M ))$ в
$Gr([n],d+1)$ (независящая от выбора $\mathcal{Q'}$ в силу Леммы 3). {Система
$Inv(\mathcal Q)$ обозначается также как} $Inv(\mathcal M)$. \medskip

То же рассуждение, как при доказательстве Леммы 3 (с редукцией всех цветов, не
входящих в тип куба $Q$), дает еще одно описание системы инверсий:

\begin{prop} \label{pr:14}
Родитель $T\in Gr([n],d+1)$ принадлежит $Inv(\mathcal Q )$ тогда и  только
тогда, когда семья этого фатера в $Gr([n],d)$ упорядочена отношением
$\preceq_\mathcal Q $ антилексикографически.
 \end{prop}

      Или\medskip

\noindent  $Inv(\mathcal Q )=\{T\in Gr([n],d+1), \text{ ограничение } \preceq
_\mathcal Q  \text{ на } Gr(T,d)  \text{ --  антилексикография}\}$.\medskip

Например, система $Inv$ для стандартного кубильяжа пустая. Напротив, $Inv$ для
антистандартного кубильяжа -- все $Gr([n],d)$. Когда мы делаем повышающий флип
в кубильяже $\mathcal Q $, мы добавляем к домембранным кубам еще один куб. Так
что при этом система $Inv(\mathcal Q )$ увеличивается на  один элемент.

Важное свойство системы инверсий $Inv(\mathcal Q )$ в том, что она позволяет
однозначно восстановить кубильяж $\mathcal Q $.

\begin{prop} \label{pr:15}
Отображение $Inv:\textbf{Q}(n,d)  \to 2^{Gr([n],d+1)}$ инъективно.
 \end{prop}

 \begin{proof}
Из Предложения~\ref{pr:14} следует, что порядок  $\preceq _\mathcal Q $
определяет $\mathcal Q $. Но порядок $\preceq _\mathcal Q $ порождается
порядками на семьях, а  система $Inv(\mathcal Q )$ как раз и показывает, какая
семья упорядочена  антилексикографически.
 \end{proof}

Как в предыдущем разделе, это поднимает вопрос о том, какие системы в
$Gr([n],d)$ являются `инверсными', то есть  происходят (с помощью $Inv$) из
кубильяжей $Z(n,d-1)$. Вопрос был решен Циглером в статье \cite{Z} (хотя
фундамент для решения был заложен в более ранней работе Лас Вернаса \cite{LV}).
Введем следующее\medskip

{\bf Определение.} Множество $\mathcal S $ в $Gr([n],d)$ называется {\em
консистентным}, если для любого  $T\subset [n]$ размера $d+1$ семья  $\mathcal
F (T)$ этого родителя пересекается с $\mathcal S $ по начальному или конечному
отрезку  $\mathcal F (T)$ (семья $\mathcal F (T)$ располагается в
лексикографическом порядке).

 \begin{lemma} \label{lm:4}
Для любой мембраны $\mathcal M $ в  $Z(n,d)$ система $Inv(\mathcal M )$
образует консистентное подмножество в $Gr([n],d)$.
 \end{lemma}

 \begin{proof}
Включим мембрану $\mathcal M $ в кубильяж $\mathcal Q  $ зонотопа $Z=Z(n,d)$, и
пусть $\preceq $      -- естественный порядок на $\mathcal Q $, а $\preceq
_\mathcal Q $ -- соответствующий  (допустимый) порядок на грассманиане
$Gr([n],d)$. Мы знаем из раздела 11, что $Inv(\mathcal M )$ -- это идеал
относительно $\preceq      _\mathcal Q $. С другой стороны, для любой семьи
$\mathcal F (T)$ порядок $\preceq _\mathcal Q $ индуцирует лексикографию или
антилексикографию. В первом случае пересечение $Inv(\mathcal M )$ с $\mathcal F
(T)$ будет      идеалом относительно лексикографического порядка, то есть
начальным отрезком. Во втором -- идеалом относительно противоположного порядка,
то есть конечным отрезком.
 \end{proof}

 \begin{theorem} \label{tm:4}
Обратно, пусть $\mathcal S $ --      консистентная система в      $Gr([n],d)$.
Тогда существует мембрана $\mathcal M $ в зонотопе $Z=Z(n,d)$, для которой
$\mathcal S =Inv(\mathcal M )$.
  \end{theorem}

           Мембрана $\mathcal M $ единственна, согласно Предложению 14.\medskip

\begin{proof}
Рассмотрим  подсистему $\mathcal S '$ в $\mathcal S $, состоящую из множеств
$S\in \mathcal S $, не содержащих цвет $n$. Система $\mathcal S '$ лежит в
$Gr([n-1],d)$ и тоже консистентна. Поэтому по индукции $\mathcal S '$
реализуется некоторой мембраной $\mathcal M '$ в $Z'=Z([n-1],d)$, так что
$\mathcal S '=Inv(\mathcal M ')$. Дополним эту мембрану до  кубильяжа $\mathcal
Q '$ зонотопа $Z'$ (Теорема 2). И затем сделаем экспансию по цвету $n$ мембраны
$\mathcal M '$ в $\mathcal Q '$. Мы получаем кубильяж $\mathcal Q $ зонотопа
$Z$, редукция которого $\mathcal Q _{-n}$  равна $\mathcal Q '$; в частности,
мембрана $\mathcal M '$ `утолщается' до перегородки $\mathcal P $ в $\mathcal Q
$ цвета $n$. Кубы кубильяжа $\mathcal Q $, типы которых  принадлежат $\mathcal
S -\mathcal S '$, лежат  в $\mathcal P$.

Мы утверждаем, что $\mathcal S $ \emph{является стэком (идеалом) в кубильяже}
$\mathcal Q $; в этом случае $\mathcal M $ -- это мембрана, соответствующая
стэку $\mathcal S $. Для этого  нужно проверить, что если $R\prec S$ -- два
куба из перегородки $\mathcal P $, и $S$ принадлежит $\mathcal S $, то и $R$
принадлежит $\mathcal S $. Но кубы $R$ и $S$ соседние, поэтому объединение их
типов, то есть $T=\tau (R)\cup \tau (S)$ имеет размер $d+1$. Рассмотрим семью
$\mathcal F $ этого родителя $T$, расположенную лексикографически. Минимальным
элеменом (`начальным членом') этой семьи служит множество $T-n$; пусть $Q$ --
куб в $\mathcal Q $ с  типом $T-n$. Этот куб не принадлежит перегородке
$\mathcal P $, поэтому он лежит либо до перегородки, либо после нее.

Рассмотрим случай, когда $Q$ лежит до перегородки. Тогда  в естественном
порядке $\preceq $ на $\mathcal Q $ это минимальный элемент, и семья $\mathcal
F $ упорядочена лексикографически. В силу консистентности $\mathcal S$ весь
интервал между $\tau (Q)=T-n$ и $\tau (S)$ принадлежит $\mathcal S$, и,  в
частности, множество $\tau (R)$. Так что $R$ принадлежит $\mathcal S $.

Аналогично рассматривается случай, когда $Q$ лежит после перегородки. В этом
случае $Q$ -- максимальный элемент (относительно ограничения $\preceq $ на
семью $\mathcal F $), и $\tau (Q)$ не принадлежит      $\mathcal S $. Значит
естественный порядок $\preceq $ на $\mathcal F $ антилексикографический, в нем
$R$ идет      после $S$, и снова из консистентности $R$ принадлежит $\mathcal
S$.
\end{proof}

Таким образом, мы имеем два эквивалентных (лучше сказать -- криптоморфных)
описания кубильяжей $Z(n,d)$: как  (минимальных) допустимых порядков на
$Gr(n,d)$ (Манин-Щехтман) и как  консистентных подмножеств в $Gr(n,d+1)$
(Циглер). Имея такие описания, можно переписывать на соответствующий язык
разные  понятия и конструкции с кубильяжами. Например, повышающий флип
соответствует увеличению консистентного множества на один элемент. Так что
отношение $\le $ для кубильяжей (см. раздел 8) согласовано с отношением
включения для консистентных множеств: если $\mathcal Q  \le \mathcal Q '$, то
$Inv(\mathcal Q )\subset Inv(\mathcal Q ')$. Обратное свойство верно, если
$d=1$ (классической свойство слабого порядка Брюа) или $d=2$ (\cite{FW, ER}),
но нарушается  для больших $d$ (Циглер в \cite{Z} привел контрпример при
$d=3$). В следующем разделе мы приведем еще одно описание (инициированное
Леклерком и Зелевинским \cite{LZ} и развитое Галашиным и Постниковым
\cite{GaP}), в  терминах систем $Sp(\mathcal Q )$ спектров вершин кубильяжа.

           \section{Отношение разделенности}\label{sec:separ_relat}

Напомним, что с каждой вершиной $v$ кубильяжа $\mathcal Q $ зонотопа $Z(n,d)$
мы  связывали подмножество $sp(v)$ в $[n]$ -- спектр вершины. Когда $v$
пробегает все вершины $\mathcal Q $, мы получаем систему $Sp(\mathcal  Q )$,
подмножество в $2^{[n]}$. В разделе 4 было показано, что система $Sp(\mathcal Q
)$ однозначно определяет кубильяж $\mathcal Q $. Поэтому естественно встает
вопрос -- какими свойствами обладают такие системы?

Первое свойство -- это ее размер. В разделе 2 было показано, что размер системы
$Sp(\mathcal Q )$ равен ${n \choose \le d}$. Этого, конечно, мало, чтобы
характеризовать такие системы. Второе важное структурное свойство состоит в
том, что любые два множества из  $Sp(\mathcal Q )$ являются
($d-1$)-разделенными.\medskip

{\bf Определение.} Два подмножества $X$ и $Y$ в $[n]$ называются
$r$-\emph{разделенными} ($r$ -- натуральное число), если не существует
возрастающей цепочки (размера $r+2$) $i_0<i_1<...<i_{r+1}$, такой что элементы
с четными номерами лежат в одном из множеств $X-Y$ или $Y-X$, а элементы с
нечетными -- в другом. Это симметричное и рефлексивное отношение обозначим как
$S_r$. Система подмножеств называется  $r$-\emph{разделенной}, если любые два
ее члена   $r$-разделены. \medskip

Иными словами, $X$ и $Y$ $r$-разделены, если множество  $[n]$ можно разбить на
$r+1$ интервал $I_0,...,I_r$ (некоторые могут  быть пустыми), так что $X-Y$
лежит в одной части этих интервалов, а  $Y-X$ -- в другой части.

Такие интервалы удобно задавать `разделителями'  (сепараторами), то есть $r$
точками $c_1\le c_2\le ...\le c_r$ на вещественной прямой, так что интервал
$I_i$ состоит из целых точке $x$, удовлетворяющих неравенствам  $c_i<x<c_{i+1}$
(считая $c_0=0$, $c_r=n+1$).

Отношения $S_r$ слабеют с ростом $r$, $S_0\subset  S_1\subset ...\subset S_n$.
Рассмотрим  несколько первых членов этой серии отношений.\medskip

{\bf Пример 0.} Отношение $XS_0Y$ означает, что  множества $X$ и $Y$  сравнимы
по включению: $X\subset Y$ или $Y\subset X$. Можно также сказать, что отношение
$S_{-1}$ это отношение равенства =. \medskip

{\bf Пример 1.} $S_1$ -- это отношение сильной разделенности, введенное
Леклерком и Зелевинским в \cite{LZ}. Оно означает, что после удаления общей
части либо $X$ располагается левее $Y$ ($X<Y$), либо $Y$ располагается левее
$X$.\medskip

{\bf Пример 2.} $S_2$ -- это отношение хорд-разделенности,  исследованное в
\cite{Ga}. \medskip

Отношения разделенности $S_r$ имеют непосредственное отношение к кубильяжам.
Это показывает

 \begin{prop} \label{pr:16}
Пусть $\mathcal Q $ -- кубильяж зонотопа $Z(n,d)$. Тогда для любых двух вершин
$v$ и $w$ кубильяжа $\mathcal Q $ множества $sp(v)$ и $sp(w)$
$(d-1)$-разделены.
 \end{prop}

Иными словами, система $Sp(\mathcal Q )$ $(d-1)$-разделенная. Например, для
любого ромбического тайлинга $\mathcal T $ система $Sp(\mathcal T)$ 1-разделена
-- факт, замеченный в \cite{LZ}. Или другой, крайний случай. Рассмотрим, как в
разделе 7, куб $Z(d,d)$. У него есть две `выступаюшие' точки --  хвост $t$ и
голова $h$. Спектр головы -- $\{d,d-2,...\}$, спектр хвоста -- $\{d-1,
d-3,...\}$, В совокупности эти элементы, чередуясь, покрывают      все $[n]$.
Поэтому множества $sp(h)$ и $sp(t)$ являются $(d-1)$-разделенными  (но не
$(d-2)$-разделенными). Если же мы `сплющим' (спроектируем) этот куб в капсид
$Z(d,d-1)$, точка $t$ превратится во внутреннюю      (центральную) точку
стандартного кубильяжа зонотопа $Z(d,d-1)$ и, как легко проверить, будет
$(d-2)$-разделена со всеми другими  вершинами этого кубильяжа. Но она  не
$(d-2)$-разделена с проекцией точки $h$, центральной точкой другого,
антистандартного  кубильяжа. \medskip

Доказательство предложения 15. Предположим, что $X=sp(v)$ и $Y=sp(w)$ не
($d-1$)-разделены. Тогда существуют элементы-цвета $i_0<i_1<...<i_d$, такие,
что $i_0,i_2,...$ принадлежат, скажем, $X-Y$, а $i_1,i_3,...$ принадлежат
$Y-X$. Делая редукцию по всем  цветам, отличным от $i_0,...,i_d$, мы получаем
кубильяж капсида      $Z(d+1,d)$. И в нем две точки $v'$ и $w'$ (образы $v$ и
$w$ при редукции) со спектрами $i_0,i_2,...$ и $i_1,i_3,...$. Но, как мы видели
выше, эти точки  лежат в разных кубильяжах, одна -- в стандартном ($v'$ при
нечетном $d$ и $w'$ при четном), а другая -- в антистандартном.     \hfill
$\Box$\medskip

В следующем Предложении мы опишем спектр стандартного (и  антистандартного)
кубильяжей зонотопа $Z(n,d)$.

 \begin{prop} \label{pr:16}
Пусть $v$ --  вершина стандартного кубильяжа зонотопа $Z(n,d)$. Тогда
существует разложение $[n]=I_0\sqcup I_1\sqcup  ... \sqcup I_d$  на $d+1$
последовательный интервал, так что $sp(v)=I_{d-1}\sqcup I_{d-3}\sqcup ... $.
Обратно, любое множество такого вида есть спектр некоторой вершины стандартного
тайлинга.

Для антистандартного кубильяжа спектры вершин имеют вид $I_d\sqcup
I_{d-2}\sqcup ...$.
  \end{prop}

\begin{proof}
Стандартный кубильяж зонотопа $Z(n,d)$ получается из  стандартного же кубильяжа
$Z(n-1,d)$ путем экспансии (добавления)  цвета $n$ вдоль задней (невидимой)
мембраны (см. Пример из раздела  6). Поэтому спектр стандартного кубильяжа
состоит из двух частей  -- первая состоит из множеств вида $I_{d-1}\cup
I_{d-3}\cup ...$, где  интервалы $I_d,I_{d-1},...,I_0$ берутся из $[n-1]$, а
вторая  получается из спектров вершин `задней' мембраны (невидимой части
границы зонотопа $Z(n-1,d)$) добавлением цвета $n$. Но `задняя'  мембрана --
это фактически антистандартный кубильяж зонотопа       $Z(n-1,d-1)$, и спектры
ее вершин имеют вид $J_{d-1}\cup      J_{d-3}\cup ...$, где  интервалы
$J_{d-1},...,J_0$ снова берутся из $[n-1]$. В первом случае  нужно `подправить'
интервал $I_d$, заменив его на интервал  $I_d\cup \{n\}$  (заметим, что $I_d$
либо содержит цвет $n\!-\!1$, либо пуст). Во втором  случае мы добавляем новый
интервал $J_d=\{n\}$.
    \end{proof}

К примеру, для $d=1$ спектр стандартного кубильяжа $Z(n,1)$  состоят из
$(n+1)$-го интервала $[k]$, $k=0,1,...,n$. Соотвественно,  спектры
антистандартного тайлинга состоят из дополнительных      интервалов: $[n],
[n-1..n],..., [n-k..n],...,\emptyset$.

Для $d=2$ спектр стандартного ромбического тайлига состоит из  произвольных
интервалов в $[n]$, тогда как спектр антистандарного  тайлинга -- из
коинтервалов (дополнений к интервалам).

Для $d=3$ спектр стандартного кубильяжа состоят из т.н. `полуторных'
интервалов, то есть множеств вида $[1,i]\cup [j,k]$, где      $0\le i\le j\le
k-1\le n$.

{Как следствие Предложения 16 мы получаем описание спектров \emph{периферийных}
вершин, то есть вершин зонотопа. Это в точности вершины, которые принадлежат
одновременно стандартному и антистандартному кубильяжам. Тут удобно ввести
понятие {\em интервального ранга} подмножества $X \subset [n]$ как минимального
числа $r$, такого что $X$ представляется как объединение интервалов  $I_1\sqcup
... \sqcup I_r$.}\medskip

{\bf Следствие 1.} {\em Множество $X$ является спектром периферийной вершины
зонотопа $Z(n,d)$ тогда и только тогда, когда сумма интервальных рангов $X$ и
дополнительного множества $[n]-X$ не больше $d$.} \hfill $\Box$ \medskip

{В частности, если размер некоторого множества $X$ не больше, чем
$\frac{d-1}{2}$ (то есть $2|X|\le d-1$), то соответствующая точка $v(X)$
является периферийной, то есть вершиной зонотопа $Z(n.d)$. Иными словами,
внутренние (не лежащие на границе зонотопа) целые точки появляются только на
высоте $h\ge d/2$ и выше. Этот факт тесно связан с явлением, обычно упоминаемым
при обсуждении циклических политопов (см., например, \cite[Пример 0.6]{Zbook}).
А именно, что любое множество вершин циклического политопа размера не больше
$\frac{d-1}{2}$ принадлежит некоторой грани этого политопа.}\medskip

{\bf Следствие 2.} {\em  Спектр периферийной вершины $(d-1)$-разделен с любым
подмножеством в $[n]$.} \hfill $\Box$ \medskip

Это утверждение коррелирует с Предложением 3, которое  утверждает, в частности,
что любая `целая' точка вписывается в  некотрый кубильяж.

           \section{Разделенность и кубильяжи}\label{sec:separ_and_cubil}

Мы уже получили, что система $Sp(\mathcal Q )$ для кубильяжа  $\mathcal Q $
зонотопа $Z(n,d)$ является $(d-1)$-разделенной и имеет размер ${n \choose \le
d}$. Теперь пойдем в обратном направлении и покажем, что  система с такими
свойствами получается как спектр некоторого кубильяжа. Для начала покажем, что
число ${n \choose \le d}$ -- это верхняя граница для размера
$(d-1)$-разделенной системы в $[n]$.

 \begin{prop} \label{pr:18}
Пусть $\mathcal S $ -- $(d-1)$-разделенная система подмножеств в $[n]$. Тогда
ее размер не превышает ${n \choose \le d}$.
 \end{prop}

 \begin{proof}
Разобьем $\mathcal S $ на две части:  $\mathcal S _0$, члены которой  не
содержат цвет $n$, и $\mathcal S _1$, члены  которой содержат $n$. Обозначим
через $\mathcal S _2$ систему множеств вида $X-n$, где $X$ пробегает $\mathcal
S _1$. Очевидно, что  размер $\mathcal S _2$ равен размеру $\mathcal S _1$.
Легко понять, что      $\mathcal S _0\cup \mathcal S _2$ --
($d-1$)-разделенная система в $[n-1]$, поэтому (по индукции) ее размер  не
превосходит ${n-1 \choose \le d}$.

Рассмотрим теперь пересечение $\mathcal T =\mathcal S_0\cap \mathcal S_2$. Оно
состоит из множеств $X\subset [n-1]$, которые принадлежат  $\mathcal S $,
причем $Xn$ тоже  принадлежит $\mathcal S $. Мы утверждаем, что {\em система
$\mathcal T $ ($d-2$)-разделенная}. В  самом деле, предположим, что в $\mathcal
T $ есть множества  $X$ и $Y$, которые не $(d-2)$-разделенные. Значит есть
последовательность  $i_0<i_1<...<i_{d-1}$, такая, что члены с четными индексами
принадлежат (скажем) $X-Y$, а с нечетными -- $Y-X$. Но тогда
последовательность $i_0<i_1<...<i_{d-1}<i_d=n$ обладает тем же  свойством либо
для $\widetilde{X}=Xn$ и $Y$, либо для $X$ и $\widetilde{Y}=Yn$. Это
противоречит предположению о ($d-1$)-разделенности $\mathcal S $.

Снова по индуктивному предположению размер $\mathcal T $ не превосходит  ${n-1
\choose \le d-1}$. Отсюда размер $\mathcal S $, равный сумме размеров
$\mathcal S _0$ и $\mathcal S _2$, то есть сумме размеров $\mathcal S_0\cup
\mathcal S _2$ и $\mathcal S _0\cap \mathcal S _2$, не  превосходит ${n-1
\choose \le d}+{n-1 \choose \le d-1}={n \choose \le d}$.
\end{proof}

Ввиду Предложения 17 ($d-1$)-разделенные системы в $[n]$ размера ${n \choose
\le d}$  называают {\em максимальными по размеру}, противопоставляя это {\em
максимальности по включению}. Следующий результат установлен Галашиным и
Постниковым \cite{GaP} в несколько большей общности (они не  ограничивались
кубильяжами циклических зонотопов).

\begin{theorem} \label{tm:5}
Пусть $\mathcal S $ -- $(d-1)$-разделенная система в $[n]$ размера      ${n
\choose \le d}$. Тогда она имеет вид $Sp(\mathcal Q )$ для некоторого
(единственного) кубильяжа зонотопа $Z(n,d)$.
  \end{theorem}

 \begin{proof}
Оно близко к доказательству Предложения 17. Снова разделим $\mathcal S $ на
$\mathcal S _0$ и  $\mathcal S _1$  и   обозначим $\mathcal S _2$ как выше.
Тогда из доказательства Предложения 17  следует, что размер $\mathcal S _0\cup
\mathcal S _2$ равен ${n-1 \choose \le  d}$, а размер пересечения $\mathcal S
_0\cap \mathcal S _2$ равен      ${n-1 \choose \le  d-1}$. По индукции мы можем
считать, что система $\mathcal S_0\cup \mathcal S _2$ реализуется некоторым
кубильяжем $\mathcal Q' $ в $Z(n-1,d)$, а пересечение $\mathcal S _0\cap
\mathcal S _2$ -- кубильяжем $\mathcal Q ''$ зонотопа      $Z(n-1,d-1)$.
Очевидно, что $\mathcal Q ''$       реализуется мембраной $\mathcal M $ в
кубильяже $\mathcal Q' $. Остается взять в качестве $\mathcal Q$      экспансию
по цвету $n$ кубильяжа $\mathcal Q' $ вдоль этой мембраны. На уровне
спектров это сведется к замене всех множеств $X$ из $\mathcal S_2$ на множества
$Xn$. Так что $Sp(\mathcal Q)=\mathcal S _0\cup \mathcal S_1=\mathcal S $.
 \end{proof}

Таким образом, мы получаем еще одно криптоморфное воплощение  кубильяжей. И
снова можно в этих новых терминах пересказывать  разные конструкции.

Приведем еще один интересный пример разделенных систем.

 \begin{prop} \label{pr:19}
Пусть $\mathcal{{M}}_1$   и $\mathcal{{M}}_2$  -- две мембраны в зонотопе
$Z(n,d)$, причем  $Inv(\mathcal{M}_1)\subset~ Inv(\mathcal{M}_2)$. Тогда
система $Sp(\mathcal{M}_1)\cup Sp(\mathcal{M}_2)$ $(d-1)$-разделена.
 \end{prop}

 \begin{proof}
Пусть это не так, и некоторые множества $X\in Sp(\mathcal{M}_1)$  и  $Y\in
Sp(\mathcal{M}_2)$ не $(d-1)$-разделены. Тогда сушествует последовательность
элементов $i_1<i_2<...<i_{d+1}$  из $[n]$, которые поочередно лежат то в $X-Y$,
то в $Y-X$. Рассмотрим случай, когда в $X-Y$ лежат элементы $i_{d+1}, i_{d-1},
...$ (а элементы $i_{d}, i_{d-2}, ...$ лежат в $Y-X$), и образуем
$d$-элементное  множество $K=\{i_1,...,i_d\}$. По Лемме~\ref{lm:3}, примененной
к $X$, мы получаем, что $K$ является инверсным для мембраны $\mathcal{{M}}_1$,
а значит и для мембраны $\mathcal{{M}}_2$. Но по той же Лемме~\ref{lm:3} мы
получаем, что множество $X$ неинверсное для мембраны $\mathcal{{M}}_2$.

Аналогично разбирается случай, когда элементы $i_{d+1}, i_{d-1}, ...$ лежат в
$Y-X$. Однако теперь мы образуем $d$-элементное множество
$K'=\{i_2,...,i_{d+1}\}$. По Лемме~\ref{lm:3} $K'$ неинверсное для мембраны
$\mathcal{{M}}_2$, а значит и для $\mathcal{{M}}_1$. Но по той же лемме $K'$
инверсное для $\mathcal{{M}}_1$.
 \end{proof}

           \section{Чистота и расширяемость}\label{sec:purity}

Предложение 17  поднимает следующий вопрос -- когда  данная
($d-1$)-разделен\-ная система $\mathcal S $ подмножеств в $[n]$ расширяется до
системы максимального размера, то есть до ($d-1$)-разделенной  системы размера
${n \choose \le d}$? Или, в терминах раздела 4, когда соответствующее множество
целых точек зонотопа $Z(n,d)$ вписывается в некоторый кубильяж этого зонотопа.
Если это так, система $\mathcal S $ называется {\em расширяемой} (или, точнее,
$(n,d)$-\emph{расширяемой}). Если любая ($d-1$)-разделенная система расширяема,
то говорят, что для пары $(n,d)$ выполнено свойство \emph{чистоты}.

Например, мы видели в разделе 4, что любая одноэлементная система (то есть
любое подмножество в $[n]$) расширяема. Согласно Теореме 2 из раздела 13 любая
мембрана вписывается в кубильяж. Покажем, что  любая двухэлементная система
расширяема.

 \begin{prop} \label{pr:20}
Пусть $X$ и $Y$ -- два ($d-1$)-разделенных подмножества в $[n]$. Тогда двухэлементная система $\{X,Y\}$ расширяема.
 \end{prop}

            \begin{proof}
Будем доказывать индукцией по $n$. Предположим, что  цвет $i$ не входит в $X$ и
$Y$. Тогда их можно рассматривать как  подмножества в $[n]-i$. По индукции
множества $X$ и $Y$ вписываются в      кубильяж $\mathcal Q $ зонотопа
$Z([n]-i,d)$. Остается сделать экспансию по  цвету $i$ (см. раздел 4)
относительно невидимой (в направлении      $v_i$) части границы этого зонотопа.

Аналогично рассуждаем в том случае, когда цвет $i$ входит как в  $X$, так и в
$Y$. Только теперь надо делать экспансию по цвету $i$  относительно видимой по
этому цвету части границы зонотопа.

Таким образом, все сводится к случаю, когда множества  $X$ и $Y$ дополнительны
друг к другу в $[n]$. Пользуясь интервальным  представлением, мы видим, что $X$
и $Y$ образованы из чередующихся  интервалов в некотором разложении $I_0\sqcup
I_1\sqcup  ... \sqcup  I_r$ цепи $[n]$ на  последовательно идущие непустые
интервалы. В силу ($d-1$)-разделенности $X$ и $Y$ мы имеем $r<d$. Но это
означает, что $X$ и $Y$  (точнее, реализующие их точки $v(X)$ и $v(Y)$)
принадлежат `периферии'  зонотопа $Z(n,d)$ и потому вписываются в любой
кубильяж.
         \end{proof}

Как мы увидим в следующем разделе, это утверждение уже неверно для трех
множеств. Это означает, что свойство чистоты  выполнено не для любых $(n,d)$. А
для каких же?

Тривиально чистота имеет место для $d=1$. Чистота для $d=2$ была  установлена
Леклерком и Зелевинским в семинальной работе \cite{LZ}  (см. также наш обзор
\cite{UMN}). Чистота для $d=3$ установлена в работе      Галашина \cite{Ga}.
Кроме того, по тривиальным причинам чистота  имеет место для $n=d+1$ (капсиды).
Этими случаями и исчерпывается чистота, как показано в \cite{GaP}. В следующем
разделе мы более подробно рассмотрим случай $n=d+2$ при $d\ge  4$ и приведем
простые и  явные примеры нерасширяемых систем. Для распространения таких
результатов на случай $n>d+2$ полезны будут следующие два факта.

\begin{prop} \label{pr:20}
Пусть $\mathcal S \subset 2^{[n]}$. Система $\mathcal S$ $(n,d)$-расширяема
тогда и только тогда, когда она (рассмотренная как система подмножеств в
$[n+1]$) $(n+1,d)$-расширяема.
 \end{prop}

 \begin{proof}
Пусть $\mathcal S $ $(n,d)$-расширяема. Тогда она ($d-1$)-разделенная и
реализуется как подмножество вершин некоторого кубильяжа $\mathcal Q $ зонотопа
$Z=Z(n,d)$. Пусть кубильяж $\mathcal Q '$ зонотопа $Z(n+1,d)$ получен из
$\mathcal Q $ экспансией по цвету $n+1$ относительно невидимой (задней)
мембраны $Z$. Так как $\mathcal Q \subset \mathcal Q '$, то $\mathcal S $
вписывается в $\mathcal Q '$.

Обратно, пусть $\mathcal S $ вписывается в некоторый кубильяж $\mathcal Q '$
зонотопа $Z(n+1,d)$. И пусть $\mathcal P =\mathcal P _{n+1}$ -- перегородка
цвета  $n+1$ в $\mathcal Q '$. Так как  множества из $\mathcal S $ не содержат
цвет $n+1$, эта система вершин лежит (нестрого) до перегородки $\mathcal P $.
Поэтому при редукции  $\mathcal P $ мы получим кубильяж $\mathcal Q $ зонотопа
$Z$, в который $\mathcal S $ вписывается.
\end{proof}

 \begin{prop} \label{pr:21}
Пусть $\mathcal S $ -- $(d-1)$-разделенная система подмножеств в $[n]$, и
$n'=n+1$. Образуем систему $\mathcal S '\subset 2^{[n+1]}$, $\mathcal S
'=\{Xn', X\in \mathcal S \}$. Система $\mathcal S $ $(n,d)$-расширяема тогда и
только тогда, когда система $\mathcal S \cup \mathcal S '$
$(n+1,d+1)$-расширяема.
 \end{prop}

            \begin{proof}
Пусть $\mathcal S $ $(n,d)$-расширяемая, то есть вписывается в  некоторый
кубильяж $\mathcal Q $ зонотопа $Z(n,d)$. Реализуем $\mathcal Q $ как мембрану
$\mathcal M $ в некотором кубильяже $\mathcal Q '$ зонотопа $Z(n,d+1)$, так что
$\mathcal S $ образует часть вершин $\mathcal M $. Пусть, наконец, $\mathcal Q
''$ -- экспансия $\mathcal Q '$ по цвету $n'$ относительно этой мембраны
$\mathcal M $; это некоторый кубильяж  зонотопа    $Z(n+1,d+1)$. При этом
мембрана $\mathcal M $ расширяется до перегородки $\mathcal P $ цвета $n'$ в
$\mathcal Q ''$. Система $\mathcal S $ лежит на одной (видимой) стороне этой
перегородки, тогда  как $\mathcal S '$ лежит на другой (невидимой) стороне
перегородки $\mathcal P $. Откуда видно, что $\mathcal S \cup \mathcal S '$
вписывается в $\mathcal Q ''$.

Обратно, пусть система $\mathcal S \cup \mathcal S '$ (как множество  целых
точек в  зонотопе $Z''=Z(n+1,d+1)$) вписывается в некоторый кубильяж $\mathcal
Q ''$ зонотопа  $Z''$. Рассмотрим перегородку $\mathcal P $ цвета $n'$ в
$\mathcal Q ''$. Ясно, что  $\mathcal S $ располагается на одной стороне этой
перегородки, тогда  как $\mathcal S '$ -- на другой (и получается из $\mathcal
S $ сдвигом на вектор $v_{n'}$). Редукция  этой перегородки приводит к
кубильяжу $\mathcal Q '$ зонотопа $Z'=Z(n,d+1)$; при  этом перегородка
$\mathcal P $ сжимается до мембраны $\mathcal M $, а системы $\mathcal S$ и
$\mathcal S '$ сливаются в одну систему $\mathcal S $ точек на этой мембране.
Остается спроектировать эту мембрану и получить кубильяж $\mathcal Q $ зонотопа
$Z(n,d)$,      содержаций систему $\mathcal S $ как подмножество вершин.
\end{proof}

{\bf Следствие.} {\em  Пусть $\mathcal S $ -- нерасширяемая  система в зонотопе
$Z(d+r,d)$, $r\le r'$ и $d\le d'$. Тогда  существует  нерасширяемая система в
зонотопе $Z(d'+r',d')$.}  \hfill $\Box$\medskip

Мы говорили выше, что любая мембрана вписывается в кубильяж. Покажем, что не
любая пара мембран (даже $(d-1)$-разделенных) вписываются в кубильяж.
Воспользуемся примером, построенным Циглером \cite{Z}. А именно, он построил
две мембраны $\mathcal M_1$ и $\mathcal M_2$ в зонотопе $Z(8,4)$ (он задавал их
в форме консистентных систем), которые обладали двумя свойствами: 1)
$Inv(\mathcal M_1)\subset Inv(\mathcal M_2)$, и 2) эти две мембраны не
вписываются в один кубильяж зонотопа  $Z(8,4)$ (справедливости ради надо
сказать, что Циглер формулировал свойство 2) немного в других терминах). Так
вот, в силу Предложения 19 объединение спектров $\mathcal M_1$ и $\mathcal M_2$
является 3-разделенной системой в $Gr([8],4)$, но не продолжающейся до
максимальной по размеру. Это также дает пример отсутствия чистоты.

           \section{Случай $Z(6,4)$}\label{sec:Z64}

Здесь мы более подробно исследуем вопрос о расширяемых/нерасширяемых системах
точек в $Z(d+2,d)$. Мы уже касались в разделе 11 вопроса о кубильяжах таких
зонотопов и выяснили, что их бывает $2n$ штук, где $n=d+2$. Мы подробно
рассмотрим $Z(6,4)$ (в случае $Z(5,3)$ все 2-разделенные подсистемы
расширяемы). На этом простом примере видны почти все эффекты общего  случая
$Z(d+2,d)$.

При малом $n-d$ почти все целые точки лежат на границе зонотопа $Z(n,d)$.
Например, при $n=d+1$ только две точки не на границе. В случае $n=d+2$
непериферийных точек $2n$. В самом деле, по формуле (1.2) число  периферийных
точек равно $2({d+1 \choose d-1}+...+{d+1 \choose 0})=
2(2^{d+1}-1-(d+1))=2^n-2n$. И так как все периферийные точки ($d-1$)-разделены
со всеми точками, то нам надо разобраться с разделенностью этих $2n$
непериферийных точек-множеств. Здесь мы и начнем считать, что $n=6$, а $d=4$.

В этом случае имеется $52=64-12$ периферийных  точек-множеств. Выпишем явно 12
непериферийных точек-множеств. Как было показано в Следствии 1 Предложения 16,
периферийные точки задаются разбиением отрезка $[n]$ на $d$ последовательно
идущих интервала; после этого надо взять объединение либо `четных' интервалов,
либо `нечетных'. В случае, когда $d=4$, а $n=6$, непериферийные точки-множества
-- это те, которые не укладываются (альтернированно) в четыре интервала, а
должны задействовать 5 или 6 (непустых) интервалов. Понятно, что такие
интервалы либо одноэлементные (и тогда их шесть), либо почти все
одноэлементные, и только один интервал двухэлементный (и тогда интервалов
пять). Выпишем эти 12 множеств как строки сдедующей таблицы.\medskip

$$\begin{array}{cccccc}
  1 &  & 3 &  & 5 &  \\
  1 &  & 3 &  & 5 & 6 \\
  1 &  & 3 &  &  & 6 \\
  1 &  & 3 & 4 &  & 6 \\
  1 &  &  & 4&  & 6 \\
  1 & 2 &  & 4 &  & 6 \\
   & 2 &  & 4 &  & 6 \\
   & 2 &  & 4 &  &  \\
   & 2 &  & 4 & 5 &  \\
   & 2 &  &  & 5 &  \\
   & 2 & 3 &  & 5 &  \\
   &  & 3 &  & 5 &
\end{array}$$

Эти 12 множеств удобно расположить по кругу, как на цифербате  часов, помещая
135 на месте шести часов, а дополнительное  множество 246 на месте двенадцати
часов. См. рис. 16.

\begin{figure}[htb]
\unitlength=.65mm
\special{em:linewidth 0.4pt}
\linethickness{0.4pt}
\begin{picture}(125.00,120.00)(-40,0)
\put(71.00,5.00){\circle{10.00}}
\put(71.00,5.00){\makebox(0,0)[cc]{\small 135}}

\put(95.00,10.00){\circle{10.00}}
\put(95.00,10.00){\makebox(0,0)[cc]{\small 35}}

\put(115.00,30.00){\circle{10.00}} \put(115.00,30.00){\circle{11.50}}
\put(115.00,30.00){\makebox(0,0)[cc]{\small 235}}

\put(120.00,55.00){\circle{10.00}}
\put(120.00,55.00){\makebox(0,0)[cc]{\small 25}}

\put(115.00,80.00){\circle{10.00}}
\put(115.00,80.00){\makebox(0,0)[cc]{\small 245}}

\put(95.00,100.00){\circle{10.00}}
\put(95.00,100.00){\makebox(0,0)[cc]{\small 24}}

\put(70.00,105.00){\circle{10.00}} \put(70.00,105.00){\circle{11.50}}
\put(70.00,105.00){\makebox(0,0)[cc]{\small 246}}

\put(45.00,100.00){\circle{10.00}}
\put(45.00,100.00){\makebox(0,0)[cc]{\small 1246}}

\put(25.00,80.00){\circle{10.00}}
\put(25.00,80.00){\makebox(0,0)[cc]{\small 146}}

\put(20.00,55.00){\circle{10.00}}
\put(20.00,55.00){\makebox(0,0)[cc]{\small 1346}}

\put(25.00,30.00){\circle{10.00}} \put(25.00,30.00){\circle{11.50}}
\put(25.00,30.00){\makebox(0,0)[cc]{\small 136}}

\put(45.00,10.00){\circle{10.00}}
\put(45.00,10.00){\makebox(0,0)[cc]{\small 1356}}
\end{picture}
 \caption{циферблат}
 \label{fig:15}
  \end{figure}

Какие же множества среди этих 12-ти 3-разделенные? Легко  понять, что
неразделенными является множество и его дополнение, а также два соседа этого
дополнения. Так что неразделенными являются  множество и три `противоположных'
к нему. Это можно увидеть также  из того, что подряд идущие 5 множеств на
циферблате вписываются в один кубильяж (и дают все 12 кубильяжей зонотопа
$Z(6,4)$). Например, `цифры' от часа до пяти (а именно,  множества  24, 245,
25, 235 и 35) вписываются в стандартный  кубильяж. Из этого описания
3-неразделенности видно, что три  множества (размещенные как XII часов, IV часа
и VIII часов, то есть  множества 246, 235 и 136) образуют (вместе с периферией)
максимальную по включению 3-разделенную систему из 52+3=55 множеств,  тогда как
максимум по размеру равен 52+5=57. Это показывает, что в  $Z(6,4)$ имеется
система размера 55, которая не (6,4)-расширяемая (и максимальна по включению).

Аналогичные рассуждения показывают, что в $Z(d+2,d)$ (при  $d\ge  4$) имеется
$(d+2,d)$-нерасширяемая система.\footnote{Отсутствие чистоты при $d=7$
фактически утверждается в \cite[Proposition 8.1]{OT}.} Любопытно  отметить, что
в  случае $d=8$ на цифеблате из 20 `часов' можно указать 5 `цифр'  (0,  4, 8,
12 и 16), которые дают максимальную по включению 7-разделенную  подсистему в
этой 20-элементной непериферийной системе. Вместе с  периферией это дает
1009-элементную максимальную по включению 7-разделенную систему подмножеств
множества [10]. В то время как  максимум по размеру равен 1004+9=1013. Разница
равна 1013-1009=4=$d/2$,  что дает основания предположить, что {\em размер
максимальной по   включению ($d-1$)-разделенной системы отличается от максимума
по  размеру не более чем на $d/2$.}

Вместе со Следствием из предыдущего раздела мы получаем теорему
Галашина-Постникова \cite{GaP}, что      \emph{если $d\ge  4$ и $n\ge  d+2$, то
существует $(n,d)$-нерасширяемая система  множеств в $[n]$. }В частности, для
таких $(n,d)$ нет чистоты.

\section{Слабая разделенность}\label{sec:weak}

В разделе 16 и последующих вводилось и изучалось отношение $k$-разделенности в
множестве $2^{[n]}$. Обобщая понятие слабой разделенности Леклерка и
Зелевинского \cite{LZ}, можно предложить понятие слабой $k$-разделенности в
случае нечетного $k$. (Для четного $k$ мы не видим хорошего определения.)
Удобнее, однако, задавать это отношение как усиление отношения
$(k+1)$-разделенности.

Итак, пусть  $k$ -- нечетное натуральное число. Напомним, что подмножества $A$
и $B$ в $[n]$  назывались     $(k+1)$-разделенными: если существует такое
разбиение $[n]$ на $k+2$ последовательных интервала $[n]=I_0\sqcup I_1\sqcup
... \sqcup I_{k+1}$, что множество  $A - B$ располагается в объединении
`четных' интервалов $I_0\sqcup I_2\sqcup ... I_{k+1}$, а $B-A$ -- в объединении
`нечетных' интервалов  $I_1\sqcup ... I_k$  (или наоборот).\medskip

\textbf{Определение.}  Множества  $A$   и $B$   как выше называются
\emph{сильно $(k+1)$-разделен\-ными} (или \emph{слабо $k$-разделенными}), если
размер $A$ не больше, чем размер  $B$. Образно говоря, `окружающее' множество
по размеру не превосходит `окружаемое'.\medskip

Например, множество 15 слабо 1-разделено с множествами 234 и 24, но не с
множеством 3. Отношение  слабой 1-разделенности -- это в точности отношение
слабой разделенности Леклерка-Зелевинского. Если множества $A$   и $B$
$k$-разделены, то они слабо $k$-разделены, что оправдывает употребление эпитета
`слабо'. Если $A$  и $B$  имеют одинаковый размер, слабая   $k$-разделенность
превращается в $(k+1)$-разделенность.

В некоторых отношениях слабая $k$-разделенность `похожа' на $k$-разделенность.
Например: в Предложении 17 утверждается, что размер любой $k$-разделенной
системы подмножеств $[n]$ не превышает ${n \choose \le k+1}$. Покажем, что это
же верно для слабой $k$-разделенности.\medskip

\textbf{Предложение 22}. \emph{Пусть $k$ -- нечетное число. Размер слабо
$k$-разделенной системы подмножеств в $[n]$ не превышает    ${n \choose \le
k+1}$.}\medskip

\begin{proof}
Оно проводится индукцией по $k$; при $k=-1$ утверждение очевидно, так как
сильная 0-разделенность -- это просто равенство. Поэтому размер такой системы
равен 1, как и ${n \choose 0}$. При $k\ge 1$ рассуждение начинается как в
предложении 17. Пусть $\mathcal{W}$  -- слабо $k$-разделенная система в $[n]$;
разделим ее на две части: $\mathcal{W}_0$, члены которой не содержат цвет  $n$,
и $\mathcal{W}_1$, члены которых содержат цвет   $n$. Пусть $\mathcal{W}_2$
состоит из множеств  вида  $X-n$, где  $X$  пробегает $\mathcal{W}_1$. Ясно,
что $\mathcal{W}_0\cup \mathcal{W}_2$ снова слабо $k$-разделенная система (уже
в $[n-1]$), и по индукции ее размер не превосходит ${n-1 \choose \le k+1}$.

Обозначим через  $\mathcal D $  систему множеств $X$, которые сами лежат в
$\mathcal W_0 $, а их `двойник' $Xn$      лежит в $\mathcal W_1$. Мы
утверждаем, что \emph{размер $\mathcal D $ не превосходит  ${n-1 \choose \le
k}$.} Как и в предложении 17, Предложение 22 следует из этой оценки. Следующие
три леммы доказывают эту оценку. \medskip

{\bf Лемма 5.} {\em  Пусть множества  $A$ и $B$  из $\mathcal D $ имеют  один и
тот же размер. Тогда они $(k-1)$-разделены. }\medskip

Доказательство. $(k-1)$-неразделенность означает, что ($A-B$ и $B-A$) нельзя
альтернированно разместить в разбиении $[n-1]$ на $k$ интервалов. То есть что
для их размещения нужно задействовать $k+1$ интервал. Предположим, что такое
размещение имеет схему $ABAB...AB$ (слово из  $k+1$ буквы). В этом случае
сравним $B$  и  $An$. И хотя эти множества помещаются в $k+2$ интервала,
размер окружающего множества $An$ превышает размер окружаемого $B$, что
противоречит тому, что $An$  и $B$ слабо  $k$ -разделены.  Аналогично в случае
схемы $BA...BA$, где мы будем сравнивать множества $A$ и $Bn$.  $\Box$ \medskip

Обозначим через $\mathcal D _i$ подсистему тех множеств из $\mathcal D$,
которые имеют размер $i$. И дополним ее двумя другими  системами подмножеств в
$[n-1]$. Для их определения рассмотрим две вспомогательные  системы $\mathcal S
$ и $\mathcal A $. Чтобы их задать, мы рассмотрим разбиение $[n-1]$ на $k$
интервалов, $[n-1]=I_0\sqcup ...\sqcup I_{k-1}$. Тогда объединение `нечетных'
интервалов $(I_1\sqcup ...\sqcup I_{k-2})$ этого разбиения дает элемент системы
$\mathcal S$, а объединение `четных' интервалов дает элемент системы  $\mathcal
A $. Через $\mathcal S _i$ мы обозначаем подсистему в $\mathcal S $,
образованную множествами размера $i$, и аналогично понимаются $\mathcal A
_i$.\medskip

{\bf Лемма 6.} {\em  Объединение систем $\mathcal S _{n-1}\sqcup ...\sqcup
\mathcal S _{i+1}\sqcup \mathcal D _i\sqcup \mathcal A_{i-1}\sqcup ...\sqcup
\mathcal A_0$  образует  слабо $(k-2)$-разделенную систему подмножеств
множества $[n-1]$. } \medskip

Доказательство получается их следующих трех простых замечаний.  Первое:
множество из $\mathcal S _l$ слабо $(k-2)$-разделено с любым  множеством
меньшего или равного размера. В самом деле, оно имеет вид $I_1\sqcup...\sqcup
I_{k-2}$ для некоторого разбиения $[n-1]=I_0\sqcup ...\sqcup I_{k-1}$ на $k$
интервалов. Условие на размер дает сильную $(k-1)$-разделенность, то есть
слабую $(k-2)$-разделенность. Второе: по тем же причинам множество из $\mathcal
A _l$ слабо $(k-2)$-разделено с любым  множеством большего или равного размера.
Третье -- множества из $\mathcal D _i$ слабо $(k-2)$-разделены по Лемме 5 (они
$(k-1)$-разделены и имеют одинаковый размер). $\Box$\medskip

По предположению индукции размер системы $\mathcal S _{n-1}\sqcup ...\sqcup
\mathcal S _{i+1}\sqcup      \mathcal D _i\sqcup \mathcal A _{i-1}\sqcup
...\sqcup \mathcal A _0$ не превосходит ${n-1 \choose \le k-1}$.  С другой
стороны размер системы $\mathcal S $ (как спектра стандартного кубильяжа
зонотопа  $Z(n-1,k-1)$) равен ${n-1 \choose \le k-1}$. Откуда мы заключаем, что
размер $\mathcal D _i$ не превосходит разности размеров $\mathcal S _0\sqcup
\mathcal S _1\sqcup ...\sqcup \mathcal S _i$ и $\mathcal A _0\sqcup ...\sqcup
\mathcal A _{i-1}$. И осталось разобраться с  размером этой разности. \medskip

\textbf{Лемма 7.} \emph{Эта разность равна числу вершин стандартного (а на
самом деле -- любого) кубильяжа зонотопа $Z(n-1, k)$ на высоте $i$.}\medskip

Имея это равенство, мы быстро завершаем доказательство Предложения 22. В самом
деле, размер $\mathcal D _i$ не превосходит числа вершин стандартного кубильяжа
на уровне $i$. Поэтому размер всей системы $\mathcal D$ не превосходит числа
вершин стандартного кубильяжа $Z(n-1, k)$, то есть ${n-1 \choose \le k}$. Но
именно нужно было доказать.\medskip

Доказательство леммы 7. Рассмотрим зонотоп $Z=Z(n-1,k)$. Элементы системы
$\mathcal S $ реализуются как спектры вершины видимой части  границы $Z$,
элементы $\mathcal A $ -- как спектры вершин невидимой части границы. Обрежем
$Z$ на высоте $i$. Тогда элементы $\mathcal S _0\sqcup \mathcal S _1\sqcup
...\sqcup \mathcal S _i$ реалзуются как вершины видимой части этого обрубка
(`чаши'), тогда как элементы $\mathcal A _0\sqcup ...\sqcup  \mathcal A _{i-1}$
-- как вершины  невидимой части (исключая те, которые находятся на высоте $i$).

А теперь возьмем любой кубильяж $Z$, напрмер, стандартный. И рассмотрим
гирлянды (см. раздел 7) в этом кубильяже, точнее, те, которые начинаются в
точках видимой части границы `обрубка'. Каждая гирлянда где-то пересекает
другую  границу `обрубка', а это либо невидимоя часть обрубка, либо его верхнее
`горизонтальное' основание (потолок). И при этом получается любая вершина
кубильяжа этой невидимой части границы обрубка. Отсюда мы получаем, что
интересующая нас разность в точности равна числу вершин нашего  кубильяжа,
находящихся на высоте $i$.  Лемма 7 доказана, а вместе с ней доказано и
Предложение 22.
\end{proof}

Наконец, можно поставить вопрос о чистоте отношения слабой разделенности. Для
$k=1$ утвердительный ответ был получен в \cite{UMN}. Однако уже при $k=3$ ответ
отрицателен. Для получаеия контрпримера обратимся снова к разделу 19, то есть к
системе подмножеств в $[6]$. В этом случае есть 52 периферийных точек-множеств,
которые 3-разделимы с любыми подмножествами в $[6]$. Мы утверждаем, что

а) Три множества 25, 1356 и 1246 слабо 3-разделены. Они даже 3-разделены, см. раздел 19.

б) Вместе с периферийными точками-множествами  эти три множества образуют
нерасширяемую дальше (максимальную по включению) слабо 3-разделенную систему
размера 52+3=55. Тогда как ${6 \choose \le 4}=57$. Утверждение про
нерасширяемость проверяется непосредственно. Например, множество 25 4-разделено
с множествами 136, 1346 и 146, но разделено не сильно из соображений размера.
\medskip

Отметим, с другой стороны, что нерасширяемая 3-разделенная система $\{246, 235,
\\ 136\}$, рассмотренная как слабо 3-разделенная, допускает расширение за счет
добавления двух множеств 146 и 245.

\addcontentsline{toc}{section}{Дополнение 1. Поликатегорный взгляд на кубильяжи}

   \section*{Дополнение 1. Поликатегорный взгляд на кубильяжи}

Попробуем дать некоторое представление об этой достаточно изощренной
конструкции, придуманной Маниным, Шехтманом, Капрановым и Воеводским
(\cite{MSch-3, VK, KV}).

В основе лежит идея, что кубильяжи можно трактовать как  морфизмы, но между
чем? Между кубильяжами меньшей размерности. Но  те, в свою очередь, можно
рассматривать как морфизмы, и так мы  приходим не просто к категории, но к
поликатегории. Грубо  говоря, поликатегория -- это система категорий, в которой
для любых  объектов $a$ и $b$ одной категории множество $Hom(a,b)$ в свою
очередь является  множеством объектов другой категории, более высокого уровня.

Чтобы настроиться на этот поликатегорный взгляд, вернемся к ромбическим
тайлингам. Пусть есть две мембраны в  тайлинге, то есть две змейки, идущие в
зоногоне $Z(n,2)$ из  нижней  вершины  $\emptyset$ в верхнюю  вершину $[n]$.
Предполагая, что  вторая мембрана $\mathcal M _2$ проходит правее первой
$\mathcal M_1$, мы можем сказать,  что $\mathcal M _1\le \mathcal M _2$, если
между этими мембранами-змейками  можно вписать (частичный)  ромбический тайлинг
(или, что эквивалентно, обе мембраны вписаны в  некоторый ромбический тайлинг).
Интерпретируя мембраны-змейки как линейные порядки (или перестановки) на
множестве $[n]$, мы получаем (слабый) порядок Брюа.

Однако можно поступить более тонко и сказать, что нас интересует не просто
существование тайлингов, `соединяющих'  $\mathcal M _1$ и $\mathcal M _2$, но и
сами `соединяющие' их тайлинги. Иначе говоря, морфизмом  из $\mathcal M _1$ в
$\mathcal M _2$ объявляется произвольный (частичный) тайлинг $\mathcal T $
между $\mathcal M_1$ и $\mathcal M _2$. Так мы получаем вместо посета Брюа
категорию Брюа, в которой объектами служат перестановки (или линейные порядки
на $[n]$), морфизмами --  соединяющие их тайлинги, а композиция морфизмов
производится просто путем объединения тайлингов. На самом деле, мы  получаем
большее: ведь два тайлинга $\mathcal T _1$ и $\mathcal T_2$ между $\mathcal M
_1$ и $\mathcal M _2$ тоже можно  сравнивать с помощью повышающих флипов. Так
что $Hom(\mathcal M_1,\mathcal M _2)$ не просто (довольно богатое) множество,
но тоже посет! И композиция  морфизмов-тайлингов уважает эти структуры посетов.

Но если уж мы встали на этот путь, то грех останавливаться. Можно не просто
говорить, что тайлинги $\mathcal T _1$ и $\mathcal T _2$ соединяются
повышающими флипами, но само такое соединение назвать морфизмом  (уже
следующего, второго уровня). А соединение -- это некоторый  (частичный)
кубильяж уже в зонотопе $Z(n,3)$. И эту конструкцию  можно продолжать и далее,
поднимаясь все выше и выше по  размерности, пока мы не дойдем до зонотопа-куба
$Z(n,n)$, в котором имеется уже только один тривиальный кубильяж.

           Так выглядит эта поликатегорная картина в общих чертах.

В следующем примере мы приведем лишь наиболее интересный фрагмент
поликатегорного полотна, да и то в предельно простом случае. А  именно, возьмем
$n=3$, и изобразим лишь первый, второй и третий этажи  получающейся `пагоды'.

Объектами первого этажа служат вершины куба $Z(3,3)$, то есть подмножества
базисного множества $[3]=\{1,2,3\}$. Но что считать морфизмами? Казалось бы,
естественно полагать, что из множества $X$  идет стрелка в $Y$, если $X\subset
Y$. На этом пути мы приходим просто к  посету $2^{[3]}$. Кубильяжная точка
зрения говорит, что морфизмом надо считать `путь' от $X$ в $Y$, частичную
змейку из  $X$ в $Y$. Композиция змеек очевидна: если есть змейка из $X$ в $Y$
и из $Y$ в $Z$, надо просто приставить их друг к другу и получить змейку из $X$
в $Z$.

 \begin{figure}[htb]
\unitlength=.8mm
\special{em:linewidth 0.4pt}
\linethickness{0.4pt}
\begin{picture}(136.00,48.00)(-15,0)
\bezier{72}(30.00,5.00)(22.00,10.00)(15.00,15.00)
\bezier{60}(15.00,15.00)(15.00,22.00)(15.00,30.00)
\bezier{72}(15.00,30.00)(22.00,35.00)(30.00,40.00)
\bezier{22}(30.00,40.00)(36.00,36.00)(45.00,30.00)
\bezier{20}(45.00,30.00)(45.00,23.00)(45.00,15.00)
\bezier{72}(45.00,15.00)(38.00,11.00)(30.00,5.00)
\put(30,5){\vector(-3,2){15}}
\put(15,15){\vector(0,1){15}}
\put(15,30){\vector(3,2){15}}
\put(30,25){\vector(0,1){15}}
\put(30,5){\vector(3,2){15}}
\put(45,15){\vector(-3,2){15}}
\bezier{32}(15.00,15.00)(22.00,20.00)(30.00,25.00)
\put(10,20){$a$}
\put(38,20){$b$}

\bezier{72}(80.00,5.00)(72.00,10.00)(65.00,15.00)
\bezier{20}(65.00,15.00)(65.00,22.00)(65.00,30.00)
\bezier{22}(65.00,30.00)(72.00,35.00)(80.00,40.00)
\bezier{72}(80.00,40.00)(86.00,36.00)(95.00,30.00)
\bezier{20}(95.00,30.00)(95.00,23.00)(95.00,15.00)
\bezier{22}(95.00,15.00)(88.00,11.00)(80.00,5.00)
\put(80,5){\vector(-3,2){15}}
\put(65,15){\vector(3,2){15}}
\put(80,20){\vector(3,2){15}}
\put(80,5){\vector(0,1){15}} \put(80,25){\vector(0,1){15}}
\put(95,30){\vector(-3,2){15}}
\put(70,21){$a$}
\put(85,19){$b$}

\bezier{72}(130.00,5.00)(122.00,10.00)(115.00,15.00)
\bezier{60}(115.00,15.00)(115.00,22.00)(115.00,30.00)
\bezier{72}(115.00,30.00)(122.00,35.00)(130.00,40.00)
\bezier{72}(130.00,40.00)(136.00,36.00)(145.00,30.00)
\bezier{60}(145.00,30.00)(145.00,23.00)(145.00,15.00)
\bezier{72}(145.00,15.00)(138.00,11.00)(130.00,5.00)
\put(130,5){\vector(-3,2){15}}
\put(115,15){\vector(0,1){15}}
\put(115,30){\vector(3,2){15}}
\put(145,15){\vector(0,1){15}}
\put(130,5){\vector(3,2){15}}
\put(145,30){\vector(-3,2){15}}

\put(130,5){\line(0,1){15}}
\put(130,20){\line(-3,2){15}}
\put(130,20){\line(3,2){15}}
\bezier{30}(130.00,25.00)(130.00,32.00)(130.00,40.00)
\bezier{32}(115.00,15.00)(122.00,20.00)(130.00,25.00)
\bezier{32}(130.00,25.00)(136.00,20.00)(145.00,15.00)

\put(110,20){$a$}
\put(148,20){$b$}
\end{picture}
  \caption{Змейки в зоногоне $Z(3,2)$}
 \label{fig:17}
  \end{figure}

Объектами второго этажа являются {частичные змейки. Если есть две змейки $a$ и
$b$ с одним и тем же началом и одним и тем же концом, возникает вопрос} --
можно ли замостить  промежуток между ними ромбами? И морфизмами второго этажа
(между $a$ и  $b$)  являются  именно такие ромбические замощения области между
$a$ и $b$ (здесь надо считать, что змейка $b$ проходит правее змейки $a$). В
нашем случае с $n=3$ такое замощение бывает либо одно  (как слева на рис. 17),
или ни одного (как на среднем рисунке), где змейки идут из      $\emptyset$ в
$[3]$).

Однако (в нашем простом случае) есть единственное исключение,  когда между
змейкой $123$ и $321$ (здесь мы понимаем сокращение $123$ как линейный порядок
на множестве $\{1,2,3\}$) возможны два замощения  или тайлинга:  стандартный и
антистандартный. Так что если нарисовать часть  второго этажа, содержащую
только `полные' змейки (из $\emptyset $  в $[3]$,  то есть шесть линейных
порядков на множестве $[3]$), мы получим следущую картину (рис. 18).

 \begin{figure}[htb]
\unitlength=1mm
\special{em:linewidth 0.4pt}
\linethickness{0.4pt}
\begin{picture}(97.00,43.00)
\put(40.00,20.00){\makebox(0,0)[cc]{123}}
\put(42.00,18.00){\vector(1,-1){13.00}}
\put(57.00,3.00){\makebox(0,0)[cc]{132}}
\put(59.00,3.00){\vector(1,0){18.00}}
\put(80.00,3.00){\makebox(0,0)[cc]{312}}
\put(82.00,5.00){\vector(1,1){13.00}}
\put(97.00,20.00){\makebox(0,0)[cc]{321}}

\put(43.00,23.00){\vector(3,1){35.00}}
\put(43.00,17.00){\vector(3,-1){36.00}}

\put(42.00,22.00){\vector(1,1){13.00}}
\put(57.00,37.00){\makebox(0,0)[cc]{213}}
\put(59.00,37.00){\vector(1,0){18.00}}
\put(80.00,37.00){\makebox(0,0)[cc]{231}}
\put(82.00,35.00){\vector(1,-1){13.00}}
\bezier{216}(42.00,22.00)(67.00,32.00)(94.00,22.00)
\bezier{216}(42.00,18.00)(65.00,8.00)(94.00,18.00)

\put(58.00,6.00){\vector(3,1){36.00}}
\put(58.00,34.00){\vector(3,-1){35.00}}

\put(68.00,20.00){\makebox(0,0)[cc]{\Large $\Downarrow$}}
\end{picture}
  \caption{Частичная картина поликатегории в случае $n=3$.}
 \label{fig:17}
  \end{figure}

Мы хотели бы обратить внимание, что из змейки 123 в змейку 321 идут два
морфизма-тайлинга: стандартный и антистандартный. И эти два морфизма, как
объекты третьего этажа, тоже связаны морфизмом (третьего уровня), который
изображен двойной стрелкой $\Rightarrow$. (Конечно, на третьем этаже имеется
еще много `банальных'
      (изолированных) объектов.)

Этим примером мы хотели бы завершить рассказ о  поликатегорном взгляде на
кубильяжи. Подробности фомализма можно  посмотреть в статье Манина-Шехтмана
\cite{MSch-3} или Капранова-Воеводского \cite{VK, KV}. Само понятие
поликатегории кратко обсуждается в книге \cite{Mac} и более подробно -- в книге
\cite{Lu-1}. Отметим еще статью \cite{KV-3}, где на на поликатегорном языке
обсуждается уравнение Замолодчикова.

\addcontentsline{toc}{section}{Дополнение 2. Связь с триангуляциями и посетом Тамари-Сташева  }

    \section*{Дополнение 2. Связь с триангуляциями и посетом Тамари-Сташева  }

Практически одновременно с введением кубильяжей в форме высшего посета Брюа
появились (в работе Воеводского и Капранова \cite{KV}) посеты Тамари-Сташева,
которые во многом параллельны посетам Брюа и тесно с ними связаны. На эту тему
написано  много работ, с которыми можно познакомиться по обзору \cite{RR}.

Чтобы определить эти посеты, снова рассмотрим циклическую  конфигурацию
векторов $C(n,d)$. Концы этих векторов являются точками  на гиперплоскости в
$\mathbb R ^d$, заданной уравнением $x_1=1$, то есть      гиперплоскостью на
`высоте' 1. Выпуклая оболочка этих точек  $v_1,...,v_n$ называется {\em
циклическим политопом} и обозначается  $P(n,d)$; это пересечение зонотопа
$Z(n,d)$ с гиперплоскостью, и его размерность равна $d-1$. При  $d>2$ точки
$v_i$ являются вершинами этого политопа.

Вместо кубильяжей зонотопа $Z(n,d)$ теперь интересуются  триангуляциями
политопа $P(n,d)$, то есть разбиениями на $(d-1)$-мерные  симплексы, вершины
которых лежат в множестве $\{v_1,...,v_n\}$. При  $d=2$ это разбиение отрезка
$[v_1,v_n]$ на подотрезки, при $d=3$ -- разбиение выпуклого $n$-угольника на
$(n-2)$ треугольника, и так  далее. Множество триангуляций обозначим ${\bf
TS}(n,d)$.

Триангуляции во многом похожи на кубильяжи. Например, среди  них есть
`стандартная' и `антистандартная' триангуляции (или  нижняя и верхняя
триангуляции в терминологии  \cite{RR}), которые  получаются как ограничения
соответствующих кубильяжей. В  частности, этими триангуляциями исчерпываются
множества ${\bf TS}(d+1,d)$. Замена стандартной триангуляции на антистандартную
в ${\bf  TS}(d+1,d)$ называется повышающим флипом. Аналогичную вещь (повышающий
флип)  можно делать при любом $n$, если удается найти фрагмент вида стандартной
триангуляции ${\bf TS}(d+1,d)$. Мы говорим, что $\mathcal T \le \mathcal T '$,
если от $\mathcal T $ можно  добраться до $\mathcal T '$ серией повышающих
флипов. Это дает структуру  посета в множестве ${\bf TS}(n,d)$, называемую {\em
посетом  Тамари-Сташева}.\medskip

\textbf{Пример 1.} $d=2$. Задать тут триангуляцию -- значит задать  разбиение
на отрезки отрезка $[t_1,t_n]$, причем так, чтобы концы  отрезков были какие-то
$t_i$. Любое подмножество в $\{t_2,...,t_{n-1}\}$  задает такое разбиение на
отрезки. Типичный повышающий флип  состоит в замене двух последовательных
отрезков $[t_i,t_j]$ и    $[t_j,t_k]$ ($t_i<t_j<t_k$) на отрезок $[t_i,t_k]$,
то есть фактически  удаление $t_j$ из подмножества. Мы видим, что посет
Тамари-Сташева в  этом случае (анти)изоморфен булевой решетке подмножеств
множества  $\{2,...,n-1\}$. \medskip

\textbf{Пример 2.} Случай $d=3$ более интересен. Надо взять $n$ точек на
параболе $y=x^2$, упорядоченных по возрастанию $x$; выпуклая оболочка этих
точек дает циклический многоугольник $P=P(n,3)$. Возьмем некоторую  его
триангуляцию и войдем (сверху) в треугольник со стороной  $[v_1,v_n]$. После
чего выйдем через любую из оставшихся сторон. Если при этом мы покидаем $P$, на
этом путь заканчивается. Если попали в другой треугольник триангуляции, процесс
продолжается -- мы снова выходим через одну из двух других сторон  этого
треугольника. И так далее. В результате получается то, что называется {\em
плоским бинарным деревом} (с $(n-2)$-мя некорневыми вершинами, соответствующими
треугольникам триангуляции), см. рис. 20 и 21. В книге \cite{Sta-2} (упр. 6.19)
приведены еще 65 способов задания множества таких деревьев, среди которых стоит
выдеить способ `правильной' расстановки скобок в строке из $(n-1)$-й буквы.
Этот способ использовал Тамари при определении его посета ${\bf TS}(n,3)$.

Чтобы говорить не просто о множестве ${\bf TS}(n,3)$, но и о частичном  порядке
на нем, надо конкретизировавать повышающие флипы. Они  устроены так: надо взять
фрагмент триангуляции, изображенный слева на рис. 19, и заменить диагональ
($ik$) четырехугольника на другую диагональ ($jl$).

\begin{figure}[h]
\begin{center}
\includegraphics[scale=0.3]{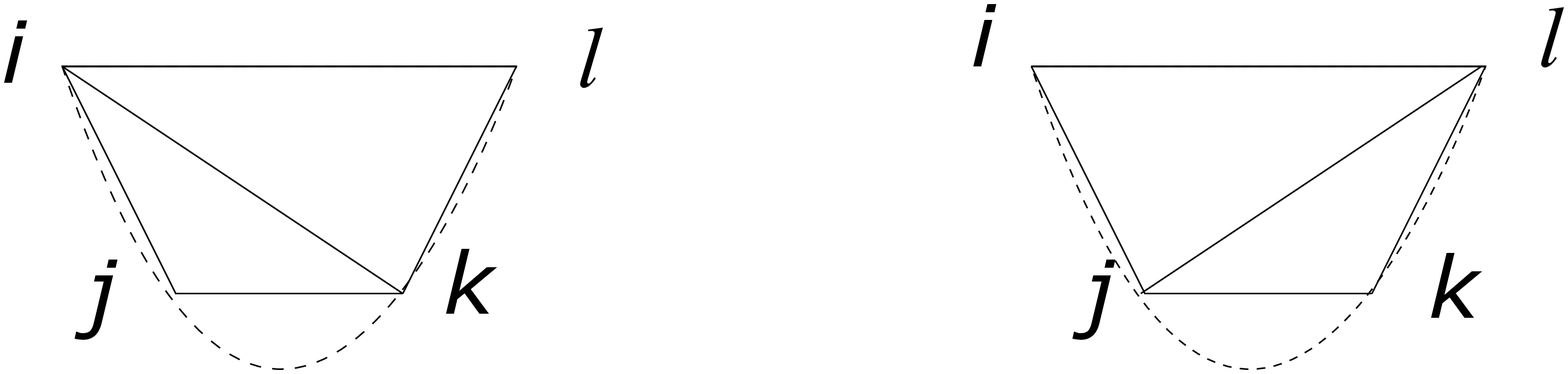}
\end{center}
 \caption{флип триангуляции}
 \label{fig:19}
  \end{figure}

Ниже нарисованы пять трангуляций пятиугольника и соответствующие бинарные
деревья.

\begin{figure}[h]
\begin{center}
\includegraphics[scale=0.25]{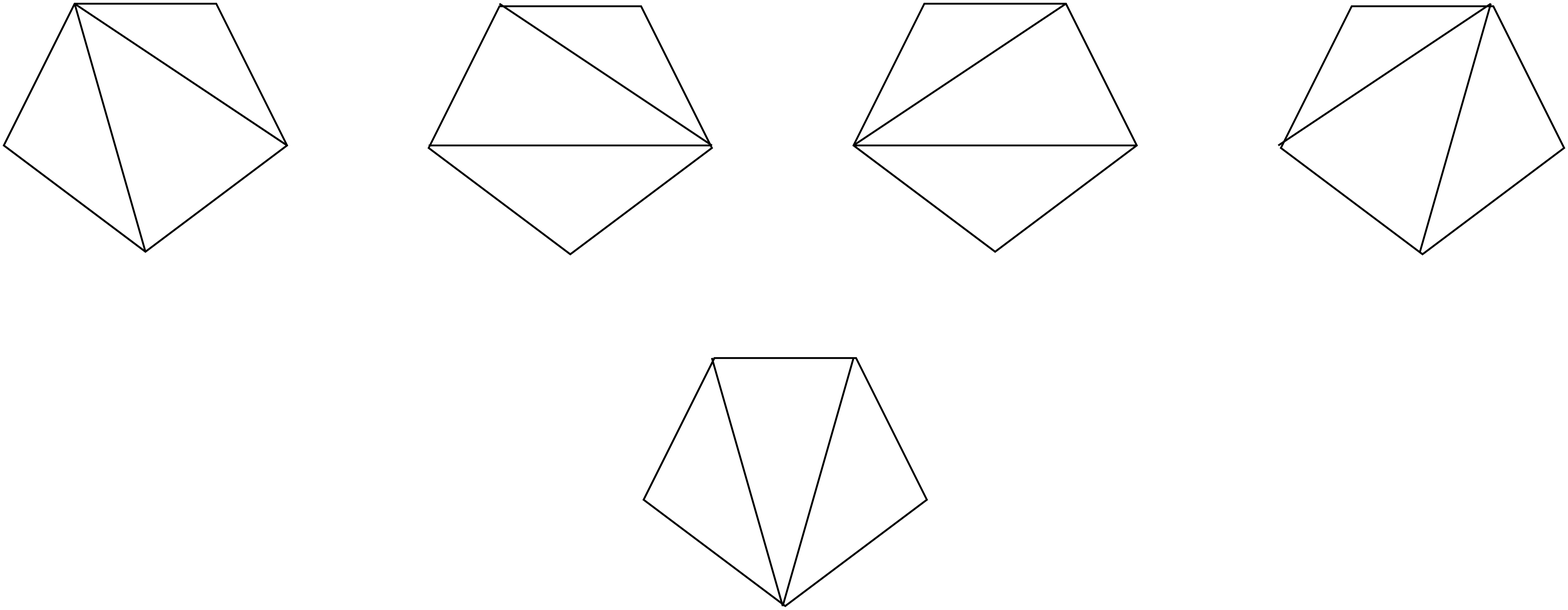}
\end{center}
 \caption{\small пять триангуляций пятиугольника $P(5,2)$.}
 \label{fig:20}
  \end{figure}

\begin{figure}[hb]
\begin{center}
\includegraphics[scale=0.2]{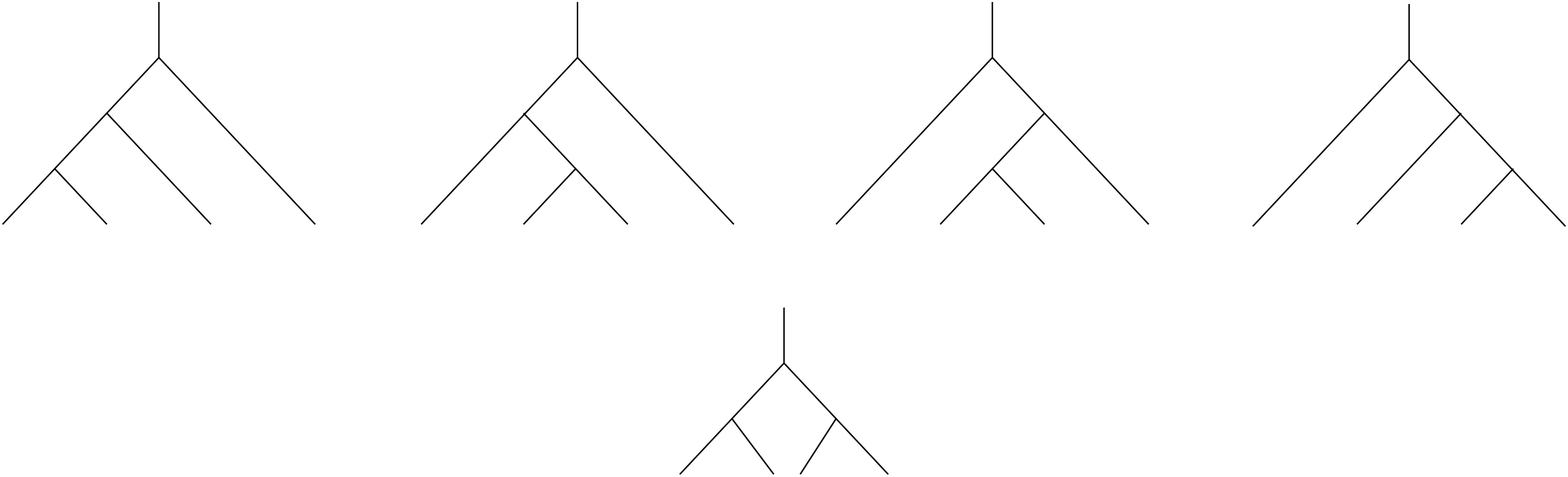}
\end{center}
 \caption{\small соответствующие пять бинарных дерева.}
 \label{fig:21}
  \end{figure}

           Помимо очевидной аналогии с кубильяжами имеется более точная
      связь между ними. Пусть $\mathcal Q $ -- кубильяж зонотопа
      $Z(n,d)$. Если мы
      пересечем этот зонотоп гиперплоскостью $x_1=1$, мы получим в
      точности политоп $P(n,d)$; сечения кубов той же гиперплоскостью дают
      триангуляцию политопа $P(n,d)$. Это дает естественное отображение
                                    $$
                           sec:\textbf{Q}(n,d) \to {\bf TS}(n,d).
                                      $$
Триангуляция $\mathcal T =sec(\mathcal Q )$ дает полное представление о кубах,
лежащих в `основании' кубильяжа $\mathcal Q $ (то есть укорененных в 0), но
мало что говорит о том, как лежат кубы выше. Поэтому нет оснований надеяться на
инъективность отображения $sec$, и уже простейшие примеры показывают
неинъективность. В то же время  есть веские основания надеяться, что верна
следующпя\medskip

\textbf{Гипотеза 1.} {\em Отображение $sec:\textbf{Q}(n,d) \to {\bf TS}(n,d)$
      сюръективно.}\medskip

Во всяком случае это так для $d\le 3$. В случае $d=2$ это простое  упражнение.
Случай $d=3$ можно получить из работы  \cite[sec. 7]{DKK-18} или из результатов
\cite{Ga}.

Отметим, что в обзоре \cite{RR} (в пункте 8.3) также говорится про связь
посетов Тамари и Брюа, но там обсуждается другое  отображение. Что же касается
отображения $sec$, то оно обсуждается в  работе \cite{DMH}. Там  заявляется о
сюръективности этого отображения, правда посет Тамари понимается  несколько
иначе, более `комбинаторно'. А вся  их работа посвящена применению
Тамари-посетов к солетонным решениям уравнения Кадомцева-Петвиашвили. Имеется
еще несколько работ о связи уравнения КП и кубильяжей-триангуляций, из которых
отметим самую последнюю \cite{Ko}

Ометим еще, что отображение $sec:{\bf Q}(n,d) \to {\bf  TS}(n,d)$  согласовано
с флипами. Если мы делаем флип в кубильяже  $\mathcal Q $,  то он может
индуцировать флип в триангуляции $sec(\mathcal Q )$. Точнее, если мы  делаем
флип внутри некоторого капсида, укорененного в 0, мы получаем флип
соответствующей триангуляции. Если же капсид  расположен `высоко', то флип
кубильяжа не оказывает никакого воздействия на триангуляцию. В любом случае это
дает, что  отображение $sec:{\bf Q}(n,d) \to {\bf TS}(n,d)$ согласовано со
структурами посетов.

\addcontentsline{toc}{section}{Дополнение 3. Слабые мембраны }

      \section*{Дополнение 3. Слабые мембраны    }

Это дополнение непосредственно продолжает материал разделов 9-11. С другой
стороны оно примыкает к тематике      Дополнения 2 про триангуляции циклических
политопов.

В разделе 6 вводилось понятие мембраны в кубильяже. Это $(d-1)$-мерный
подкомплекс кубильяжа, который при  проекции $\pi$ вдоль  $d$-го координатного
вектора $e_d$ биективно проектируется туда же, куда весь зонотоп $Z(n,d)$.
Аналогично понимается и слабая мембрана. Это тоже некоторая ($d-1$)-мерная
пленка в зонотопе $Z(n,d)$, биективно проектирующаяся туда же, куда и весь
зонотоп. Но есть два отличия. Первое -- эта пленка уже не является
подкомплексом кубильяжа $\mathcal Q $, но подкомплексом  некоторого измельчения
этого кубильяжа. Второе -- проектируем мы не  вдоль вектора (направления)
$e_d$, но вдоль вектора $e_d+\varepsilon e_1$, где $\varepsilon $ -- небольшое
положительное число. То есть мы глядим на зонотоп почти в  направлении $e_d$,
но чуть-чуть `снизу'. Иначе говоря, проекция $\pi_\varepsilon$ вдоль вектора
$e_d+\varepsilon e_1$ переводит точку $x=(x_1,x_2,...,x_d)\in \mathbb R^d$ в
$\pi_\varepsilon(x)=(x_1-\varepsilon x_d, x_2,...,x_{d-1})\in \mathbb R^{d-1}$.

Скажем более точно. Циклический зонотоп $Z=Z(n,d)$ как бы растет из точки 0
`вверх', до высоты $n$, и все его вершины, как и вершины кубильяжа $\mathcal Q
$, находятся на целочисленной высоте. Рассечем и зонотоп $Z$, и все кубы
кубильяжа $\mathcal Q $ `горизонтальными' гиперплоскостями  $H_k$, заданными
уравнениями $x_1=k$, где $k$ пробегает целые число от 1 до $n-1$. В результате
каждый куб $Q$ будет разрезан на $d$ частей , называемых {\em кусками}. Каждый
кусок -- это некоторый гиперсимплекс (в терминологии \cite{GS}). На рис. 22
показано рассечение трехмерного куба на три куска: нижний тетраэдр, октаэдр и
верхний тетраэдр.

\begin{figure}[htb]
\begin{center}
\includegraphics[scale=0.7]{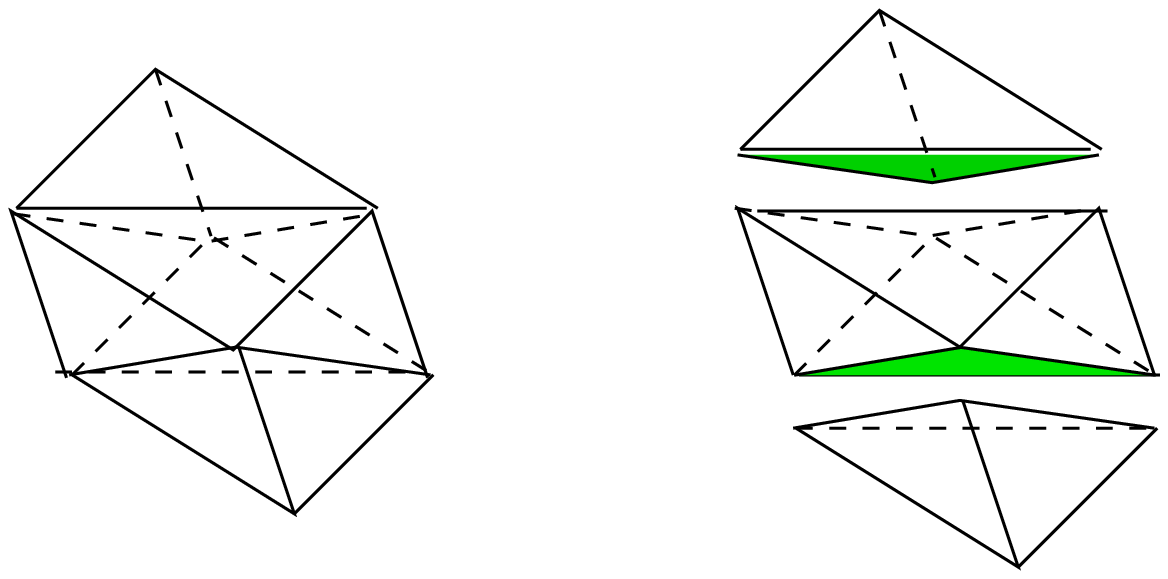}
\end{center}
 \caption{куски}
 \label{fig:22}
  \end{figure}

Получается разбиение зонотопа $Z$, но уже не на кубы, но на более мелкие
куски-гиперсимплексы, которое мы обозначаем $\mathcal Q \!\! \equiv $ и
называем {\em измельчением} кубильяжа      $\mathcal Q $.

Измельченный кубильяж  $\mathcal Q \!\!\! \equiv $ имеет те же вершины, что и
$\mathcal Q $. Но нас больше будут      интересовать `фасеты' $\mathcal Q
\!\!\equiv $, или фасеты кусков. Они бывают `вертикальные' (это куски фасет
кубов кубильяжа $\mathcal Q $) и `горизонтальные', лежащие в некоторой
гиперплоскости $H_k$. (И те, и другие -- гиперсимплексы, только теперь
размерности $d-1$.) Большинство фасет куска $P$ вертикальные, и только две (или
одна,  если кусок -- симплекс) горизонтальные.

\emph{Слабая мембрана} (или $w$-\emph{мембрана}) в кубильяже $\mathcal Q $ --
это подкомплекс $\mathcal W $ измельчения $\mathcal Q $, который при проекции
$\pi _\varepsilon $ вдоль вектора $e_d+\varepsilon e_1$ , биективно
(гомеоморфно) проектируется на  зонотоп (меньшей размерности) $\pi _\varepsilon
(Z)$. Обычные мембраны (из раздела 6, их можно называть сильными) могут
пониматься как слабые. Как и сильные мембраны, слабые мембраны гомеоморфны
($d-1$)-мерному диску, край которого -- это обод зонотопа $Z$  (относительно
$\pi $ или $\pi _\varepsilon $), и они делят  зонотоп $Z$ на две части -- перед
мембраной ($Z_-(\mathcal W )$) и после мембраны  ($Z_+(\mathcal W )$); каждая
из этих частей тоже разбита на куски. Проекции (при $\pi _\varepsilon $) клеток
мембраны дают некоторое разбиение зонотопа $\pi _\varepsilon (Z)$ на
плитки-гиперсимплексы, которое можно было бы назвать \emph{гиперкомби}, по
аналогии с понятием комби в размерности 2, \cite{DKK-17, DKK-18}. Это
интересный объект, но мы      (пока) не будем им заниматься.

Главное отличие слабых мембран от сильных в том, что первые  могут иметь
горизонтальные участки (`выступы' или `балконы'), см. рис. 23 и 24.

\begin{figure}[h]
\begin{center}
\includegraphics[scale=0.2]{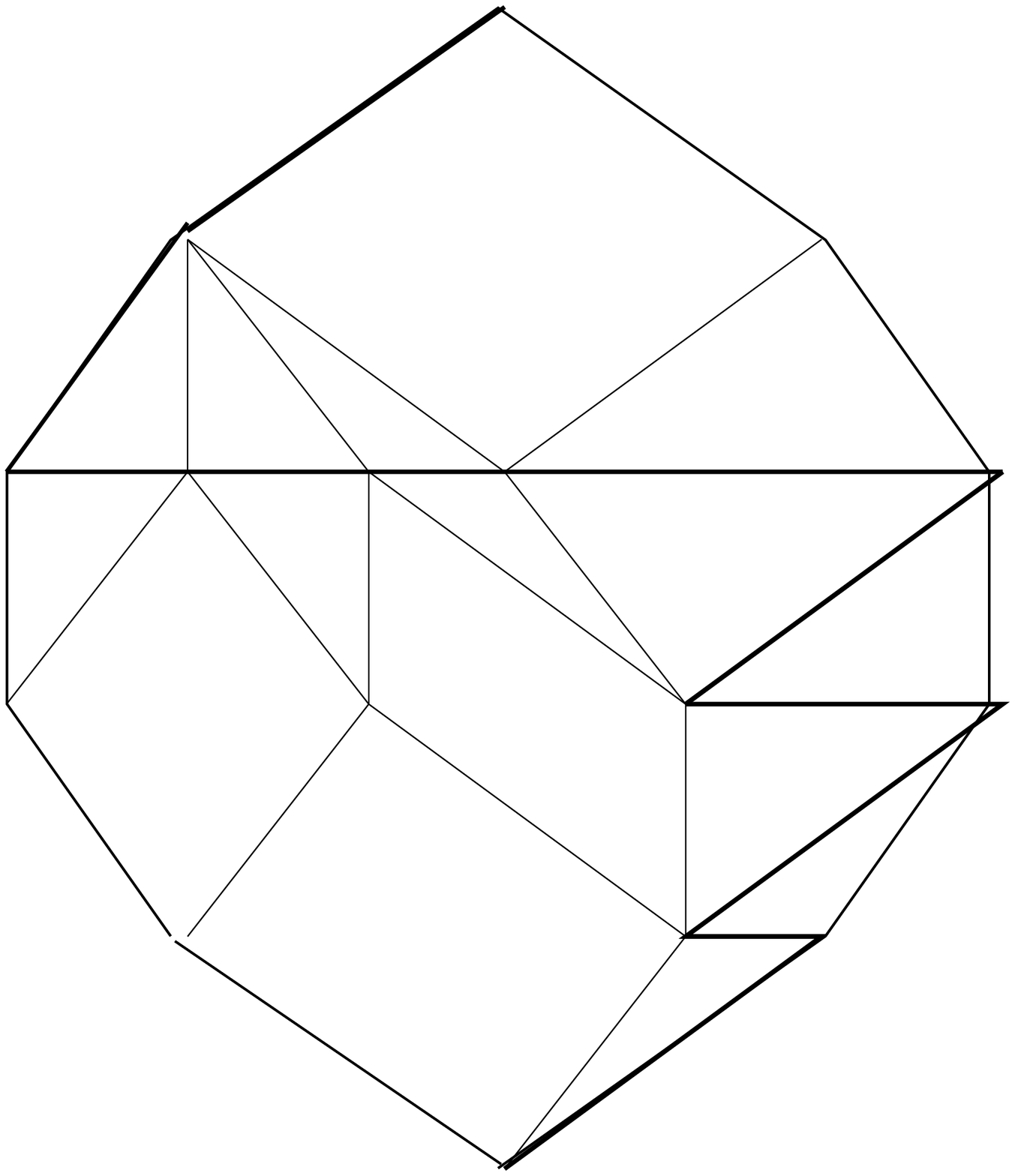}
\end{center}
 \caption{слабая мембрана в ромбическом тайлинге; видны три балкона.}
 \label{fig:23}
  \end{figure}

Простейшие примеры $w$-мембран (назовем их \emph{главными}) получаются
следующим образом. Пусть $\mathcal Q $ -- кубильяж зонотопа $Z=Z(n,d)$. Мы идем
(сверху вниз) сначала по передней (видимой) части границы $Z$ до уровня  $k$,
$0\le k\le n$. Затем идем горизонтально по этому уровню $k$ вплоть до невидимой
(задней) части границы и наконец спускаемся до 0 по задней границе. На рис.
\ref{fig:24} слева изображена главная $w$-мембрана (при $k=2$) для зоногона
$Z(5,2)$, а справа -- для зонотопа $Z(5,3)$ (вид спереди и немного снизу).

\begin{figure}[h]
\begin{center}
\includegraphics[scale=0.2]{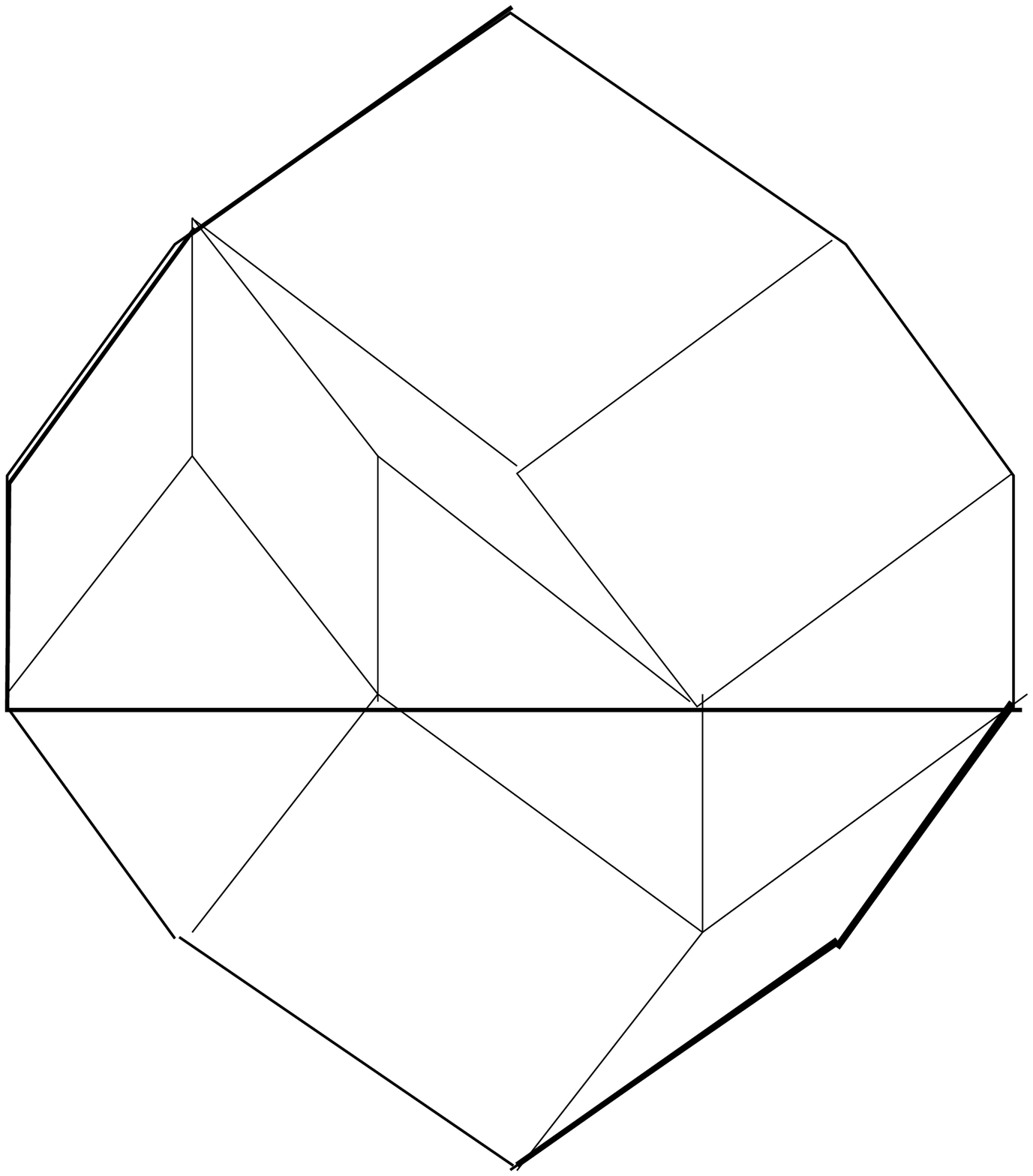}
\hspace{3cm}
\includegraphics[scale=0.2]{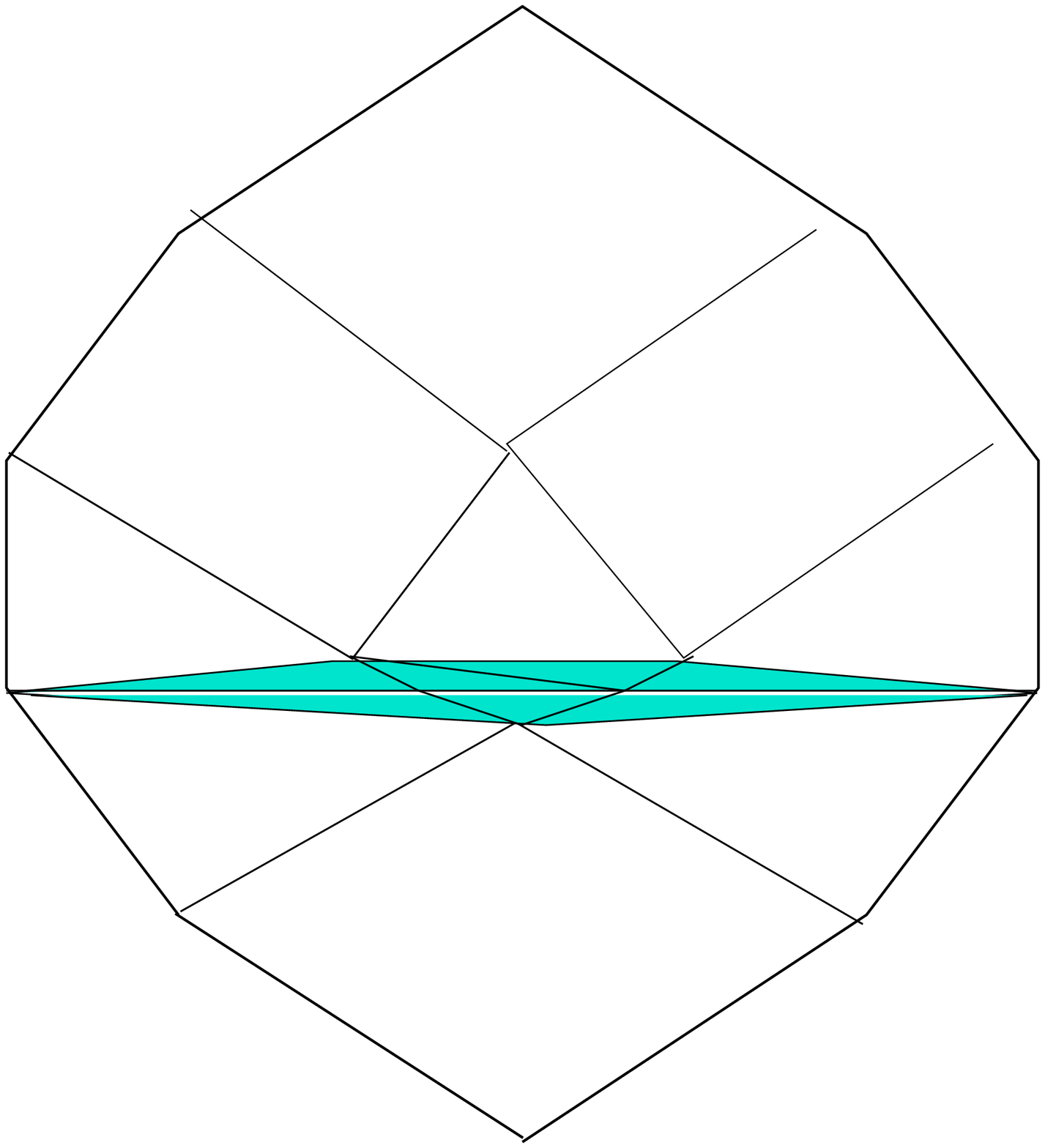}
\end{center}
 \caption{простые мембраны.}
 \label{fig:24}
  \end{figure}

Тут уступы имеются лишь на высоте 2; в общем случае уступы могут  встречаться
на любой высоте. Для тривиального кубильяжа куба $Z(d,d)$ все мембраны главные.

Отметим, что триангуляции циклического политопа из Дополнения 2 можно понимать
как главные слабые мембраны (для $k=1$).

Основная вещь, о которой мы хотим рассказать -- это аналог  естественного
порядка на множестве кусков измельченного кубильяжа $\mathcal Q \!\! \equiv $.
Поступим в точности как в разделе 9. А именно, введем понятие
\emph{непосредственного следования} $\prec $ для кусков. Только проектируем (и,
соотвественно, говорим о видимых и невидимых фасетах кусков) не отображением
$\pi $, а чуть измененным отображением $\pi _\varepsilon $. При этом  заметим
следующее. Пусть кусок $P$ (лежащий в кубе $Q$) непосредственно  предшествует
куску $P'$ (лежащему в кубе $Q'$), $P\prec P'$. Тогда возможны  два случая:

           1) куски $P$ и $P'$ лежат на одном этаже (между гиперплоскостями
$H_k$ и $H_{k+1}$), и тогда $Q\prec Q'$;

           2) куски $P$ и $P'$ разделены гиперплоскостью $H_k$; тогда $Q=Q'$, и $P$ лежит под $P'$.

Это замечание помогает доказать следующий аналог Предложения 6:\medskip

           {\bf Утверждение.} {\em  Отношение $\prec $ на $\mathcal Q\!\! \equiv $ ациклично. }\medskip

В самом деле, пусть $P_1\prec ...\prec P_m$ -- некоторый направленный  путь в
$\mathcal Q \equiv $. И пусть $Q_l$ -- куб кубильяжа  $\mathcal Q $, содержащий
кусок $P_l$. Тогда мы получаем направленный путь $Q_1...Q_m$ в $\mathcal Q $, в
котором  соседние кубы либо связаны отношением $\prec _\mathcal Q $, либо
совпадают. В силу      Предложения 6 раздела 9, если на этом пути встречается
хоть раз отношение $\prec _\mathcal Q $, то $Q_1$ отлично от $Q_m$ и тем  более
$P_1$ отлично  от $P_m$. Поэтому можно считать, что все $Q_l$ одни и те же и
равны      некоторому кубу $Q$. Тогда все куски $P_l$ -- это куски этого куба
$Q$. Теперь утверждение очевидно,  потому что каждый следующий кусок лежит на
этаж выше предыдущего. \hfill $\Box$\medskip

Основываясь на этом утверждении, мы транзитивно замыкаем отношение $\prec $ и
получаем отношение порядка $\preceq $ на  $\mathcal Q \!\!\!\equiv $, которое
снова называем \emph{естественным}. И точно так же, как это делалось в разделе
11, можно ввести понятие стэка (или идеала) для этого порядка и  отождествить
слабые мембраны (для кубильяжа $\mathcal Q $) со стэками. Польза этого простого
наблюдения в том, что стэк `разбирается'  путем удаления некоторого
максимального (относитеольно порядка $\preceq $) элемента (куска). В терминах
слабых мембран мы получаем понятие слабого (понижающего) флопа. Важное
следствие состоит в том, что от любой слабой мембраны до `минимальной' (и уже
сильной) мембраны можно добраться серией таких слабых понижающих флопов. Кроме
того, множество слабых  мембран (в кубильяже $\mathcal Q $) является (как в
Предложении 11)  дистрибутивной решеткой.

Однако на этом аналогии частично обрываются, и ситуация  начинает зависеть от
четности/нечетности $d$. Вернемся к  простейшему кубильяжу куба $Z(d,d)$. У
измельчения такого куба имеется $d$ кусков. И почти у всех кусков (кроме
средних) все вершины лежат на ободе  этого куба. Как мы знаем из раздела 7, не
на ободе находятся только две вершины, обозначенные там как $t$ и $h$ (хвост и
голова). Обе эти вершины принадлежат либо одному `среднему' куску (если $d$
нечетно), либо одному `среднему' сечению (если $d$ четно, см. Рис.
\ref{fig:4}). И это сказывается на числе вершин слабой мембраны. Когда мы
удаляем нецентральный кусок из груды $\mathcal S(\mathcal W )$ домембранных
кусков (делая соотвествующий понижающий флоп), то число вершин слабой мембраны
не меняется. Однако когда мы  удаляем центральные куски, положение усложняется.
Удаление более верхнего куска (при четном $d$) приводит к увеличению числа
вершин мембраны на 1, удаление более низкого -- к уменьшению на 1, а удаление
единственного центрального куска (при нечетном $d$) не меняет числа вершин
мембраны.\medskip

{\bf Вывод.} {\em При нечетном $d$ слабые мембраны имеют одно и то же  число
вершин, равное ${n \choose \le d-1}$. При четном $d$  число вершин у слабых
мембран может меняться (в каких пределах?).}\medskip

Например,  рис.  \ref{fig:23}  изображает тайлинг $Z(5,2)$ cо слабой мембраной,
у которой 12 вершин (вместо `нормальных' шести).\medskip

А теперь вспомним Предложение 22, которое говорит,  что (при нечетном $d$)
размер слабо $(d-2)$-разделенной системы в $[n]$ не превосходит ${n \choose \le
d-1}$. В~\cite{DKK19} доказывается, что \emph{при нечетном $d$ спектр любой
слабой мембраны в $Z(n,d)$  является слабо $(d-2)$-разделенной системой} (ср. с
Предложением 15) . Также есть основания предполагать следующее.
\medskip

\textbf{Гипотеза 2} (ср. с Теоремой 5). \emph{Если $\mathcal{W}$ -- слабо
$(d-2)$-разделенная система размера ${n \choose \le d-1}$, то она реализуется
как спектр некоторой слабой мембраной в зонотопе} $Z(n,d)$.\medskip

Во всяком случае, для $d=3$ эта гипотеза верна.

\addcontentsline{toc}{section}{Дополнение 4. Доказательство ацикличности}

\section*{Дополнение 4. Доказательство ацикличности}

В этом дополнении мы собираемся доказать Предложение 6 об ацикличности
отношения $\prec$ на множестве кубов кубильяжа  $\mathcal{Q}$. На самом деле мы
покажем большее -- что отношение $\prec$ ациклично не только на множестве кубов
фиксированного кубильяжа $\mathcal{Q}$, но на множестве $\mathcal{C}$ всех
кубов всех кубильяжей. (Абстрактным) \emph{кубом} называется произвольный
$d$-мерный куб $C$ в зонотопе $Z(n,d)$, порожденный некоторым набором векторов
в ${\bf V}$  и растущий из целой точки. Такой куб задается указанием своего
корня $v(X)$ ($X\subset [n]$) и своего типа $T\subset [n]$; $T$ имеет размер
$d$ и не пересекается с $X$. Как показано в Предложении 3, такой куб
вписывается в некоторый кубильяж. Множество всех кубов в $Z(n,d)$ обозначается
$\mathcal{C}(n,d)$; очевидно, оно равно  объединению множества всех кубов всех
кубильяжей $Z(n,d)$.

На множестве $\mathcal{C}(n,d)$ введем бинарное отношение $\prec$ ровно так,
как это было сделано в разделе 9. Напомним, что  $Q\prec Q'$, если эти кубы $Q$
и $Q'$ соседние по некоторой общей фасете $F$, причем эта фасета $F$ невидимая
в $Q$ и видимая в $Q'$. Предложение 6 следует из более сильного
утверждения.\medskip

\textbf{Теорема.} \emph{Отношение $\prec$ на множестве $\mathcal{C}(n,d)$ ациклично. }\medskip

Для доказательства теоремы мы выразим комбинаторно отношение $\prec$. Пусть куб
$Q$ равен $(X,T)$, куб $Q'$ --  $(X',T')$,  и   $Q\prec Q'$. Обозначим через
$F$ фасету, по которой эти кубы соседствуют. Комбинаторно эта фасета равна
$(S,J)$, где  $S$  - корень фасеты, а $J$ -- ее тип. Очевидно, $J=T\cap T'$,
размер ее равен $d-1$. Пусть $T=Ji$, $T'=Jk$ для некоторых цветов $i$ и $k$.
Вектор $v_i$ может либо входить в корень $F$, либо выходить из него. В первом
случае $S=Xi$, во втором $X=Si$. Аналогично с $k$.  Так что возможен один из
четырех случаев, изображенных на рисунке 25.

\begin{figure}[htb]
                         \unitlength=1.1mm
\special{em:linewidth 0.4pt}
\linethickness{0.4pt}
\begin{picture}(120.00,55.00)(-5,0)

\put(19.8,20){\line(0,1){15}} \put(20.2,20){\line(0,1){15}}
\put(20,20){\line(-1,1){10}} \put(20,20.2){\vector(-1,1){10}}
\put(10,30){\line(0,1){15}}
\put(10,45){\line(1,-1){10}}
\put(20,35){\line(2,-1){10}}
\put(30,30){\line(0,-1){15}}
\put(30,15){\line(-2,1){10}} \put(30,15.2){\vector(-2,1){10}}

\put(49.8,20){\line(0,1){15}} \put(50.1,20){\line(0,1){15}}
\put(50,20){\line(-1,1){10}} \put(50,20.2){\vector(-1,1){10}}
\put(40,30){\line(0,1){15}}
\put(40,45){\line(1,-1){10}}
\put(50,35){\line(2,1){10}}
\put(60,40){\line(0,-1){15}}
\put(60,25){\line(-2,-1){10}} \put(50,20.2){\vector(2,1){10}}

\put(79.8,20){\line(0,1){15}} \put(80.2,20){\line(0,1){15}}
\put(80,20){\line(-2,-1){10}} \put(80,20.2){\vector(-2,-1){10}}
\put(70,15){\line(0,1){15}}
\put(70,30){\line(2,1){10}}
\put(80,35){\line(2,-1){10}}
\put(90,30){\line(0,-1){15}}
\put(90,15){\line(-2,1){10}} \put(90,15.2){\vector(-2,1){10}}

\put(109.8,20){\line(0,1){15}} \put(110.2,20){\line(0,1){15}}
\put(110,20){\line(-2,-1){10}} \put(100,15.2){\vector(2,1){10}}
\put(100,15){\line(0,1){15}}
\put(100,30){\line(2,1){10}}
\put(110,35){\line(1,1){10}}
\put(120,45){\line(0,-1){15}}
\put(120,30){\line(-1,-1){10}} \put(110,20.2){\vector(1,1){10}}

\put(70,15){\vector(2,1){10}}
\put(90,15){\vector(-2,1){10}}
\put(80.00,17.00){\makebox(0,0)[cc]{$S$}}

\put(15.00,32.00){\makebox(0,0)[cc]{$Q$}}
\put(25.00,25.00){\makebox(0,0)[cc]{$Q'$}}
\put(50.00,17.00){\makebox(0,0)[cc]{$X\!\!=\!\!S\!\!=\!\!X'$}}
\put(18.00,17.00){\makebox(0,0)[cc]{$X\!\!=\!\!S$}}
\put(113.00,17.00){\makebox(0,0)[cc]{$S\!\!=\!\!X'$}}

\put(52.00,29.00){\makebox(0,0)[cc]{$J$}}
\put(82.00,27.00){\makebox(0,0)[cc]{$J$}}
\put(112.00,29.00){\makebox(0,0)[cc]{$J$}}

\put(70.00,12.00){\makebox(0,0)[cc]{$X$}}
\put(100.00,12.00){\makebox(0,0)[cc]{$X$}}

\put(30.00,12.00){\makebox(0,0)[cc]{$X'$}}
\put(90.00,12.00){\makebox(0,0)[cc]{$X'$}}

\put(74.00,19.00){\makebox(0,0)[cc]{$i$}}
\put(86.00,19.00){\makebox(0,0)[cc]{$k$}}
\put(20.00,5.00){\makebox(0,0)[cc]{I}}
\put(50.00,5.00){\makebox(0,0)[cc]{II}}
\put(80.00,5.00){\makebox(0,0)[cc]{III}}
\put(110.00,5.00){\makebox(0,0)[cc]{IV}}
\end{picture}
 \caption{четыре варианта примыкания кубов к стенке $F$ типа $J$}
 \label{fig:25}
  \end{figure}

Мы должны теперь выразить то, что фасета $F$ невидимая в кубе $Q$. Для этого
будем временно считать, что корень фасеты $F$ расположен в нуле. Пусть
$J=\{j_1<j_2<...<j_{d-1}\}$; тогда линейное уравнение
$\det(v_{j_1},...,v_{j_{d-1}}, \cdot)=0$ задает гиперплоскость, содержащую
фасету $F$. То, что эта фасета невидимая в $Q$, означает, что вектор $v_i$
лежит в положительном полупространстве (если он входит в корень $F$) и в
отрицательном полупространстве (если он выходит из корня $F$). То есть в
случаях I и II $\det(v_{j_1},...,v_{j_{d-1}}, v_i)<0$, а в случаях III и IV
этот определитель больше 0.

Симметрично фасета $F$ видимая в кубе $Q'$, если вектор $v_k$ лежит в
положительном полупространстве в случаях II  и IV, и в отрицательном
полупространстве в случаях I  и III. То есть определитель
$\det(v_{j_1},...,v_{j_{d-1}}, v_k)$ должен иметь знак + или - соответственно.

Это возвращает нас к вопросу о знаках определителей, о котором мы рассказали в
разделе 5. А именно,  точки $j_1<j_2<...<j_{d-1}$ на отрезке $[n]$, делят этот
отрезок на $d$ интервалов, которые удобно нумеровать справа налево. Нулевой
интервал состоит из точек-чисел строго больших $j_{d-1}$, первый интервал --
между $j_{d-2}$ и $j_{d-1}$ и так далее; последний интервал состоит из чисел,
строго меньших $j_1$.  Если число $j$ попадает в интервал с четным номером, то
определитель $\det(v_{j_1},...,v_{j_{d-1}}, v_j)$ положителен; если с нечетным
номером -- отрицателен. Можно говорить, что число $j$ \emph{четно} относительно
$J$, если $j$ попадает в интервал с четным номером, и \emph{нечетно} в
противном случае.

Подведем итог. Пусть задана фасета $F$, представленная как пара $(S,J)$, и куб
$Q=(X,T)$, имеющий $F$ своей фасетой, так что $T=Ji$ для некоторого $i$, не
принадлежащего $J$. \emph{Фасета  $F$ невидима в кубе $Q$ тогда и только тогда,
когда выполнено одно из двух:}

a) $X=S$ \emph{и $i$ нечетно относительно} $J$;

b) $X=S-i$ \emph{и  $i$ четно относительно}  $J$.

Симметрично, пусть куб $Q'=(X',T')$ имеет $F$ своей фасетой, так что $T'=Jk$.
\emph{Фасета $F$ видима в кубе $Q$ тогда и только тогда, когда выполнено одно
из двух:}

a') $X'=S$ \emph{и $k$ четно относительно} $J$;

b') $X'=S-k$\emph{ и  $k$ нечетно относительно}  $J$.

Таким образом, $Q\prec Q'$ тогда и только тогда, когда выполнены комбинации: a)
и a')  (случай II рисунка), или a) и b') (случай I рисунка), или  b) и  a')
(случай IV) или b)  и  b')  (случай III).

Итак, мы выразили комбинаторно отношение $\prec $. Теперь приступим к
доказательству теоремы. Для этого разделим все множество $\mathcal{C}$ кубов на
три группы, или уровня, $\mathcal{C}_0$, $\mathcal{C}_1$ и $\mathcal{C}_2$. К
нулевому уровню $\mathcal{C}_0$  мы относим те кубы $(X,T)$, для которых цвет
$n$ не входит ни в $X$, ни в  $T$; к первому уровню (`перегородка цвета $n$')
$\mathcal{C}_1$ мы относим те кубы, для которых $n$ принадлежит типу $T$;
наконец, второй уровень $\mathcal{C}_2$ состоит из кубов, для которых $n\in X$.
Следующее утверждение является ключевым.\medskip

\textbf{Лемма.} \emph{Уровень монотонно не убывает относительно $\prec$.}\medskip

\begin{proof}
Пусть $Q\prec Q'$. Обозначим через  $l$ и $l'$ уровни $Q$ и $Q'$,
соотвественно. Мы должны проверить две вещи. Первая: если $l'=0$, то и $l=0$.
Вторая: если $l=2$, то и $l'=2$.

Начнем со второго утверждения. Допустим, что $l=2$ (то есть $n$ принадлежит
$X$), но $n$ не принадлежит $X'$. Так как в случае a') $X'$ содержит $X$, этот
случай не реализуем и имеет место случай b'). Но тогда $k=n$, а $n$ всегда
четно, так что мы получаем противоречие.

Проверим теперь первое утверждение. Допустим, что $l'=0$, тогда как $l>0$. Это
значит, что $n$ не принадлежит ни $X'$, ни  $T'$, и в то же время принадлежит
либо $X$, либо   $T$. Принадлежать $X$ цвет $n$  не может, как было только что
показано. Таким образом, $n\in T=Si=(T'-k)\cup \{i\}$. Так как $n$ не
принадлежало $T'$, то $i=n$. Но цвет $n$ всегда четен, так что мы попадаем в
случай a), когда $X=S-i=S-n$. Значит $n$ принадлежит $S$ и тем более $X'=Sk$.
Противоречие.
\end{proof}

Приступим к доказательству теоремы. Оно ведется индукцией по $n$. При $n=d$
утверждение теоремы верно, поскольку имеется всего один куб. Будем считать, что
$d<n$.

Предположим, что имеется циклический монотонный путь $Q_0\prec Q_1 \prec ...
\prec Q_N=Q_0$. В силу предыдущей Леммы этот путь целиком располагается на
одном из уровней, нулевом, первом или втором.

Если цикл проходит на уровне 0, то он лежит в зонотопе $Z(n-1,d)$, что
противоречит индуктивному предположению.

Если цикл лежит на втором уровне, то заменим каждый корень $X_i$ на $X_i-n$. Мы
снова получаем циклический путь в зонотопе $Z(n-1.d)$, что невозможно по
индуктивному предположению.

Остается рассмотреть случай, когда циклический путь располагается на уровне 1.
То есть, если $Q_i=(X_i,T_i)$   для  $i=0,...,N$, то $n\in T_i$ для всех $i$. В
этом случае, заменяя каждое множество $T_i$ на `редуцированное'  множество
$\widetilde{T_i}=T_i-n$, мы получим `путь'  $\widetilde{Q_0}, ... ,
\widetilde{Q_N}$ в $Z(n-1,d-1)$, где $\widetilde{Q_i}=(X_i,\widetilde{T_i})$.
Главное замечание состоит в том, что мы снова получаем цикл, хотя и
противоположно ориентированный, что противоречит индуктивному предположению.
Это замечание следует из приводимой ниже Леммы о реверсе.\medskip

\textbf{Лемма о реверсе.} \emph{Пусть для кубов $Q=(X,T)$ и $Q'=(X',T')$ из
$\mathcal{C}(n,d)$ выполнено соотношение $Q\prec Q'$. Предположим, что $n\in T,
T'$, и положим $\widetilde{T}=T-n$, $\widetilde{T'}=T'-n$. Тогда для кубов
$\widetilde{Q}=(X,\widetilde{T})$ и $\widetilde{Q'}=(X',\widetilde{T'})$ в
зонотопе $\widetilde{Z}=Z(n\!-\!1,d\!-\!1)$ выполнено противоположное
соотношение $\widetilde{Q'} \ \widetilde{\prec} \ \widetilde{Q}$, где
$\widetilde{\prec}$ -- соответствующее отношение на множестве
$\mathcal{C}(n-1,d-1)$.}\medskip

\begin{proof}
Пусть $J=T\cap T'$; очевидно, что $J$ также содержит $n$. Положим
$\widetilde{J}=J-n$. Причина реверса состоит в том, что в силу
$J=\widetilde{J}n$ четность какого-либо цвета $i$ относительно множества
$\widetilde{J}$ противоположна четности этого цвета относительно $J$.
\end{proof}

      \end{document}